\documentclass[english]{article}
\usepackage[margin=2.5cm]{geometry}
\usepackage{hyperref}
\usepackage{amsthm,amsfonts,amssymb,amsmath}
\usepackage{graphicx}

\newtheorem{theorem}{Theorem}[section]

\newtheorem{lemma}[theorem]{Lemma}

\newtheorem{fact}[theorem]{Fact}
\newtheorem{claim}[theorem]{Claim}
\newtheorem{corollary}[theorem]{Corollary}

\theoremstyle{definition}
\newtheorem{definition}[theorem]{Definition}

\newtheorem{notation}[theorem]{Notation}

\numberwithin{equation}{section}
\numberwithin{figure}{section}
\allowdisplaybreaks

\usepackage{babel}
\usepackage{lmodern}
\usepackage[T1]{fontenc}
\usepackage[utf8]{inputenc}

\newcommand{\eps}{\varepsilon}
\renewcommand{\phi}{\varphi}
\newcommand{\su}{\subseteq}
\newcommand{\sm}{\setminus}
\renewcommand{\l}{\ell}
\newcommand{\sgn}{\operatorname{sgn}}
\newcommand{\R}{\mathbb{R}}
\newcommand{\pr}{\operatorname{pr}}
\renewcommand{\equiv}{\sim}

\usepackage{mathrsfs}

\begin{document}
\title{On the speed of algebraically defined graph classes}
\author{Lisa Sauermann\thanks{Department of Mathematics, Stanford University, Stanford, CA 94305. Email: {\tt lsauerma@stanford.edu}.}}

\maketitle

\begin{abstract}\noindent
The speed of a class of graphs counts the number of graphs on the vertex set $\lbrace 1,\dots, n\rbrace$ inside the class as a function of $n$. In this paper, we investigate this function for many classes of graphs that naturally arise in discrete geometry, for example intersection graphs of segments or disks in the plane. While upper bounds follow from Warren's theorem (a variant of a theorem of Milnor and Thom), all the previously known lower bounds were obtained from ad hoc constructions for very specific classes. We prove a general theorem giving an essentially tight lower bound for the number of graphs on $\lbrace 1,\dots, n\rbrace$ whose edges are defined using the signs of a given finite list of polynomials, assuming these polynomials satisfy some reasonable conditions. This in particular implies lower bounds for the speed of many different classes of intersection graphs, which essentially match the known upper bounds. Our general result also gives essentially tight lower bounds for counting containment orders of various families of geometric objects, including circle orders and angle orders. Some of the applications presented in this paper are new, whereas others recover results of Alon-Scheinerman, Fox, McDiarmid-M\"{u}ller and Shi. For the proof of our result we use some tools from algebraic geometry and differential topology.
\end{abstract}

\section{Introduction}

\subsection{Background}

Given a class of graphs, or equivalently a graph property describing the graphs in this class, it is a very natural question to ask about size of the class. More precisely, for each positive integer $n$ one may count the number of graphs on the vertex set $\lbrace 1,\dots, n\rbrace$ satisfying the given property and investigate how this number grows as a function of $n$. This function is often called the speed of the given graph property. There is an extensive body of work classifying the possible behavior of this function for different graph properties (see for example \cite{balogh-et-al-1, balogh-et-al-2, balogh-et-al-3, scheinerman-zito}).

Many natural classes of graphs arising in discrete and computational geometry have been studied intensively both because of their structural properties and due to their relevance in practical applications. In this paper, we prove an essentially tight lower bound on the speed of many graph classes obtained from discrete geometry. In fact, these graphs can be defined algebraically by polynomial conditions. Therefore, following an approach of Alon and Scheinerman \cite{alon-scheinerman}, Warren's theorem \cite{warren} implies an upper bound on the speed of these graph classes (Warren's theorem is a variant of a theorem of Milnor \cite{milnor} and Thom \cite{thom}). We show that this upper bound is essentially tight for any such class of algebraically defined graphs, assuming that the corresponding polynomials satisfy some reasonable conditions.

Intersection graphs are particularly natural classes of graphs obtained from discrete geometry. Given $n$ geometric objects from some family $\mathcal{F}$ (for example the family of all segments in the plane) numbered from $1$ to $n$, their intersection graph is the graph on the vertex set $\lbrace 1,\dots,n\rbrace$ where two vertices are joined by an edge if and only if the corresponding objects intersect. Intersection graphs have been studied intensively \cite{cabello-cardinal-et-al, cabello-jejcic, cardinal-felsner-et-al, ehrlich-et-al, felsner-et-al, fox-pach}, in particular for segments in the plane \cite{chalopin-goncalves, kratochvil-matousek, kratochvil-nesteril, mcdiarmid-mueller2, pach-solymosi, pawlik-et-al} and disks in the plane \cite{mcdiarmid-mueller, mcdiarmid-mueller2}. This is partially due to numerous practical applications of intersection graphs, for example in database mining \cite{berman-et-al}, for modelling broadcast networks (see \cite{clark-et-al} and the refereinces therein), and even in genetics (see \cite{benzer} and \cite[Section 16.1.1]{spinrad}).

As mentioned above, for many families $\mathcal{F}$ of geometric objects, Warren's theorem \cite{warren} can be used to bound the number of graphs occuring as intersection graphs of a collection of $n$ numbered objects in $\mathcal{F}$ (see for example  \cite{pach-solymosi} for segments in the plane and \cite{mcdiarmid-mueller} for disks in the plane, and see \cite[Section 6.2]{matousek} or \cite[Section 4.1]{spinrad} for a general exposition). In contrast, all known lower bounds for the number of intersection graphs of $n$ numbered objects in a given family $\mathcal{F}$ were obtained by (sometimes fairly involved) ad hoc constructions for some specific families $\mathcal{F}$. Specifically, McDiarmid and M\"{u}ller \cite{mcdiarmid-mueller} proved lower bounds for disks and unit disks (in the plane), and Fox \cite{fox} provided a lower bound construction for segments (in the plane). Shi \cite{shi} extended Fox' construction to the graphs of various non-linear functions, including parabolas and higher-degree polynomials. All of these lower bounds essentially match the upper bounds that Warren's theorem \cite{warren} gives in respective cases. However, these lower bound constructions are specific to the particular family $\mathcal F$ and do not easily generalize to other families $\mathcal F$ of geometric objects.

In this paper, we prove a general theorem giving an essentially tight lower bound for the number of graphs whose edges are defined using the signs of a given finite list of polynomials, assuming these polynomials satisfy some reasonable conditions. Our theorem in particular implies essentially tight lower bounds for the number of intersection graphs of segments, disks and many other geometric objects in the plane (or in higher dimension). It also implies an essentially tight lower bound for the number of graphs obtained by considering the pairwise linking or non-linking relations of $n$ numbered disjoint circles in $\R^3$.

From discrete geometry, one can not only obtain graphs of interest, but also partial orders, so-called containment orders. A collection of $n$ geometric objects from some family $\mathcal{F}$ numbered from $1$ to $n$ defines a partial order on the set $\lbrace 1,\dots,n\rbrace$ obtained from the containment relations between the objects: In this partial order we have $x\prec y$ for distinct $x,y\in \lbrace 1,\dots,n\rbrace$ if and only if the object with number $x$ is a subset of the object with number $y$. Well-studied examples of such partial orders include circle orders (obtained form $n$ disks in the plane) and angle orders (obtained from $n$ ``filled'' angles in the plane, each of which is an intersection of two closed half-planes), see \cite{alon-scheinerman, fishburn, fishburn-trotter, sidney-et-al}. Using Warren's theorem \cite{warren}, Alon and Scheinerman \cite{alon-scheinerman} gave an upper bound for the number of containment orders obtained from a collection of $n$ numbered objects in $\mathcal{F}$ in terms of the degrees of freedom of the family $\mathcal{F}$ (which they defined in \cite{alon-scheinerman}).

For many geometric families $\mathcal{F}$, our general result implies an essentially matching lower bound for this number of containment orders. In particular, this essentially determines the number of circle orders, angle orders and containment orders obtained polygons with a fixed number of vertices in the plane.

In order for our result to apply straightforwardly not only to algebraically defined graphs, but also to partial orders, we work in the framework of algebraically defined edge-labelings of complete graphs. To be more specific, given a finite set $\Lambda$ of labels, a list of polynomials $P_1,\dots,P_k\in \R[x_1,\dots,x_d,y_1,\dots,y_d]$, a function $\phi: \lbrace +,-,0\rbrace^k \to \Lambda$ and points $a_1,\dots,a_n\in \R^d$, one can define an edge-labeling of the complete graph on the vertex set $\lbrace 1,\dots,n\rbrace$ as follows: For any $1\leq i<j\leq n$ the label of the edge $ij$ is defined to be the value of $\phi$ applied to the signs of the polynomial expressions $P_1(a_i,a_j),\dots,P_k(a_i,a_j)$. Fixing $\Lambda$, the polynomials $P_1,\dots,P_k$ and $\phi$, we are then concerned with the number of edge-labelings which can be obtained in this way for some points $a_1,\dots,a_n\in \R^d$

Taking the set of labels to be $\Lambda=\lbrace\text{``edge''}, \text{``non-edge''}\rbrace$, edge-labelings of the complete graph on the vertex set $\lbrace 1,\dots,n\rbrace$ correspond precisely to ordinary graphs  on the vertex set $\lbrace 1,\dots,n\rbrace$. However, the setting of edge-labelings also allows us to encode partial orders in a natural way.

Our main result, Theorem \ref{theo-main} below, gives a lower bound for the number of algebraically defined edge-labelings of complete graphs for a fixed finite set $\Lambda$ of labels, fixed polynomials $P_1,\dots,P_k\in \R[x_1,\dots,x_d,y_1,\dots,y_d]$ and a fixed function $\phi: \lbrace +,-,0\rbrace^k \to \Lambda$ satisfying some reasonable conditions. This bound is essentially tight (see Theorem \ref{theorem-upper-bound}). All the above-mentioned applications of our general result will be discussed in Section \ref{sect-applications}.

Before we present the slightly technical statement of our result in the next subsection, let us give a brief motivating example for our set-up. This example will show why intersection graphs of open disks in the plane can be interpreted as algebraically defined edge-labelings of complete graphs as described above: Each disk in the plane is given by specifying its center $(x,y)$ and its radius $r>0$. Thus, the family of open disks in the plane corresponds to the open set $U$ of points $(x,y,r)\in \R^3$ with $r>0$. Two disks corresponding to the points $(x,y,r), (x',y',r')\in U$ intersect if and only if $(x-x')^2+(y-y')^2<(r+r')^2$. Thus, a graph on the vertex set $\lbrace 1,\dots,n\rbrace$ is an intersection graph of $n$ numbered open disks in the plane if and only if there are points $(x_1,y_1,r_1), \dots, (x_n,y_n,r_n)\in U$ such for all $1\leq i<j\leq n$ we have $(x_i-x_j)^2+(y_i-y_j)^2-(r_i+r_j)^2<0$ if and only if $ij$ is an edge of the graph. Taking the set of labels $\Lambda=\lbrace\text{``edge''}, \text{``non-edge''}\rbrace$, intersection graphs of $n$ open disks in the plane then correspond to algebraically defined edge-labelings of the complete graph on the vertex $\lbrace 1,\dots,n\rbrace$ with labels in $\Lambda$.

\subsection{Statement of the result}

An edge-labeling of a graph $G$ with labels in some set $\Lambda$ is a function $F: E(G) \to \Lambda$. For every edge $e\in E(G)$, we call $F(e)\in \Lambda$ the label of the edge $e$.

For a real number $x$, define $\sgn(x)\in \lbrace +,-,0\rbrace$ by taking the sign of $x$ if $x\neq 0$ and setting $\sgn(x)=0$ if $x=0$.

\begin{definition}\label{defi-representable}Let us fix a finite set $\Lambda$, an integer $d\geq 1$, polynomials $P_1,\dots,P_k\in \R[x_1,\dots,x_d,y_1,\dots,y_d]$, a function $\phi: \lbrace +,-,0\rbrace^k \to\Lambda$, and a non-empty open subset $U\su \R^d$. Then, for any points $a_1,\dots,a_n\in U\su \R^d$, we define $F_{P_1,\dots,P_k,\phi}(a_1,\dots,a_n)$ to be the following edge-labeling of the complete graph on the vertex set $\lbrace 1,\dots, n\rbrace$ with labels in $\Lambda$: For $1\leq i<j\leq n$, define the label of the edge $ij$ to be the value $\phi\big(\sgn P_1(a_i,a_j),\dots ,\sgn P_k(a_i,a_j)\big)\in \Lambda$. 

Furthermore, let us say that an edge-labeling of the complete graph on the vertex set $\lbrace 1,\dots, n\rbrace$ with labels in $\Lambda$ is \emph{$(P_1,\dots,P_k,\phi,U, \Lambda)$-representable} if it occurs as  $F_{P_1,\dots,P_k,\phi}(a_1,\dots,a_n)$ for some $a_1,\dots,a_n\in U$.
\end{definition} 

In our motivating example at the end of the previous subsection, the intersection graphs of numbered open disks in the plane correspond to  $(P,\phi,U,\Lambda)$-representable edge-labelings, where we take open set $U=\lbrace (x,y,r)\in \R^3\mid r>0\rbrace$, the polynomial $P(x,y,r,x',y',r')=(x-x')^2+(y-y')^2-(r+r')^2$, the set $\Lambda=\lbrace\text{``edge''}, \text{``non-edge''}\rbrace$ and the function $\phi$ given by $\phi(-)=\text{``edge''}$ and $\phi(+)=\phi(0)= \text{``non-edge''}$.

The following theorem gives an upper bound for the number of $(P_1,\dots,P_k,\phi,U,\Lambda)$-representable edge-labelings for any $P_1,\dots,P_k$, $\phi$, $U$ and $\Lambda$. It follows from a theorem of Warren \cite{warren} with exactly the same method as in \cite{alon-scheinerman, mcdiarmid-mueller, pach-solymosi}. For the reader's convenience a proof will be given in Subsection \ref{sect-a-upper-bound} of the appendix.

\begin{theorem}\label{theorem-upper-bound} Let us fix a finite set $\Lambda$, an integer $d\geq 1$, polynomials $P_1,\dots,P_k\in \R[x_1,\dots,x_d,y_1,\dots,y_d]$, a function $\phi: \lbrace +,-,0\rbrace^k \to \Lambda$, and a non-empty open subset $U\su \R^d$. Then the number of $(P_1,\dots,P_k,\phi,U,\Lambda)$-representable edge-labelings of the complete graph on the vertex set $\lbrace 1,\dots, n\rbrace$ is at most $n^{(1+o(1))dn}$.
\end{theorem}

Our main result, Theorem \ref{theo-main} below, states that under some reasonable assumptions, the upper bound in Theorem \ref{theorem-upper-bound} is tight. In fact, we prove something slightly stronger: In some applications one would only like to consider $(P_1,\dots,P_k,\phi,U,\Lambda)$-representable edge-labelings for which one can choose the points $a_1,\dots,a_n\in U$ in Definition \ref{defi-representable} in such a way that $P_s(a_i,a_j)\neq 0$ for all $1\leq i<j\leq n$ and all $1\leq s\leq k$. This motivates the following strengthening of the notion of being $(P_1,\dots,P_k,\phi,U,\Lambda)$-representable.

\begin{definition}\label{defi-strongly-representable}
Let us fix a finite set $\Lambda$, an integer $d\geq 1$, polynomials $P_1,\dots,P_k\in \R[x_1,\dots,x_d,y_1,\dots,y_d]$, a function $\phi: \lbrace +,-,0\rbrace^k \to \Lambda$, and a non-empty open subset $U\su \R^d$. Then, let us say that an edge-labeling of the complete graph on the vertex set $\lbrace 1,\dots, n\rbrace$ is \emph{strongly $(P_1,\dots,P_k,\phi,U, \Lambda)$-representable} if it occurs as  $F_{P_1,\dots,P_k,\phi}(a_1,\dots,a_n)$ for some $a_1,\dots,a_n\in U$ such that $P_s(a_i,a_j)\neq 0$ for all $1\leq i<j\leq n$ and all $1\leq s\leq k$.
\end{definition} 

For our lower bound complementing the upper bound in Theorem \ref{theorem-upper-bound}, we need an assumption that the open set $U$ is reasonable shaped. This will be made precise by the following definition.

\begin{definition}\label{defi-definable-polynomials}
Let us call a subset $U\su \R^d$ \emph{definable by polynomials} if there exists a finite list of real polynomials $Q_1,\dots,Q_\l$ and a subset $S\su \lbrace +,-,0\rbrace^\l$ such that
\[U=\lbrace x\in \R^d \mid (\sgn Q_1(x),\dots,\sgn Q_\l(x))\in S\rbrace.\]
\end{definition}

Our main result is the following theorem.

\begin{theorem}\label{theo-main} Let us fix a finite set $\Lambda$, an integer $d\geq 1$, polynomials $P_1,\dots,P_k\in \R[x_1,\dots,x_d,y_1,\dots,y_d]$, a function $\phi: \lbrace +,-,0\rbrace^k \to \Lambda$, and a non-empty open subset $U\su \R^d$ which is definable by polynomials. Suppose that for any two distinct points $a,a'\in U$ there exists a point $b\in U$ with $P_s(a, b)\neq 0$ and $P_s(a', b)\neq 0$ for all $1\leq s\leq k$ and such that
\[\phi\big(\sgn P_1(a,b),\dots ,\sgn P_k(a,b)\big)\neq \phi\big(\sgn P_1(a',b),\dots ,\sgn P_k(a',b)\big).\]
Then there are at least $n^{(1-o(1))dn}$ strongly $(P_1,\dots,P_k,\phi,U, \Lambda)$-representable edge-labelings of the complete graph on the vertex set $\lbrace 1,\dots, n\rbrace$.
\end{theorem}

Note that a strongly $(P_1,\dots,P_k,\phi,U, \Lambda)$-representable edge-labeling is in particular $(P_1,\dots,P_k,\phi,U, \Lambda)$-representable. Thus, Theorem \ref{theo-main} shows that the upper bound in Theorem \ref{theorem-upper-bound} is sharp whenever the assumptions of Theorem \ref{theo-main} are satisfied.

Let us comment on the assumption in Theorem \ref{theo-main} concerning the existence of the desired point $b\in U$ for any two distinct points $a,a'\in U$. Roughly speaking, this assumption is saying that for any two distinct points $a,a'\in U$ there exists a point $b\in U$ such that for $i<j$ both of the pairs $(a_i,a_j)=(a,b)$ and $(a_i,a_j)=(a',b)$ are allowed in Definition \ref{defi-strongly-representable} and they lead to different outcomes for the label of the edge $ij$. An assumption of this form is necessary in Theorem \ref{theo-main}, since otherwise one could artificially increase $d$ by considering additional variables that do not occur in any of the polynomials $P_1,\dots,P_k$ (this means, one could interpret $P_1,\dots,P_k$ as polynomials in $\R[x_1,\dots,x_{d+1},y_1,\dots,y_{d+1}]$ and replace $U$ by $U\times \R$).

However, this assumption in Theorem \ref{theo-main} is usually very easy to check in applications. For example, when studying the number intersection graphs of geometric objects in some family $\mathcal{F}$, the assumption is, roughly speaking, that the family $\mathcal{F}$ does not contain two ``copies'' of the same object (in other words, for any two distinct objects in $\mathcal{F}$ there exists an object in $\mathcal{F}$ intersecting exactly one of them). Similarly, the assumption that the set $U$ is definable by polynomials is usually immediate from the choice of $U$ in a given application.

We remark that the $o(1)$-terms in Theorems \ref{theorem-upper-bound} and \ref{theo-main} tend to zero for $n\to \infty$. The terms may depend on $d$, and on the polynomials $P_1,\dots,P_k$.

The rest of this paper is organized as follows. In Section \ref{sect-applications}, we will discuss applications of Theorem \ref{theo-main} to counting intersection graphs, linking graphs of circles in $\R^3$, containment orders and partial orders of a given dimension. The remaining sections are devoted to the proof of Theorem \ref{theo-main}. More specifically, Section \ref{sect-algebraic} contains some algebraic preliminaries for the proof. Theorem \ref{theo-main} will then be proved in Section \ref{sect-proof-theo-main}, apart from the proofs of several lemmas which will be postponed to Sections \ref{sect-lemma-b-i-j} to \ref{sect-proof-lemma-main}. Finally, the appendix contains the proofs of Theorem \ref{theorem-upper-bound} and of the algebraic statements in Section \ref{sect-algebraic}.

The proofs of the algebraic statements in Section 3 use some relatively basic tools from algebraic geometry and differential topology. However, all these proofs are in the appendix. The main part of the paper does not require any previous knowledge about algebraic geometry or differential topology. We do, however, use some multi-variable analysis, including the local integrability of vector fields on $\R^d$.

\section{Applications}\label{sect-applications}

In \cite{alon-pach-et-al} and \cite{alon-scheinerman}, the authors establish that a number of geometric relations (e.g.\ segments intersecting each other or disks being contained in each other) can be encoded by polynomial conditions. Using these encodings, our Theorem \ref{theo-main} can be applied to most of the geometric relations studied in \cite{alon-pach-et-al} and \cite{alon-scheinerman}. In particular, we obtain matching lower bounds to the upper bounds in \cite{alon-scheinerman} on the number of circle orders, angle orders and $m$-vertex-polygon orders.

In this section, we will comment on these applications. For most of these applications, It is already demonstrated in \cite{alon-pach-et-al} and \cite{alon-scheinerman} that the desired geometric relations can be expressed by polynomial conditions. However, we need to check the assumptions in Theorem \ref{theo-main} requiring that the set $U$ is open and definable by polynomials and that for any distinct points $a,a'\in U$ there exists a point $b\in U$ with the desired properties.

\subsection{Intersection graphs of open disks in the plane}

As already mentioned in the introduction, open disks in the plane can easily be encoded as points in $U=\lbrace (x,y,r)\in \R^3\mid r>0\rbrace$, where $(x,y)\in \R^2$ is the center of the disk and $r>0$ its radius. Clearly, the set $U$ is open and definable by polynomials.

Two open disks corresponding to the points $(x,y,r), (x',y',r')\in U$ intersect if and only if $P(x,y,r,x',y',r')<0$, where $P(x,y,r,x',y',r')=(x-x')^2+(y-y')^2-(r+r')^2$. Taking $\Lambda=\lbrace\text{``edge''}, \text{``non-edge''}\rbrace$ and defining $\phi:\lbrace +,-,0\rbrace\to\Lambda$ to be the function given by $\phi(-)=1$ and $\phi(+)=\phi(0)=-1$, the intersection graphs of $n$ open disks in the plane numbered from $1$ to $n$ correspond to the $(P,\phi,U, \Lambda)$-representable edge-labelings of the complete graph on the vertex $\lbrace 1,\dots,n\rbrace$.

Now, we need to check that for any distinct $(x,y,r), (x',y',r')\in U$ there is a point $(x^*,y^*,r^*)\in U$ such that $P(x,y,r,x^*,y^*,r^*)\neq 0$ and $P(x',y',r',x^*,y^*,r^*)\neq 0$ and such that $\phi(\sgn P(x,y,r,x^*,y^*,r^*))\neq \phi(\sgn P(x',y',r',x^*,y^*,r^*))$. But this simply means that for any distinct open disks $D, D'$ there exists a disk $D^*$ which intersects exactly one of the disks $D$ and $D'$ and such that the boundary circle of $D^*$ is neither tangent (from the outside) to the boundary circle of $D$ nor to the boundary circle of $D'$. This geometric statement is very easy to check.

Thus, Theorems \ref{theorem-upper-bound} and \ref{theo-main} yield the following corollary, which reproves a result of McDiarmid and M\"{u}ller \cite{mcdiarmid-mueller}.

\begin{corollary}[\cite{mcdiarmid-mueller}]\label{coro-inter-disks}
The number of graphs on the vertex set $\lbrace 1,\dots,n\rbrace$ that are intersection graphs of $n$ numbered open disks in the plane equals $n^{(3+o(1))n}$.
\end{corollary}

McDiarmid and M\"{u}ller \cite{mcdiarmid-mueller} actually proved a stronger lower bound of the form $C^n n^{3n}$ and also a stronger upper bound of the form $C'^n n^{3n}$ for some absolute constants $C$ and $C'$.

Very similarly, when considering the number of intersection graphs of $n$ open \emph{unit} disks in the plane, Theorems \ref{theorem-upper-bound} and \ref{theo-main}, reprove another result of McDiarmid and M\"{u}ller \cite{mcdiarmid-mueller}.

\begin{corollary}[\cite{mcdiarmid-mueller}]\label{coro-inter-unit-disks}
The number of  graphs on the vertex set $\lbrace 1,\dots,n\rbrace$ that are intersection graphs of $n$ numbered open unit disks in the plane equals $n^{(2+o(1))n}$.
\end{corollary}

Again, McDiarmid and M\"{u}ller \cite{mcdiarmid-mueller} actually proved a stronger lower bound of the form $C^n n^{2n}$ and also a stronger upper bound of the form $C'^n n^{2n}$ for some absolute constants $C$ and $C'$.

Although our lower bounds obtained from Theorem \ref{theo-main} are weaker than the lower bounds of McDiarmid and M\"{u}ller \cite{mcdiarmid-mueller}, we still included Corollaries \ref{coro-inter-disks} and \ref{coro-inter-unit-disks} as simple and illustrative sample applications of Theorem \ref{theo-main}. In contrast to this general theorem, the lower bound constructions of McDiarmid and M\"{u}ller are very specific to the particular problems for open disks and open unit disks in the plane, respectively.

We remark that our argument above straightforwardly generalizes to higher dimensions. Thus, applying Theorems \ref{theorem-upper-bound} and \ref{theo-main} we obtain the following new results.

\begin{corollary}
For any $m\geq 1$, the number of graphs on the vertex set $\lbrace 1,\dots,n\rbrace$ that are intersection graphs of $n$ numbered open balls in $\R^m$ equals $n^{(m+1+o(1))n}$.
\end{corollary}

\begin{corollary}
For any $m\geq 1$, the number of graphs on the vertex set $\lbrace 1,\dots,n\rbrace$ that are intersection graphs of $n$ numbered open unit balls in $\R^m$ equals $n^{(m+o(1))n}$.
\end{corollary}

\subsection{Intersection graphs of segments in the plane}

Pach and Solymosi \cite{pach-solymosi} proved that at most $n^{(4+o(1))n}$ graphs on the vertex set $\lbrace 1,\dots,n\rbrace$ are intersection graphs of $n$ numbered segments in the plane. Using a construction specific to segments, Fox  \cite{fox} proved that this bound is tight. We can also obtain the tightness of this bound as a corollary of Theorem \ref{theo-main}.

\begin{corollary}[\cite{fox, pach-solymosi}]\label{coro-inter-segments}
The number of graphs on the vertex set $\lbrace 1,\dots,n\rbrace$ that are intersection graphs of $n$ numbered segments in the plane equals $n^{(4+o(1))n}$.
\end{corollary}

In order to deduce the lower bound in Corollary \ref{coro-inter-segments} from Theorem \ref{theo-main}, we need to encode segments in the plane by points in $\R^4$ such that it can be determined by polynomial conditions whether two segments intersect. In order to do so, we follow the approach in \cite{pach-solymosi}.

Whenever a graph on $\lbrace 1,\dots,n\rbrace$ is an intersection graph of $n$ numbered segments in the plane, these segments can be chosen such that none of them is vertical (note that otherwise we can rotate the entire arrangement of the segments).

Each non-vertical segment in the plane can be described by a quadruple $(\alpha, \beta, \gamma, \delta)\in \R^4$ with $\gamma<\delta$, in such a way that the segment is given by $ \lbrace (x,y)\in \R^2\mid y=\alpha x+\beta, \gamma\leq x\leq \delta\rbrace$. So let $U=\lbrace (\alpha, \beta, \gamma, \delta)\in \R^4\mid \gamma<\delta\rbrace$. Then, the set $U$ is open and definable by polynomials.

The segments $ \lbrace (x,y)\in \R^2\mid y=\alpha x+\beta, \gamma\leq x\leq \delta\rbrace$ and $ \lbrace (x,y)\in \R^2\mid y=\alpha' x+\beta', \gamma'\leq x\leq \delta'\rbrace$ intersect each other if and only if either
\[\alpha> \alpha'\text{ and }\max(\gamma, \gamma')\cdot (\alpha-\alpha')\leq \beta'-\beta\leq \min(\delta, \delta')\cdot (\alpha-\alpha')\]
or
\[\alpha< \alpha'\text{ and }\max(\gamma, \gamma')\cdot (\alpha'-\alpha)\leq \beta-\beta'\leq \min(\delta, \delta')\cdot (\alpha'-\alpha)\]
or
\[\alpha= \alpha'\text{ and }\beta=\beta'\text{ and }\max(\gamma, \gamma')\leq \min(\delta, \delta').\]
All of these relations can be checked using the signs of finitely many polynomials $P_1,\dots,P_k$ in the variables $\alpha, \beta, \gamma, \delta, \alpha', \beta', \gamma', \delta'$ and a suitable function $\phi: \lbrace +,-,0\rbrace^k \to \Lambda$, where $\Lambda=\lbrace\text{``edge''}, \text{``non-edge''}\rbrace$. This way, the intersection graphs of $n$ numbered segments are precisely the $(P_1,\dots,P_k,\phi,U, \Lambda)$-representable edge-labelings of the complete graph on the vertex $\lbrace 1,\dots,n\rbrace$.

Finally, in order to apply Theorem \ref{theo-main}, we need to check the condition that for every two distinct $a, a'\in U$ there exists  $b\in U$ such that $P_s(a, b)\neq 0$ and $P_s(a', b)\neq 0$ for all $1\leq s\leq k$ and such that
\[\phi\big(\sgn P_1(a,b),\dots ,\sgn P_k(a,b)\big)\neq \phi\big(\sgn P_1(a',b),\dots ,\sgn P_k(a',b)\big).\]
But this condition simply means that for any two distinct non-vertical segments $S$ and $S'$, there exists a non-vertical segment $T$ intersecting exactly one of them (and such that $T$ is sufficiently generic in the following sense: $T$ is not parallel to $S$ or $S'$, $T$ does not contain any of the end-points of $S$ or $S'$, the end-points of $T$ do not lie on either $S$ or $S'$, no end-point of $T$ has the same $x$-coordinate as any endpoint of $S$ or $S'$, and the line through $T$ intersects the $y$-axis at a different point than the lines through $S$ and $S'$). This is again very easy to check.

Thus, Theorem \ref{theo-main} indeed recovers the lower bound in Corollary \ref{coro-inter-segments} due to Fox \cite{fox}.

Fox' construction for segments \cite{fox} was extended by Shi \cite{shi} to graphs of the restrictions of various non-linear functions to some closed interval (note that non-vertical segments are graphs of restrictions of linear functions to closed intervals). The functions in Shi's work include parabolas, higher degree polynomials and rational functions. In each of these cases, Shi provides specific constructions establishing the lower bound for the corresponding number of intersection graphs. Shi  \cite[Section 2]{shi} also deduced corresponding upper bounds from Warren's theorem \cite{warren} by finding encodings into $\R^d$ and polynomials that detect the intersection relations between the objects. We omit the details here, but in each of these cases it can be checked easily that the assumptions of Theorem \ref{theo-main} are satisfied. Thus, Theorem \ref{theo-main} implies the  Shi's various lower bounds in a uniform and non-constructive way.

Theorems \ref{theorem-upper-bound} and \ref{theo-main} can also be used to count the number of intersection graphs of many other kinds of geometric objects. For example, one obtains the following corollaries.

\begin{corollary}
For any $m\geq 3$, the number of graphs on the vertex set $\lbrace 1,\dots,n\rbrace$ that are intersection graphs of $n$ numbered rays in the plane equals $n^{(3+o(1))n}$.
\end{corollary}

\begin{corollary}
For any $m\geq 3$, the number of graphs on the vertex set $\lbrace 1,\dots,n\rbrace$ that are intersection graphs of $n$ numbered $m$-gons (polygons with $m$ vertices each) in the plane equals $n^{(2m+o(1))n}$.
\end{corollary}

\begin{corollary}
For any $m\geq 1$, the number of graphs on the vertex set $\lbrace 1,\dots,n\rbrace$ that are intersection graphs of $n$ numbered axis-parallel boxes in $\R^m$ equals $n^{(2m+o(1))n}$.
\end{corollary}

\subsection{Linking graphs of circles in $\R^3$}\label{subsect-linking}

We say that two disjoint circles in $\R^3$ form a link if they are linked in a topological sense (meaning that one circle describes a non-trivial element of the fundamental group of the complement of the other circle). Given $n$ disjoint circles in $\R^3$ numbered from $1$ to $n$, one can define their linking graph as the graph on the vertex set $\lbrace 1,\dots,n\rbrace$ where two vertices are joined by an edge if and only if the corresponding circles form a link.

From Theorems \ref{theorem-upper-bound} and \ref{theo-main} one obtains the following corollary concerning the number of linking graphs of $n$ circles in $\R^3$.

\begin{corollary}\label{coro-linking-circles}
The number of graphs on the vertex set $\lbrace 1,\dots,n\rbrace$ that are linking graphs of $n$ numbered circles in $\R^3$ equals $n^{(6+o(1))n}$.
\end{corollary}

Note that each circle $C$ in $\R^3$ lies in a unique plane, which we will from now on call \emph{the plane of} $C$. Whenever a graph on $\lbrace 1,\dots,n\rbrace$ is a linking graph of $n$ numbered circles in $\R^3$, these circles can be chosen in such a way that none of the $n$ planes of the $n$ circles is parallel to the $z$-axis (otherwise we can rotate the configuration of the circles to achieve this).

Each circle $C$ in $\R^3$ whose plane is not parallel to the $z$-axis can be described by a 6-tuple $(a,b,c,d,e,r)\in \R^6$ with $r>0$. Here $(a,b,c)$ is the center of $C$. Furthermore, $d,e\in \R$ are such that the vector $(d,e,1)$ is orthogonal   to the plane of $C$ (recall that this plane is not parallel to the $z$-axis). Finally, $r>0$ is the radius of the circle.

So let us define $U=\lbrace (a,b,c,d,e,r)\in \R^6\mid r>0\rbrace$, then each point in $U$ corresponds to a circle in $\R^3$ whose plane is not parallel to the $z$-axis. Note that the set $U$ is open and definable by polynomials.

We now need to show that for two points $(a,b,c,d,e,r), (a',b',c',d',e',r')\in U$ it can be checked using the signs of a finite list of polynomials $P_1,\dots,P_k$ in $a,b,c,d,e,r, a',b',c',d',e',r'$ whether the circles $C$ and $C'$ corresponding to $(a,b,c,d,e,r)$ and $(a',b',c',d',e',r')$ form a link. The key observation in order to show this is the following: $C$ and $C'$ form a link if and only if there exists a point of $C$ which lies on the plane of $C'$ inside the circle $C'$ and another point of $C$ which lies on the plane of $C'$ outside the circle $C'$. Furthermore, by symmetry, the same holds with the roles of $C$ and $C'$ interchanged.

Our strategy for checking whether $C$ and $C'$ form a link via polynomial conditions is the following. If the planes of $C$ and $C'$ are parallel (which happens if and only if $(d,e)=(d',e')$), then the circles $C$ and $C'$ cannot form a link. Otherwise let $\l$ be the line of intersection of the planes of $C$ and $C'$. If the line $\l$ does not intersect the circle $C$, then $C$ does not have any points on the plane of $C'$, so $C$ and $C'$ cannot be linked. Similarly, if $\l$ is tangent to the circle $C$, then $C$ has only one point on the plane of $C'$, so $C$ and $C'$ cannot be linked. Hence we may assume that $\l$ intersects $C$ in two distinct points $X_1$ and $X_2$ (recall that both $\l$ and $C$ lie in the plane of $C$). Now, $X_1$ and $X_2$ are the unique points of $C$ on the plane of $C'$. Thus, $C$ and $C'$ form a link if and only if one of the two points $X_1$ and $X_2$ has distance less than $r'$ from the center $(a',b',c')$ of $C'$ and the other point has distance more than $r'$.

However, when implementing this strategy, one has to be careful, since the coordinates of $X_1$ and $X_2$ are not polynomials in $a,b,c,d,e,r, a',b',c',d',e',r'$ (in fact, the expressions for the coordinates of $X_1$ and $X_2$ contain square roots). Nevertheless, the strategy can be implemented, and the slightly tedious details can be found in Subsection \ref{subsect-app-linking} of the appendix.

We remark that the list of polynomials $P_1,\dots,P_k$ in $a,b,c,d,e,r, a',b',c',d',e',r'$ that we use to check whether $C$ and $C'$ form a link consists of $k=4$ polynomials. The first two are the polynomials $e-e'$ and $f-f'$. The third polynomial is non-zero whenever the line $\l$ (the intersection of the planes of $C$ and $C'$) is not tangent to $C$. And the fourth polynomial is non-zero whenever the circles $C$ and $C'$ are disjoint.

Now, taking these polynomials $P_1,\dots,P_k$ in $a,b,c,d,e,r, a',b',c',d',e',r'$ as well as a suitable function $\phi: \lbrace +,-,0\rbrace^k \to \Lambda$, where $\Lambda=\lbrace\text{``edge''}, \text{``non-edge''}\rbrace$, the linking graphs of $n$ numbered circles in $\R^3$ are precisely the $(P_1,\dots,P_k,\phi,U, \Lambda)$-representable edge-labelings of the complete graph on the vertex $\lbrace 1,\dots,n\rbrace$.

Thus, the upper bound in Corollary \ref{coro-linking-circles} follows from Theorem \ref{theorem-upper-bound}.

In order to deduce the lower bound in Corollary \ref{coro-linking-circles} from Theorem \ref{theo-main}, we need to check the condition that for every two distinct $a, a'\in U$ there exists  $b\in U$ such that $P_s(a, b)\neq 0$ and $P_s(a', b)\neq 0$ for all $1\leq s\leq k$ and such that
\[\phi\big(\sgn P_1(a,b),\dots ,\sgn P_k(a,b)\big)\neq \phi\big(\sgn P_1(a',b),\dots ,\sgn P_k(a',b)\big).\]
But this condition simply means that for any two distinct circles $C$ and $C'$ in $\R^3$ (whose planes are not parallel to the $z$-axis), there exists a circle $D$ forming a link with exactly one of them (and such that $D$ is sufficiently generic in the following sense: The plane of $D$ is not parallel to the $z$-axis and if $(d_D,e_D,1)$ is an orthogonal vector to this plane, and $(d_C,e_C,1)$ and $(d_{C'},e_{C'},1)$ are defined analogously for $C$ and $C'$, then $d_D\not\in \lbrace d_{C}, d_{C'}\rbrace$ and $e_D\not\in \lbrace e_{C}, e_{C'}\rbrace$. Furthermore, the intersection line of the planes of $D$ and $C$ is not tangent to the circle $C$, the intersection line of the planes of $D$ and $C'$ is not tangent to the circle $C'$, and $D$ is disjoint from $C$ and $C'$). Such a circle $D$ can be found by first taking any circle $D$ which forms a link with exactly one of the circles $C$ and $C'$ and is disjoint from both of them  (for example, one can take a very small circle looping closely around $C$), and then perturbing $D$ slightly in order to satisfy the conditions on being sufficiently generic.

Thus, Corollary \ref{coro-linking-circles} indeed follows from Theorems \ref{theorem-upper-bound} and \ref{theo-main}.

With essentially the same reasoning one obtains an analogous result for the number of linking graphs of $n$ \emph{unit} circles in $\R^3$.

\begin{corollary}\label{coro-linking-unit-circles}
The number of graphs on the vertex set $\lbrace 1,\dots,n\rbrace$ that are linking graphs of $n$ numbered unit circles in $\R^3$ equals $n^{(5+o(1))n}$.
\end{corollary}

Note that in the setting of unit circles in Corollary \ref{coro-linking-unit-circles}, we consider the set $U=\R^5$, since every unit circle in $\R^3$ whose plane is not parallel to the $z$-axis can be described by a 5-tuple $(a,b,c,d,e)\in \R^5$. We can check whether two unit circles are linked using polynomial conditions in the same way as before (we just replace the radii $r$ and $r'$ by $1$). The only place where we have to be slightly more careful is when checking the assumption of Theorem \ref{theo-main}: Given any two distinct unit circles $C$ and $C'$ in $\R^3$, we need a unit circle $D$ in $\R^3$ which forms a link with exactly one of the circles $C$ and $C'$ and is disjoint from both of them. But it is not hard to see that such a unit circle $D$ indeed exists.

\subsection{Circle orders and some other containment orders}

Given $n$ distinct closed disks in the plane numbered from $1$ to $n$, one obtains a partial order on the set $\lbrace 1,\dots,n\rbrace$ by defining $x\prec y$ for distinct $x,y\in \lbrace 1,\dots,n\rbrace$ if and only if the disk with number $x$ is contained in the disk with number $y$. The partial orders on $\lbrace 1,\dots,n\rbrace$ obtained in this way (for some choice of $n$ closed disks) are called \emph{circle orders} (see \cite{sidney-et-al}).

Alon and Scheinerman \cite{alon-scheinerman} proved that the number of circle orders on $\lbrace 1,\dots,n\rbrace$ is at most $n^{(3+o(1))n}$. Theorem \ref{theo-main} implies that this bound is tight.

\begin{corollary}\label{coro-order-circles}
The number of circle orders on $\lbrace 1,\dots,n\rbrace$ equals $n^{(3+o(1))n}$.
\end{corollary}

In order to deduce the lower bound in Corollary \ref{coro-order-circles} from  Theorem \ref{theo-main}, we represent each circle order as a $(P_1, P_2,\phi,U, \Lambda)$-representable edge-labeling for suitably chosen $P_1$, $P_2$, $\Lambda$, $\phi$ and $U\su \R^3$. Each closed disk in the plane corresponds to a point in $U=\lbrace (x,y,r)\in \R^3\mid r>0\rbrace$ by taking $(x,y)\in \R^2$ to be the center of the disk and $r>0$ its radius. Clearly, the set $U$ is open and definable by polynomials.

For any two distinct closed disks $D$ and $D'$ we have $D\subset D'$ if and only if the points $(x,y,r), (x',y',r')\in U$ corresponding to $D$ and $D'$, respectively, satisfy $(x-x')^2+(y-y')^2\leq (r-r')^2$ and $r-r'< 0$. So define $P_1(x,y,r,x',y',r')=(r-r')^2-(x-x')^2-(y-y')^2$ and $P_2(x,y,r,x',y',r')=r-r'$. Furthermore consider the set $\Lambda=\lbrace \text{``$\prec$''}, \text{``$\succ$''}, \text{``incomparable''}\rbrace$ and a function $\phi:\lbrace +,-,0\rbrace\to\Lambda$ satisfying
\[\phi(+,-)=\phi(0,-)=\text{``$\prec$''},\quad  \phi(+,+)=\phi(0,+)=\text{``$\succ$''},\quad \phi(-,+)=\phi(-,0)=\phi(-,-)=\text{``incomparable''}\]
(note that we do not need to specify $\phi(+,0)$ and $\phi(0,0)$ since it is not possible to have $P_1(x,y,r,x',y',r')\geq 0$ and $P_2(x,y,r,x',y',r')=0$ for distinct $(x,y,r),(x',y',r')\in U\su \R^3$). Then we have $D\subset D'$ if and only if $\phi(\sgn P_1(x,y,r,x',y',r'),\sgn P_2(x,y,r,x',y',r')) =\text{``$\prec$''}$, and we have $D'\subset D$ if and only if $\phi(\sgn P_1(x,y,r,x',y',r'),\sgn P_2(x,y,r,x',y',r')) =\text{``$\succ$''}$.

Given a circle order on $\lbrace 1,\dots,n\rbrace$, let us define an edge-labeling of the complete graph on the vertex set $\lbrace 1,\dots,n\rbrace$ with labels in  $\Lambda=\lbrace \text{``$\prec$''}, \text{``$\succ$''}, \text{``incomparable''}\rbrace$ as follows: For $1\leq i<j\leq n$, let the label of the edge $ij$ be ``$\prec$'' if we have $i \prec j$ in the circle order, `$\succ$'' if we have $i \succ j$ in the circle order, and ``incomparable'' if $i$ and $j$ are incomparable in the circle order. This gives a correspondence between the circle orders on $\lbrace 1,\dots,n\rbrace$ and the $(P_1, P_2,\phi,U, \Lambda)$-representable edge-labelings of the complete graph on the vertex $\lbrace 1,\dots,n\rbrace$.

Thus, the number of circle orders on $\lbrace 1,\dots,n\rbrace$ equals the number of $(P_1, P_2,\phi,U, \Lambda)$-representable edge-labelings of the complete graph on the vertex $\lbrace 1,\dots,n\rbrace$. In order to apply Theorem \ref{theo-main} to obtain a lower bound for this number, it only remains to check the following condition: For any distinct $(x,y,r), (x',y',r')\in U$ we need to have a point $(x^*,y^*,r^*)\in U$ such that $P_i(x,y,r,x^*,y^*,r^*)\neq 0$ and $P_i(x',y',r',x^*,y^*,r^*)\neq 0$ for $i=1,2$ and
\begin{multline*}
\phi(\sgn P_1(x,y,r,x^*,y^*,r^*), \sgn P_2(x,y,r,x^*,y^*,r^*))\\
\neq \phi(\sgn P_1(x',y',r',x^*,y^*,r^*), \sgn P_2(x',y',r',x^*,y^*,r^*)).
\end{multline*}
But this simply means that for any distinct closed disks $D, D'$ there exists a disk $D^*$ which contains exactly one of the disks $D$ and $D'$  and has a radius distinct from the radii of $D$ and $D'$ and such that the boundary circle of $D^*$ is neither tangent (from the inside) to the boundary circle of $D$ nor to the boundary circle of $D'$. As in the previous geometric applications, this statement is again very easy to check.

Thus, Theorem \ref{theo-main} indeed implies the lower bound in Corollary \ref{coro-order-circles} (and Theorem \ref{theorem-upper-bound} reproves the upper bound due to Alon and Scheinerman \cite{alon-scheinerman}).

Note that our arguments above trivially generalize to any dimension other than two. Thus, Theorems \ref{theorem-upper-bound} and \ref{theo-main} also yield the following corollary 

\begin{corollary}
For any $m\geq 1$, the number of partial orders on $\lbrace 1,\dots,n\rbrace$ given by the containment relations of $n$ numbered closed balls in $\R^m$ equals $n^{(m+1+o(1))n}$.
\end{corollary}

In particular, for $m=1$ we obtain that the number of partial orders on $\lbrace 1,\dots,n\rbrace$ given by the containment relations of $n$ numbered closed intervals equals $n^{(2+o(1))n}$.

Alon and Scheinerman \cite{alon-scheinerman} also proved a similar result for containment orders of polygons with a fixed number of vertices. For fixed $m\geq 3$, an \emph{$m$-gon order} is a partial order on the set $\lbrace 1,\dots,n\rbrace$ given by the containment relations of $n$ different $m$-gons (polygons with $m$ vertices each) in the plane which are numbered from $1$ to $n$ (see \cite{alon-scheinerman, sidney-et-al}). By encoding $m$-gons by points in $\R^{2m}$ such that the containment relations are described by the signs of a finite list of polynomials, Alon and Scheinerman \cite{alon-scheinerman} established that the number of $m$-gon orders on $\lbrace 1,\dots,n\rbrace$ is at most $n^{(2m+o(1))n}$ (for fixed $m$). With the same arguments as above for disks, one can deduce from Theorem \ref{theo-main} that this bound is sharp. Thus, one obtains the following corollary.

\begin{corollary}
For any $m\geq 3$, the number of $m$-gon orders on $\lbrace 1,\dots,n\rbrace$ equals $n^{(2m+o(1))n}$.
\end{corollary}

Alon and Scheinerman \cite{alon-scheinerman} also investigated angle orders. An \emph{angle} is given by the intersection of two closed half-planes (which are bounded by non-parallel lines). An \emph{angle order}  is a partial order on the set $\lbrace 1,\dots,n\rbrace$ given by the containment relations of $n$ numbered angles (see \cite{alon-scheinerman, fishburn, fishburn-trotter}). Alon and Scheinerman \cite{alon-scheinerman} proved that the number of angle orders on $\lbrace 1,\dots,n\rbrace$ is at most $n^{(4+o(1))n}$. This can be proved as follows (we believe that one has to be slightly more careful than in \cite{alon-scheinerman}): Assuming that none of the half-planes is bounded by a vertical line, each of the half-planes is described by an inequality of the form $y\geq \alpha x+\beta$ or $y\leq \alpha x+\beta$. If one first chooses one of the $2^{2n}$ possibilities for the inequality signs in the descriptions of the half-planes determining the angles, one can then encode any configuration of $n$ angles as a point in $\R^{4n}$ such that the containment relations between the angles are described by the signs of a list of $O(n^2)$ polynomials. The arguments of Alon and Scheinerman \cite{alon-scheinerman} then give an upper bound of $n^{(4+o(1))n}$ for the number of angle orders with the chosen inequality signs for the half-planes. All in all, the number of angle orders on $\lbrace 1,\dots,n\rbrace$ is therefore at most $2^{2n}\cdot n^{(4+o(1))n}=n^{(4+o(1))n}$.

Theorem \ref{theo-main} can be used to obtain a matching lower bound for the number of angle orders on $\lbrace 1,\dots,n\rbrace$. Indeed, for a lower bound it suffices to only consider angles obtained from half-planes described by inequalities of the form $y\geq \alpha x+\beta$ for some $\alpha, \beta\in \R$. Each such angle, given by the inequalities $y\geq \alpha_1 x+\beta_1$ and $y\geq \alpha_2 x+\beta_2$ with $\alpha_1< \alpha_2$, corresponds to a point in $U=\lbrace (\alpha_1,\beta_1,\alpha_2, \beta_2)\in \R^4\mid \alpha_1<\alpha_2\rbrace$. Whether one such angle is contained in another one can be determined from the signs of a finite list of polynomials in the coordinates of the corresponding points in $U$. Again, it is easy to check that in this set-up all conditions in Theorem \ref{theo-main} are satisfied. Thus, Theorem \ref{theo-main} implies the following corollary.

\begin{corollary}
The number of angle orders on $\lbrace 1,\dots,n\rbrace$ equals $n^{(4+o(1))n}$.
\end{corollary}

\subsection{Partial orders of a given dimension}

Recall that the dimension of a partial order on the set $\lbrace 1,\dots, n\rbrace$ is the minimum integer $d$ such that there exist points $a_1,\dots,a_n\in \R^d$ satisfying the following conditions: For all distinct $i,j\in \lbrace 1,\dots, n\rbrace$ the points $a_i=(a_i^{(1)},\dots, a_i^{(d)})$ and $a_j=(a_j^{(1)},\dots, a_j^{(d)})$ satisfy $a_i^{(\l)}\neq a_j^{(\l)}$ for all $\l=1,\dots,d$, and furthermore we have $i \prec j$ in the partial order if and only $a_i^{(\l)}\leq a_j^{(\l)}$ for all $\l=1,\dots,d$.

Alon and Scheinerman \cite[Theorem 1]{alon-scheinerman} proved that for fixed $d$ the number of partial orders on $\lbrace 1,\dots, n\rbrace$ of dimension at most $d$ equals $n^{(d+o(1))n}$. This result can also be obtained as an easy corollary of Theorems \ref{theorem-upper-bound} and \ref{theo-main}.

\begin{corollary}[\cite{alon-scheinerman}]
For any $d\geq 1$, the number of partial orders on $\lbrace 1,\dots,n\rbrace$ of dimension at most $d$ equals $n^{(d+o(1))n}$.
\end{corollary}

\begin{proof}
Let $U=\R^d$, let $P_s(x_1,\dots,x_d,y_1,\dots,y_d)=y_s-x_s$ for $s=1,\dots,d$, and let us define the function $\phi:\lbrace +,-,0\rbrace^d \to \Lambda$, where $\Lambda=\lbrace \text{``$\prec$''}, \text{``$\succ$''}, \text{``incomparable''}\rbrace$, as follows: Let us send all non-zero $d$-tuples in $\lbrace +,0\rbrace^d$ to ``$\prec$'', send all non-zero $d$-tuples in $\lbrace -,0\rbrace^d$ to ``$\succ$'', and send all remaining $d$-tuples in $\lbrace +,-,0\rbrace^d$ to ``incomparable''. Then for any distinct points $a, a'\in U=\R^d$ we have $\phi(\sgn P_1(a,a'), \dots, \sgn P_d(a,a'))=\text{``$\prec$''}$ if and only if $a^{(\l)}\leq a'^{(\l)}$ for all $\l=1,\dots,d$, and we have $\phi(\sgn P_1(a,a'), \dots, \sgn P_d(a,a'))=\text{``$\succ$''}$ if and only if $a^{(\l)}\geq a'^{(\l)}$ for all $\l=1,\dots,d$.

Now, the number of partial orders on $\lbrace 1,\dots, n\rbrace$ of dimension at most $d$ equals the number of strongly $(P_1,\dots,P_s,\phi,U,\Lambda)$-representable edge-labelings of the complete graph on the vertex $\lbrace 1,\dots,n\rbrace$, and by Theorems \ref{theorem-upper-bound} and \ref{theo-main} this number equals $n^{(d+o(1))n}$, as desired (again, the assumptions of Theorem \ref{theo-main} are easy to check).
\end{proof}

Note that applying this result for $d$ and $d-1$, we see that the number of partial orders on $\lbrace 1,\dots, n\rbrace$ of dimensional exactly $d$ equals 
$n^{(d+o(1))n}- n^{(d-1+o(1))n}=n^{(d+o(1))n}$.

\section{Algebraic Preliminaries for the proof of Theorem \ref{theo-main}}\label{sect-algebraic}

In this section, we will state several preliminaries from algebra and algebraic geometry. These statements are all known and possibly obvious to experts, but for completeness we provide proofs in the appendix.

We start with the following very easy fact (for the proof see Subsection \ref{sect-a-facts} of the appendix).

\begin{fact}\label{fact-poly-vanishing} Let $m\geq 1$ and let $Q_1,\dots,Q_\l\in \R[x_1,\dots,x_m]$ be non-zero real polynomials. Then for any non-empty open set $U\su \R^m$, we can find a point $x\in U$ such that $Q_i(x)\neq 0$ for $i=1,\dots,\l$.
\end{fact}

We will need some facts from algebraic geometry. Since stating these facts will not require any scheme-theory, we will formulate everything just in terms of real algebraic sets. Our notation follows \cite{real-ag-book}.

For an ideal $I\su \R[x_1,\dots,x_m]$, define its \emph{zero-set} to be
\[\mathcal{Z}(I)=\lbrace x\in \R^m\mid Q(x)=0 \text{ for all }Q\in I\rbrace.\]
A \emph{real algebraic set} $V$ in $\R^m$ is a subset $V\su \R^m$ of the form $V=\mathcal{Z}(I)$ for some ideal $I\su \R[x_1,\dots,x_m]$. Note that there can be different ideals (even different prime ideals) $I$ yielding the same real algebraic set $\mathcal{Z}(I)$. However, for each real algebraic set $V$ in $\R^m$, we can define the ideal
\[\mathcal{I}(V)=\lbrace Q\in \R[x_1,\dots,x_m]\mid Q(x)=0 \text{ for all }x\in V\rbrace.\]
Note that then for every ideal $I\su \R[x_1,\dots,x_m]$ with $V=\mathcal{Z}(I)$, we have $I\su \mathcal{I}(V)$.

Clearly, each real algebraic set is closed (in the Euclidean topology). Throughout this whole paper, we the notions ``open'' and ``closed'' refer to the Euclidean topology unless explicitly noted otherwise.

The \emph{dimension} $\dim V$ of a real algebraic set $V$ is the dimension of the ring $\R[x_1,\dots,x_m]/\mathcal{I}(V)$. In other words, it is the maximum integer $d$ such that there exists prime ideals $\mathfrak{p}_0, \mathfrak{p}_1,\dots, \mathfrak{p}_d$ in $\R[x_1,\dots,x_m]$ with $\mathcal{I}(V)\su \mathfrak{p}_0\subsetneq \mathfrak{p}_1\subsetneq\dots\subsetneq \mathfrak{p}_d$. Note that we have $\mathcal{I}(V)=\R[x_1,\dots,x_m]$ if and only if $V$ is the empty set. Using the convention that the dimension of the zero ring is $-\infty$, the dimension of $V=\emptyset$ equals $-\infty$ (but the dimension of any non-empty real algebraic set is non-negative).

We will prove the following three facts in Subsections \ref{sect-a-facts} and \ref{sect-a-cutting-dim} of the appendix.

\begin{fact}\label{fact-dim-union} Let $V_1,\dots,V_\l$ be real algebraic sets in $\R^m$ and let $V=V_1\cup\dots\cup V_\l$ be their union. Then $V$ is also a real algebraic set in $\R^m$ and $\dim V=\max_i\, (\dim V_i)$.
\end{fact}

\begin{fact}\label{fact-two-coprime-poly} Let $\l\geq 2$ and let $P_1,\dots, P_\l\in \R[x_1,\dots,x_m]$ be polynomials such that $P_1$ is irreducible and at least one of the polynomials $P_2,\dots,P_\l$ is not divisible by $P_1$. Then the set
\[\lbrace x\in \R^m\mid P_1(x)=\dots=P_\l(x)=0\rbrace\]
is a real algebraic set in $\R^m$ of dimension at most $m-2$.
\end{fact}

\begin{fact}\label{fact-cutting-dimension} Let $V$ be a real algebraic set in $\R^{2n}=\R^n\times \R^n$ and assume that $\dim V\leq 2n-2$. Then there exists a dense open set $U\su \R^n$ such that each point $a\in U$ satisfies the following condition: The set $\lbrace b\in \R^n\mid (a,b)\in V\rbrace$ is a real algebraic set of dimension at most $n-2$.
\end{fact}

In Fact \ref{fact-cutting-dimension}, the set  $\lbrace b\in \R^n\mid (a,b)\in V\rbrace$ is actually a real algebraic set for all $a\in \R^n$, but the crucial part of the condition is that dimension of this set is at most $n-2$.

Recall that an open subset of $\R^m$ is connected if and only if it is path-connected.

\begin{fact}\label{fact-complement-variety} Let $V\su \R^m$ be a real algebraic set of dimension at most $m-2$. Furthermore, let $U\su \R^m$ be a connected open set. Then the set $U\sm V$ is open and connected (and hence path-connected).
\end{fact}

A proof of Fact \ref{fact-complement-variety} will be provided in Subsection \ref{sect-a-complement} of the appendix.

As usual, for a point $a\in \R^n$, we call an open set $U\su \R^n$ with $a\in U$ an \emph{open neighborhood of $a$}.

The following lemma is a relatively straightforward linear algebra statement. For the reader's convenience we give a proof in Subsection \ref{sect-a-lemma-matrix} of the appendix.

\begin{lemma}\label{lemma-matrix} Let $1\leq \l<d$ be integers and let $U\su \R^d$ be an open subset. Suppose that for each $x\in U$ we are given an $(\l\times d)$-matrix $A(x)$ in such a way that all coefficients of $A$ are smooth functions of $x\in U$. Furthermore, suppose that we are given some point $x_0\in U$ such that the matrix $A(x_0)$ has rank $\l$. Then there exists an open neighborhood $U'\su U$ of $x_0$ such that for each $x\in U'$, the matrix $A(x)$ has rank $\l$. Furthermore, $U'$ can be chosen in such a way that there exists a smooth vector field $w:U'\to \R^d$ with $w(x)\neq 0$ and $A(x)w(x)=0$ for each $x\in U'$.
\end{lemma}

The next fact is a consequence of a Theorem of Milnor \cite{milnor} and (in a similar form) independently Thom \cite{thom}. We will provide the details of the proof of this fact in Subsection \ref{sect-a-facts} of the appendix. Recall that we defined the notion of a subset of $\R^d$ being definable by polynomials in Definition \ref{defi-definable-polynomials}.

\begin{fact}\label{fact-components-finite} Let $U\su \R^d$ be an open set which is definable by polynomials and let $R_1,\dots,R_k\in \R[x_1,\dots, x_d]$. Then the number of connected components of the open set
\[\lbrace x\in U \mid R_i(x)\neq 0\text{ for }i=1,\dots,k\rbrace\]
is finite.
\end{fact}

Finally, we will use the following notations in the proof of Theorem \ref{theo-main}.

\begin{notation}\label{notation-equiv-class-poly}
For two polynomials $Q,R\in \R[y_1,\dots, y_d]$, let us write $Q\equiv R$ if there is a real number $c\neq 0$ with $Q=c\cdot R$. This is clearly an equivalence relation, so let us write $[Q]$ for the equivalence class of a polynomial $Q$ under this relation. Note that if $Q\equiv R$, then for all $y\in \R^d$ we have $Q(y)=0$ if and only if $R(y)=0$. Furthermore, if $Q$ and $R$ are irreducible and $[Q]\neq [R]$, then $Q$ and $R$ are coprime.
\end{notation}

\begin{notation}Note that for a polynomial $P\in \R[x_1,\dots,x_d, y_1,\dots,y_d]$, and points $a,b\in \R^d$, concatenating $a$ and $b$ gives a vector of length $2d$, and so $P(a,b)$ is well-defined. The vector $\nabla P(a,b)$ of the partial derivatives of $P$ at the point $(a,b)$ has length $2d$. Let us denote the vector formed by the first $d$ entries as $\nabla_a P(a,b)$, this is the vector of the partial derivatives with respect to the entries of $a$. Similarly, let $\nabla_b P(a,b)$ consist of the last $d$ entries of $\nabla P(a,b)$, then $\nabla_b P(a,b)$ is the vector of the partial derivatives with respect to the entries of $b$.
\end{notation}

\section{Proof of Theorem \ref{theo-main}}
\label{sect-proof-theo-main}

This section contains the proof of Theorem \ref{theo-main}, apart from the proofs of several lemmas which we postpone to the following sections.

Let us fix a finite set $\Lambda$, an integer $d\geq 1$, polynomials $P_1,\dots,P_k\in \R[x_1,\dots,x_d,y_1,\dots,y_d]$, a function $\phi: \lbrace +,-,0\rbrace^k \to \Lambda$, and an open subset $U\su \R^d$ as in the statement of Theorem \ref{theo-main}.

First, note that none of the polynomials $P_1,\dots,P_k$ is the zero polynomial. Indeed, consider any two distinct points $a,a'\in U$. By the assumption in Theorem \ref{theo-main} there exists a point $b\in U$ with $P_s(a, b)\neq 0$ and $P_s(a', b)\neq 0$ for all $1\leq s\leq k$. This in particular implies that all of the polynomials $P_s$ are non-zero.

Also note that we may assume that the polynomials $P_1,\dots,P_k$ are irreducible and mutually coprime. Otherwise we can replace the list $P_1,\dots,P_k$ by the list $P^{*}_1,\dots,P^{*}_{k'}$ of irreducible factors of $P_1,\dots,P_k$. Then for any $a,b\in U$ with $P_s(a,b)\neq 0$ for all $1\leq s\leq k$, we also have $P^{*}_t(a,b)\neq 0$ for all $1\leq t\leq k'$. Furthermore, all the signs $\sgn P_s(a,b)$ are determined by the signs $P^{*}_t(a,b)\neq 0$ for $1\leq t\leq k'$. Hence we can find a function $\phi^{*}: \lbrace +,-,0\rbrace^{k'} \to \Lambda$ such that
\[\phi^{*}\big(\sgn P^{*}_1(a,b),\dots ,\sgn P^{*}_{k'}(a,b)\big)=\phi\big(\sgn P_1(a,b),\dots ,\sgn P_k(a,b)\big)\]
for all $a,b\in U$. By considering $P^{*}_1,\dots,P^{*}_{k'}$ and $\phi^{*}$ instead, we may from now on assume that $P_1,\dots,P_k$ are irreducible and mutually coprime.

To simplify notation, let us from now on abbreviate $\phi\big(\sgn P_1(a,b),\dots ,\sgn P_k(a,b)\big)$ by $\Phi(a,b)$ for any $a,b\in U$ (note that then $\Phi(a,b)\in \Lambda$).

\begin{definition} For every $\lambda\in \Lambda$, let $T_{\lambda}\su U\times U$ be the set of all pairs $(a,b)\in U\times U$ with the following property: There exists an open neighborhood $U_a\su U$ of $a$ and an open neighborhood $U_b\su U$ of $b$ such that for all $a'\in U_a$ and $b'\in U_a$ the pair $(a',b')$ either satisfies $P_s(a',b')=0$ for some $1\leq s\leq k$ or $\Phi(a',b')=\lambda$.
\end{definition}

It is easy to see that for each  $\lambda\in \Lambda$, the set $T_\lambda$ is an open subsets of $U\times U\su \R^{2d}$. Furthermore, all the sets $T_\lambda$ for $\lambda\in \Lambda$ are disjoint: Indeed, if $(a,b)\in T_{\lambda}\cap T_{\lambda'}$ for distinct $\lambda, \lambda'\in \Lambda$, then there must be open neighborhoods $U_a\su U$ and $U_b\su U$ of $a$ and $b$, respectively, such that for all $a'\in U_a$ and $b'\in U_b$ we have $P_s(a',b')=0$ for some $1\leq s\leq k$. But this contradicts Fact \ref{fact-poly-vanishing} applied to $P_1,\dots,P_k$ and the open set $U_a\times U_b\su \R^{2d}$. Thus, the sets $T_\lambda$ for $\lambda\in \Lambda$ are disjoint open subsets of $U\times U\su \R^{2d}$.

The following definition introduces a key notion for our proof of Theorem \ref{theo-main}.

\begin{definition}Let us call a pair $(a,b)\in U\times U$ a \emph{wall pair} if $(a,b)\not\in \bigcup_{\lambda\in \Lambda}T_{\lambda}$.
\end{definition}

\begin{claim}\label{claim-special-P-s-0}Let $(a,b)\in U\times U$ be a wall pair. Then $P_s(a,b)=0$ for at least one $1\leq s\leq k$.
\end{claim}
\begin{proof} Suppose we had $P_s(a,b)\neq 0$ for all $1\leq s\leq k$. Then by continuity of the polynomials $P_s$ we can find open neighborhoods $U_a\su U$ and $U_b\su U$ of $a$ and $b$, respectively, such that $\sgn P_s(a',b')=\sgn P_s(a,b)$ for all $a'\in U_a$, $b'\in U_b$ and $1\leq s\leq k$. But then we obtain $\Phi(a',b')=\Phi(a,b)$ for all $a'\in U_a$ and $b'\in U_b$, establishing that $(a,b)\in T_\lambda$ for $\lambda=\Phi(a,b)\in \Lambda$. This contradicts $(a,b)$ being a wall pair.
\end{proof}

Note that the zero-sets of the polynomials $P_1,\dots,P_k$ divide the set $U\times U$ into different regions and the function $\Phi: U\times U\to \Lambda$ is constant inside each of these regions. Intuitively, a wall pair $(a,b)\in U\times U$ is a point lying on the boundary between multiple such regions where $\Phi$ takes different values. The following claim is the reason because of which wall pairs play an important role in our proof.

\begin{claim}\label{claim-path-no-wall} Let $(a,b), (a',b')\in U\times U$ be such that $P_s(a,b)\neq 0$ and $P_s(a',b')\neq 0$ for all $1\leq s\leq k$. Suppose that there is a continuous path $\gamma: [0,1]\to U\times U$ with $\gamma(0)=(a,b)$ and $\gamma(1)=(a',b')$ and such that there is no $t\in [0,1]$ for which $\gamma(t)$ is a wall pair. Then $\Phi(a,b)=\Phi(a',b')$.
\end{claim}
\begin{proof} Recall that the sets $T_\lambda$ for $\lambda\in \Lambda$ are disjoint open subsets of $U\times U\su \R^{2d}$. Furthermore, for each $t\in [0,1]$ we have $\gamma(t)\in \bigcup_{\lambda\in \Lambda}T_{\lambda}$, since otherwise $\gamma(t)$ would be a wall pair. Thus, the sets $\gamma^{-1}(T_\lambda)$ for $\lambda\in \Lambda$ are disjoint open subsets of $[0,1]$ and their union is the entire interval $[0,1]$. Since the interval $[0,1]$ is connected (and the set $\Lambda$ is finite), this implies that $\gamma^{-1}(T_\lambda)=[0,1]$ for some $\lambda\in \Lambda$. In particular, we have $(a,b), (a',b')\in T_\lambda$, and therefore $\Phi(a,b)=\lambda=\Phi(a',b')$.\end{proof}

\begin{definition}\label{defi-special}Let us call a pair $(a,b)\in \R^d\times \R^d$ \emph{special} if at least one of the following two conditions holds:
\begin{itemize}
\item $P_s(a,b)=0$ for at least two different indices $s\in \lbrace 1,\dots,k\rbrace$.
\item There exists an index $s\in \lbrace 1,\dots,k\rbrace$ with $P_s(a,b)=0$ and $\nabla_b P_s(a,b)=0\in \R^d$.
\end{itemize}
\end{definition}

\begin{definition}Let us call a pair $(a,b)\in U\times U$ a \emph{general wall pair} if $(a,b)$ is a wall pair and $(a,b)$ is not special.
\end{definition}

By Claim \ref{claim-special-P-s-0}, each general wall pair $(a,b)$ satisfies $P_s(a,b)=0$ for exactly one index $1\leq s\leq k$. Let us call this index $s$ the \emph{wall index} of the pair $(a,b)$.

Roughly speaking, a general wall pair $(a,b)\in U\times U$ is a point of the boundary of exactly two regions of $U\times U$ cut out by the zero-sets of the polynomials $P_1,\dots,P_k$, such that $\Phi$ takes different values on these two regions (and such that the normal vector at $(a,b)$ to the boundary of the two regions  has a non-zero entry in at least one of the coordinates corresponding to $b$).

\begin{claim}\label{claim-general-wall-sign-flip}Let $(a,b)\in U\times U$ be a general wall pair with wall index $s\in \lbrace 1,\dots,k\rbrace$. Then we have
\begin{multline*}
\phi\big(\sgn P_1(a,b),\dots ,\sgn P_{s-1}(a,b),+,\sgn P_{s+1}(a,b),\dots ,\sgn P_{k}(a,b)\big)\\
\neq \phi\big(\sgn P_1(a,b),\dots ,\sgn P_{s-1}(a,b),-,\sgn P_{s+1}(a,b),\dots ,\sgn P_{k}(a,b)\big).
\end{multline*}
\end{claim}
In other words, the conclusion in Claim \ref{claim-general-wall-sign-flip} means that  replacing the zero in the $s$-th position of the tuple $\big(\sgn P_1(a,b),\dots ,\sgn P_{k}(a,b)\big)$ by either $+$ or $-$ leads to different values when applying the function $\phi$.

\begin{proof}[Proof of Claim \ref{claim-general-wall-sign-flip}]
Suppose for contradiction that
\begin{multline*}
\phi\big(\sgn P_1(a,b),\dots ,\sgn P_{s-1}(a,b),+,\sgn P_{s+1}(a,b),\dots ,\sgn P_{k}(a,b)\big)\\
= \phi\big(\sgn P_1(a,b),\dots ,\sgn P_{s-1}(a,b),-,\sgn P_{s+1}(a,b),\dots ,\sgn P_{k}(a,b)\big).
\end{multline*}
Then both of these terms are equal to the same element $\lambda\in \Lambda$. By continuity of the polynomials $P_1,\dots,P_k$, we can find open neighborhoods $U_a\su U$ and $U_b\su U$ of $a$ and $b$, respectively, such that $\sgn P_t(a',b')=\sgn P_t(a,b)$ for all $a'\in U_a$, $b'\in U_b$ and $t\in \lbrace1,\dots,k\rbrace\sm \lbrace s\rbrace$. Then for all $a'\in U_a$ and $b'\in U_b$ we either have $P_s(a',b')=0$ or $\Phi(a',b')=\phi\big(\sgn P_1(a',b'),\dots ,\sgn P_{k}(a',b')\big)$ is one of the two terms in the equation above and therefore equal to $\lambda$. This shows that $(a,b)\in T_\lambda$, a contradiction to $(a,b)$ being a wall pair.
\end{proof}

\begin{definition}\label{defi-L-a}For $a\in U$, define the subspace $L_a\su \R^d$ as
\[L_a=\operatorname{span}\lbrace \nabla_a P_s(a,b) \mid b\in U\text{ and }1\leq s\leq k \text{ such that } (a,b) \text{ is a general wall pair with wall index } s\rbrace.\]
\end{definition}

The main difficulty in the proof of Theorem \ref{theo-main} is to prove the following lemma.

\begin{lemma}\label{lemma-main} There exists a point $a\in U$ with $L_a=\R^d$.
\end{lemma}

If $(a,b)\in U\times U$ is a general wall pair with wall index $s$, then $a$ lies in the zero-set of the polynomial $P_s(\_,b)$. If $\nabla_a P_s(a,b)\neq 0$, then the vector $\nabla_a P_s(a,b)$ is the normal vector to the tangent hyperplane of the zero-set of $P_s(\_,b)$ at the point $a$. Lemma \ref{lemma-main} states that there is a point $a\in U$, such that all of these normal vectors for all $b\in U$ and $1\leq s\leq k$ for which $(a,b)$ is a general wall pair with wall index $s$ are spanning $\R^d$. Intuitively speaking, this means that the tangent hyperplanes to the zero-sets of the polynomials $P_s(\_,b)$ at the point $a$ are ``spanning all directions of $\R^d$'' (i.e.\ there is no common line contained in all of these hyperplanes).

We postpone the proof of Lemma \ref{lemma-main} to Section \ref{sect-proof-lemma-main}. We will now finish the proof of Theorem \ref{theo-main} assuming Lemma \ref{lemma-main}.

Let us fix a point $a^*\in U$ as in Lemma \ref{lemma-main}. Then we can find points $b_1,\dots, b_d\in U$ and indices $s_1,\dots,s_d\in \lbrace 1,\dots,k\rbrace$ such that for all $1\leq i\leq d$ the pair $(a^*,b_i)$ is a general wall pair with wall index $s_i$ and such that
\[\operatorname{span} \lbrace  \nabla_a P_{s_1}(a^*,b_1),\dots,\nabla_a P_{s_d}(a^*,b_d)\rbrace = L_{a^*} = \R^d.\]

The following lemma states that for any $m$ we can find points $b_i^j\in U$ for $i=1,\dots,d$ and $j=1,\dots,m$ with certain technical conditions (each $b_i^j$ will be chosen by slightly perturbing the point $b_i$ in a carefully chosen way).

This lemma will later allow us to construct many strongly $(P_1,\dots,P_k,\phi,U,\Lambda)$-representable edge-labelings of the complete graph on the vertex set $\{1,\dots,n\}$, with the following approach. Such an edge-labeling is specified by choosing points $a_1,\dots,a_n\in U$ and determining the labels on the edges of the complete graph as in Definition \ref{defi-representable} (using $P_1,\dots,P_k$ and $\phi$). Roughly speaking, our approach will be to take the last $md$ of the points $a_1,\dots,a_n$ to be the points $b_1^1,\dots,b_1^m,\dots,b_d^1,\dots,b_d^m$ obtained in Lemma \ref{lemma-b-i-j} (for some parameter $m$ depending on $n$). Each of the remaining $n-md$ points $a_1,\dots,a_{n-md}$ can then be chosen by applying Lemma \ref{lemma-b-i-j} with any choice of $(j_1,\dots,j_d)\in \lbrace 1,\dots,m\rbrace^d$ (and taking the resulting point $a\in U$). For each of the points $a_1,\dots,a_{n-md}$, there are $m^d$ possibilities to choose a $d$-tuple $(j_1,\dots,j_d)\in \lbrace 1,\dots,m\rbrace^d$, so in total we have $(m^d)^{n-md}$ choices for all of these $d$-tuples. Using the conditions in Lemma \ref{lemma-b-i-j}, we can show that when we choose $a_\ell$ for some $1\leq \ell\leq n-md$, then each choice for the $d$-tuple $(j_1,\dots,j_d)\in \lbrace 1,\dots,m\rbrace^d$ leads to a different outcome of $\left(\Phi(a_\ell,a_{n-md+1}),\dots, \Phi(a_\ell,a_{n})\right)=\left(\Phi(a_\ell,b_1^1),\dots, \Phi(a_\ell,b_d^m)\right)$ and therefore to a different outcome of the labels on the edges from vertex $\ell$ to the vertices $n-md+1,\dots,n$. Thus, we can construct at least $(m^d)^{n-md}$ different strongly $(P_1,\dots,P_k,\phi,U,\Lambda)$-representable edge-labelings (which differ from each other in the labels on the edges between the vertices $1,\dots,n-md$ and $n-md+1,\dots,n$). By taking $m=\lfloor n/\ln n\rfloor$, we will then obtain the desired lower bound for the number of strongly $(P_1,\dots,P_k,\phi,U,\Lambda)$-representable edge-labelings in Theorem \ref{theo-main} .

\begin{lemma}\label{lemma-b-i-j} For any positive integer $m$, there exist points $b_i^j\in U$ for $i=1,\dots,d$ and $j=1,\dots,m$ such that the following holds: For every $d$-tuple $(j_1,\dots,j_d)\in \lbrace 1,\dots,m\rbrace^d$ one can find a point $a\in U$ such that for all $1\leq i\leq d$ and all $1\leq j\leq m$ the following four conditions are satisfied:
\begin{itemize}
\item[(i)] $P_{s}(a,b_i^j)\neq 0$ for all $s=1,\dots,k$.
\item[(ii)] $P_{s_i}(a,b_i^j)>0$ if $j\leq j_i$.
\item[(iii)] $P_{s_i}(a,b_i^j)<0$ if $j> j_i$.
\item[(iv)] $\sgn P_{s}(a,b_i^j)=\sgn P_{s}(a^*,b_i)$ for all $s\in \lbrace 1,\dots,k\rbrace\sm \lbrace s_i\rbrace$.
\end{itemize}
\end{lemma}

The details of the proof of Lemma \ref{lemma-b-i-j} are a bit technical, so we postpone the proof to Section \ref{sect-lemma-b-i-j}. Here, we just describe the rough idea behind the proof of the lemma. Recall that for $i=1,\dots,d$ the pair $(a^*,b_i)$ is a general wall pair with wall index $s_i$, which in particular means that $P_{s_i}(a^*,b_i)=0$. If for some $i$ we perturb the point $b_i\in U$ slightly, this will change the polynomial $P_{s_i}(\_,b_i)$ a little bit. In most cases, the zero-set of this new polynomial will not go through the point $a^*$ anymore, but will be shifted away from $a^*$ by a little bit. However, close to $a^*$, this zero-set will roughly look ``parallel'' to the zero-set of the original polynomial $P_{s_i}(\_,b_i)$. By choosing $m$ different perturbations $b_i^1,\dots,b_i^m$ of $b_i$, we obtain $m$ polynomials $P_{s_i}(\_,b_i^1),\dots, P_{s_i}(\_,b_i^m)$ whose zero-sets all look ``parallel'' to the zero-set of $P_{s_i}(\_,b_i)$ in a small open neighborhood of $a^*$. If we choose such perturbations for all the points $b_i$ for $i=1,\dots,d$, then the zero-sets of the polynomials $P_{s_i}(\_,b_i^j)$ divide a small open neighborhood of $a^*$ into open subsets in a grid-like fashion. Here, we are using that  the tangent hyperplanes to the zero-sets of the polynomials $P_{s_1}(\_,b_1),\dots,P_{s_d}(\_,b_d)$ at the point $a^*$ are ``spanning all directions of $\R^d$'' (or more precisely, their normal vectors $\nabla_a P_{s_1}(a^*,b_1),\dots,\nabla_a P_{s_d}(a^*,b_d)$ are spanning $\R^d$). Using the grid-like arrangement of these open subsets, for any $d$-tuple $(j_1,\dots,j_d)\in \lbrace 1,\dots,m\rbrace^d$ one can find a point $a$ with the desired conditions inside one of the open subsets in the grid (where $j_1,\dots,j_d$ are, in some sense, the coordinates of this subset in the grid). As mentioned above, the details of the proof of Lemma \ref{lemma-b-i-j} will be given in Section \ref{sect-lemma-b-i-j}.

In order to construct many strongly $(P_1,\dots,P_k,\phi,U,\Lambda)$-representable edge-labelings, and in particular in order to satisfy the condition in Definition \ref{defi-strongly-representable}, it will be more convenient to use the following slightly different version of Lemma \ref{lemma-b-i-j}. The main different between Lemma \ref{lemma-j-tuple-open-set} and Lemma \ref{lemma-b-i-j} is that for every $d$-tuple $(j_1,\dots,j_d)\in \lbrace 1,\dots,m\rbrace^d$ Lemma \ref{lemma-b-i-j} only claims the existence of one point $a\in U$ with the desired conditions, whereas Lemma \ref{lemma-j-tuple-open-set} demands the existence of an open subset $U'\in U$ such that every point $a'\in U'$ satisfies the desired conditions. It is, however, not hard to deduce Lemma \ref{lemma-j-tuple-open-set} from Lemma \ref{lemma-b-i-j} by simply taking a small open neighborhood of a point $a\in U$ satisfying the conditions in Lemma \ref{lemma-b-i-j}. Thinking of the strategy for the proof of Lemma \ref{lemma-b-i-j} described above, one can take this open neighborhood to be one of the open subsets in the grid-like arrangement.

\begin{lemma}\label{lemma-j-tuple-open-set}Let $m$ be a positive integer and let $b_i^j\in U$ for $i=1,\dots,d$ and $j=1,\dots,m$ be as in Lemma \ref{lemma-b-i-j}. Then for every $d$-tuple $(j_1,\dots,j_d)\in \lbrace 1,\dots,m\rbrace^d$, there exists an open set $U'\su U$ such that for all $a'\in U'$ we have:
\begin{itemize}
\item $P_s(a',b_i^j)\neq 0$ for all $i=1,\dots,d$, all $j=1,\dots,m$ and all $s=1,\dots,k$.
\item For all $i=1,\dots,d$, the number of indices $j\in \lbrace 1,\dots,m\rbrace$ with $\Phi(a',b_i^j)=\Phi(a',b_i^1)$ equals $j_i$.
\end{itemize}
\end{lemma}
\begin{proof}
First, fix a point $a\in U$ satisfying the conditions (i) to (iv) in Lemma \ref{lemma-b-i-j}. Then we can choose an open neighborhood $U'\su U$ of $a$ such that all $a'\in U'$ satisfy $P_s(a',b_i^j)\neq 0$ and $\sgn P_s(a',b_i^j)=\sgn P_s(a,b_i^j)$ for all $i=1,\dots,d$, all $j=1,\dots,m$ and all $s=1,\dots,k$.

Now, in order to check the second condition, fix some $i\in \lbrace 1,\dots,d\rbrace$. Then for every $j\in \lbrace 1,\dots,m\rbrace$ and every $s\in \lbrace 1,\dots,k\rbrace\sm \lbrace s_i\rbrace$, we have
\[\sgn P_s(a',b_i^j)=\sgn P_s(a,b_i^j)=\sgn P_s(a^*,b_i).\]
Furthermore, for $j\leq j_i$, we have
\[\sgn P_{s_i}(a',b_i^j)=\sgn P_{s_i}(a,b_i^j)=+,\]
while for $j> j_i$, we have
\[\sgn P_{s_i}(a',b_i^j)=\sgn P_{s_i}(a,b_i^j)=-.\]
Thus, for $j\leq j_i$ we obtain
\[\Phi(a',b_i^j)=\phi\big(\sgn P_1(a^*,b_i),\dots ,\sgn P_{s_i-1}(a^*,b_i),+,\sgn P_{s_i+1}(a^*,b_i),\dots ,\sgn P_{k}(a^*,b_i)\big),\]
while for $j> j_i$ we get
\[\Phi(a',b_i^j)=\phi\big(\sgn P_1(a^*,b_i),\dots ,\sgn P_{s_i-1}(a^*,b_i),-,\sgn P_{s_i+1}(a^*,b_i),\dots ,\sgn P_{k}(a^*,b_i)\big).\]
But, as $(a^*,b_i)$ is a general wall pair with wall index $s_i$, Claim \ref{claim-general-wall-sign-flip} implies that those two values of $\Phi$ are different. Thus, $\Phi(a',b_i^j)=\Phi(a',b_i^1)$ if and only if $j\leq j_i$. In particular, the number of $j\in \lbrace 1,\dots,m\rbrace$ with $\Phi(a',b_i^j)=\Phi(a',b_i^1)$ equals $j_i$.
\end{proof}

Now, we are finally able to prove a lower bound for the number of strongly $(P_1,\dots,P_k,\phi,U,\Lambda)$-representable edge-labelings of the complete graphs on the vertex set $\lbrace 1,\dots, n\rbrace$. The proof follows the rough strategy outlined earlier (above the statement of Lemma \ref{lemma-b-i-j}), but one needs to be slightly more careful in order to ensure that the points $a_1,\dots,a_n\in U$ satisfy $P_s(a_i,a_j)\neq 0$ for all $1\leq i<j\leq n$ and all $1\leq s\leq k$ (see Definition \ref{defi-strongly-representable}).

\begin{lemma}\label{lemma-lower-bound} For every $1\leq m<n/d$, the number of strongly $(P_1,\dots,P_k,\phi,U,\Lambda)$-representable edge-labelings of the complete graph on the vertex set $\lbrace 1,\dots, n\rbrace$ is at least $m^{d(n-dm)}$.
\end{lemma}
\begin{proof}
First, let us fix points $b_i^j\in U$ for $i=1,\dots,d$ and $j=1,\dots,m$ as in Lemma \ref{lemma-b-i-j}. Next, for $\l=1,\dots,n-md$, choose any $d$-tuple $(j_1^{(\l)},\dots,j_d^{(\l)})\in \lbrace 1,\dots,m\rbrace^d$. Note that for each $\l$, there are $m^d$ choices for such a $d$-tuple, so the total number of choices for all these $d$-tuples is $m^{d(n-dm)}$.

Now, let us define a strongly $(P_1,\dots,P_k,\phi,U,\Lambda)$-representable edge-labeling $F$ of the complete graph on the vertex set $\lbrace 1,\dots, n\rbrace$ that depends on the chosen $d$-tuples $(j_1^{(\l)},\dots,j_d^{(\l)})$ for $\l=1,\dots,n-md$ in such a way that we can recover the $d$-tuples $(j_1^{(\l)},\dots,j_d^{(\l)})$ from $F$. This will establish that the total number of strongly $(P_1,\dots,P_k,\phi,U,\Lambda)$-representable edge-labelings of the complete graph on the vertex set $\lbrace 1,\dots, n\rbrace$ is indeed at least $m^{d(n-dm)}$.

First, we will choose points $a_1,\dots,a_{n-dm}\in U$ with the following properties:
\begin{itemize}
\item $P_s(a_\l,b_i^j)\neq 0$ for all $\l=1,\dots,n-md$, all $s=1,\dots,k$, all $i=1,\dots,d$, and all $j=1,\dots,m$.
\item For all $\l=1,\dots,n-md$ and all $i=1,\dots,d$, the number of indices $j\in \lbrace 1,\dots,m\rbrace$ with $\Phi(a_\l,b_i^j)=\Phi(a_\l,b_i^1)$ equals $j_i^{(\l)}$.
\item $P_s(a_h,a_\l)\neq 0$ for all $1\leq h<\l\leq n-md$ and all $s=1,\dots,k$.
\end{itemize}
Suppose that for some $1\leq \l\leq n-md$ we have already chosen such points $a_1,\dots,a_{\l-1}\in U$. Then applying Lemma \ref{lemma-j-tuple-open-set} gives an open set $U'\su U$ such that all choices of $a_\l\in U'$ satisfy the first two properties above. So we just need to choose some $a_\l\in U'$ such that $P_s(a_h,a_\l)\neq 0$ for all $h=1,\dots,\l-1$ and all $s=1,\dots,k$. This is possible by Fact \ref{fact-poly-vanishing} applied to the $d$-variable polynomials $P_s(a_h, \_)$ (note that these polynomials are non-zero as $P_s(a_h, b_1^1)\neq 0$). Thus, we can indeed choose points $a_1,\dots,a_{n-dm}\in U$ with the three properties listed above.

It remains to choose the points $a_{n-md+1},\dots,a_n$. First, let $a'_{n-md+1},\dots,a'_n\in U$ be defined to be equal to $b_1^1,\dots,b_1^m, b_2^1,\dots, b_2^m,\dots, b_d^1,\dots,b_d^m$ in this order. Then we have $P_s(a_\l,a'_h)\neq 0$ for all $1\leq \l\leq n-md$, all $n-md+1\leq h\leq n$ and all $s=1,\dots,k$. Furthermore, knowing the values $\Phi(a_\l,a'_h)$ for all $1\leq \l\leq n-md$ and $n-md+1\leq h\leq n$ is the same as knowing all the values $\Phi(a_\l,b_i^j)$ for $1\leq \l\leq n-md$, $1\leq i\leq d$ and $1\leq j\leq m$. As for each $\l=1,\dots,n-md$ and each $i=1,\dots, d$, the number of indices $j\in \lbrace 1,\dots,m\rbrace$ with $\Phi(a_\l,b_i^j)=\Phi(a_\l,b_i^1)$ equals $j_i^{(\l)}$, the values $\Phi(a_\l,a'_h)$ for all $1\leq \l\leq n-md$ and $n-md+1\leq h\leq n$ therefore determine all the $d$-tuples $(j_1^{(\l)},\dots,j_d^{(\l)})$.

Let us recursively choose $a_n,a_{n-1},\dots, a_{n-md+1}\in U$ such that $P_s(a_\l,a_h)\neq 0$ for all $n-md+1\leq h\leq n$, all $1\leq \l< h$ and all $s=1,\dots,k$ and such that $\Phi(a_\l,a_h)=\Phi(a_\l,a'_h)$ for all $n-md+1\leq h\leq n$ and all $1\leq \l\leq n-md$. Suppose that for some $n-md+1\leq h\leq n$ we have already chosen such points $a_n, a_{n-1},\dots,a_{h+1}$. Recall that we have $P_s(a_\l,a'_h)\neq 0$ for all $1\leq \l\leq n-md$ and $s=1,\dots,k$. Hence there is an open neighborhood $U'\su U$ of $a'_h$ such that all choices of $a_h\in U'$ satisfy $P_s(a_\l,a_h)\neq 0$ and $\sgn P_s(a_\l,a_h)=\sgn P_s(a_\l,a'_h)$ for all $1\leq \l\leq n-md$ and $s=1,\dots,k$. In particular, for all $a_h\in U'$ we have $\Phi(a_\l,a'_h)=\Phi(a_\l,a_h)$ for all $1\leq \l< n-md$. So we just need to find some $a_h\in U'$ satisfying $P_s(a_h,a_{h^*})\neq 0$ for all $h^*=h+1,\dots,n$ and all $s=1,\dots,k$. This is possible by Fact \ref{fact-poly-vanishing} applied to the $d$-variable polynomials $P_s(\_, a_{h^*})$ (note that these polynomials are non-zero as $P_s(a_1, a_{h^*})\neq 0$). Thus, we can indeed choose points $a_n,a_{n-1},\dots, a_{n-md+1}\in U$ such that $P_s(a_\l,a_h)\neq 0$ for all $n-md+1\leq h\leq n$, all $1\leq \l< h$ and all $s=1,\dots,k$ and such that $\Phi(a_\l,a_h)=\Phi(a_\l,a'_h)$ for all $n-md+1\leq h\leq n$ and all $1\leq \l\leq n-md$.

All in all, we have chosen points $a_1,\dots,a_n\in U$ such that $P_s(a_\l,a_h)\neq 0$ for all $1\leq \l < h\leq n$ and all $s=1,\dots,k$. Thus, $F_{P_1,\dots,P_k,\phi}(a_1,\dots,a_n)$ is a strongly $(P_1,\dots,P_k,\phi,U,\Lambda)$-representable edge-labeling of the complete graph on the vertex set $\lbrace 1,\dots, n\rbrace$. We can recover the values $\Phi(a_\l,a_h)$ for all $1\leq \l < h\leq n$ from the edge-labels in the edge-labeling $F_{P_1,\dots,P_k,\phi}(a_1,\dots,a_n)$. In particular, we can recover all the values $\Phi(a_\l,a_h)$ for $n-md+1\leq h\leq n$ and $1\leq \l\leq n-md$. Recall that $\Phi(a_\l,a_h)=\Phi(a_\l,a'_h)$ for $n-md+1\leq h\leq n$ and $1\leq \l\leq n-md$ and that these values determine all the $d$-tuples $(j_1^{(\l)},\dots,j_d^{(\l)})$. Thus, we can indeed recover all the $d$-tuples $(j_1^{(\l)},\dots,j_d^{(\l)})$ from the strongly $(P_1,\dots,P_k,\phi,U,\Lambda)$-representable edge-labeling $F_{P_1,\dots,P_k,\phi}(a_1,\dots,a_n)$. In particular, there must be at least $m^{d(n-dm)}$ different strongly $(P_1,\dots,P_k,\phi,U,\Lambda)$-representable edge-labelings of the complete graph on the vertex set $\lbrace 1,\dots, n\rbrace$.
\end{proof}

We can finish the proof of Theorem \ref{theo-main} by choosing an appropriate value for $m$ in Lemma \ref{lemma-lower-bound}. For example, by taking $m=\lfloor n/\ln n\rfloor$ for sufficiently large $n$, we obtain that the number of strongly $(P_1,\dots,P_k,\phi,U,\Lambda)$-representable edge-labelings of the complete graph on the vertex set $\lbrace 1,\dots, n\rbrace$ is at least
\[\lfloor n/\ln n\rfloor^{d\cdot (n-d\lfloor n/\ln n\rfloor)}=\left(n^{1-o(1)}\right)^{d\cdot (1-o(1))n}=n^{(1-o(1))dn}.\]
This finishes the proof of Theorem \ref{theo-main} up to proving Lemmas \ref{lemma-main} and \ref{lemma-b-i-j}.

\section{Proof of Lemma \ref{lemma-b-i-j}}\label{sect-lemma-b-i-j}

In this section, we prove Lemma \ref{lemma-b-i-j}, following the strategy outlined below the statement of the lemma. So let us fix some positive integer $m$. Recall that we also fixed a point $a^*\in U$ and points $b_1,\dots, b_d\in U$ such that for all $i=1,\dots, d$ the pair $(a^*,b_i)$ is a general wall pair with wall index $s_i$ and such that
\begin{equation}\label{eq-span}
\operatorname{span} \lbrace  \nabla_a P_{s_1}(a^*,b_1),\dots,\nabla_a P_{s_d}(a^*,b_d)\rbrace = L_{a^*} = \R^d.
\end{equation}
Note that for all $i=1,\dots,d$ we have $P_{s_i}(a^*,b_i)=0$ and $P_s(a^*,b_i)\neq 0$ for all $s\in \lbrace 1,\dots,k\rbrace\setminus \lbrace s_i\rbrace$.

First, we can find $\delta>0$ and a constant $C>0$ satisfying the following properties:
\begin{itemize}
\item For all $a\in \R^d$ with $\Vert a-a^*\Vert<\delta$ we have $a\in U$.
\item For $i=1,\dots,d$ and $b_i'\in \R^d$ with $\Vert b_i'-b_i\Vert<\delta$ we have $b_i'\in U$.
\item For $i=1,\dots,d$ and $a, b_i'\in \R^d$ with $\Vert a-a^*\Vert<\delta$ and $\Vert b_i'-b_i\Vert<\delta$, all $s\in \lbrace 1,\dots,k\rbrace\sm \lbrace s_i\rbrace$ satisfy $P_s(a,b_i')\neq 0$ and $\sgn P_{s}(a,b_i')=\sgn P_{s}(a^*,b_i)$.
\item For $i=1,\dots,d$ and $a, b_i'\in \R^d$ with $\Vert a-a^*\Vert<\delta$ and $\Vert b_i'-b_i\Vert<\delta$, we have
\begin{equation}
\left\vert P_{s_i}(a,b_i')-P_{s_i}(a^*,b_i)-\nabla P_{s_i}(a^*,b_i)\cdot
\begin{pmatrix}
a-a^*\\
b_i'-b_i
\end{pmatrix}\right\vert
\leq C\cdot \big(\Vert a-a^*\Vert^2+\Vert b_i'-b_i\Vert^2\big).
\end{equation}
\end{itemize}
For the last property we used Taylor's theorem.

Furthermore, for $i=1,\dots,d$, we can find a non-zero vector $z_i\in \R^d$ with $z_i\cdot \nabla_a P_{s_j}(a^*,b_j)=0$ for all $j\in \lbrace 1,\dots,d\rbrace\sm \lbrace i \rbrace$. Then by (\ref{eq-span}) we have $z_i\cdot \nabla_a P_{s_i}(a^*,b_i)\neq 0$, so by rescaling $z_i$ we may assume $z_i\cdot \nabla_a P_{s_i}(a^*,b_i)=1$.

Recall that for each $i=1,\dots,d$, we have $P_{s_i}(a^*,b_i)=0$ since $(a^*,b_i)$ is a general wall pair with wall index $s_i$. Because the pair $(a^*,b_i)$ is not special (see Definition \ref{defi-special}), this implies $\nabla_b P_{s_i}(a^*,b_i)\neq 0$, so we can fix a vector $v_i\in \R^d$ with $v_i\cdot \nabla_b P_{s_i}(a^*,b_i)=1$. This vector $v_i$ will determine the direction in which we perturb the point $b_i$ to obtain $b_i^1,\dots,b_i^m$.

Now, let us choose some $\eps>0$ with all of the following properties:
\begin{itemize}
\item $\eps\cdot m\cdot \big(\Vert z_1\Vert + \dots + \Vert z_d\Vert \big)<\delta$.
\item $\eps\cdot m\cdot \Vert v_i\Vert<\delta$ for $i=1,\dots,d$.
\item $\eps\cdot C\cdot m^2\cdot \big(\Vert z_1\Vert + \dots + \Vert z_d\Vert+\Vert v_1\Vert + \dots + \Vert v_d\Vert \big)^2<\frac{1}{2}$.
\end{itemize}

Finally, we are ready to define the desired points $b_i^j$: For $i=1,\dots,d$ and $j=1,\dots,m$, let $b_i^j=b_i+(\frac{1}{2}-j)\cdot \eps \cdot v_i$.

Then we have
\begin{equation}\label{eq-b-dist}
\Vert b_i^j-b_i\Vert=\left(j-\frac{1}{2}\right)\cdot \eps\cdot \Vert v_i\Vert<m\cdot \eps\cdot \Vert v_i\Vert<\delta.
\end{equation}
In particular, we can conclude from the choice of $\delta$ that $b_i^j\in U$.

Now, for every $(j_1,\dots,j_d)\in \lbrace 1,\dots,m\rbrace^d$ we need to find a point $a\in U$ satisfying the conditions (i) to (iv) in the statement of Lemma \ref{lemma-b-i-j}. Let us take $a=a^*+\eps\cdot (j_1 z_1+\dots+j_d z_d)$.

Then 
\begin{equation}\label{eq-a-dist}
\Vert a-a^*\Vert=\eps \cdot \Vert j_1 z_1+\dots+j_d z_d\Vert \leq \eps\cdot \big(j_1\cdot  \Vert z_1\Vert+\dots+j_d\cdot \Vert z_d\Vert \big)\leq \eps\cdot m\cdot \big(\Vert z_1\Vert + \dots + \Vert z_d\Vert \big)<\delta.
\end{equation}
In particular, the choice of $\delta$ implies that $a\in U$.

Let us now check that $a$ satisfies the conditions (i) to (iv) in Lemma \ref{lemma-b-i-j}. So let us fix some $i\in \lbrace 1,\dots,d\rbrace$ and $j\in \lbrace 1,\dots,m\rbrace$. Recall from (\ref{eq-b-dist}) and (\ref{eq-a-dist}) that $\Vert b_i^j-b_i\Vert<\delta$ and $\Vert a-a^*\Vert<\delta$. Hence, by the choice of $\delta$, all $s\in \lbrace 1,\dots,k\rbrace\sm \lbrace s_i\rbrace$ satisfy $P_s(a,b_i^j)\neq 0$ and $\sgn P_{s}(a,b_i^j)=\sgn P_{s}(a^*,b_i)$. This establishes condition (iv) and it also establishes condition (i) except for $s=s_i$.

In order to check conditions (ii) and (iii), let us now investigate $P_{s_i}(a,b_i^j)$. Recall that by (\ref{eq-b-dist}) and (\ref{eq-a-dist}), we have $\Vert b_i^j-b_i\Vert<\delta$ and $\Vert a-a^*\Vert<\delta$, and therefore by the choice of $\delta$ and $C$
\[\left\vert P_{s_i}(a,b_i^j)-P_{s_i}(a^*,b_i)-\nabla P_{s_i}(a^*,b_i)\cdot
\begin{pmatrix}
a-a^*\\
b_i^j-b_i
\end{pmatrix}\right\vert
\leq C\cdot \big(\Vert a-a^*\Vert^2+\Vert b_i^j-b_i\Vert^2\big).\]
Since $(a^*,b_i)$ is a general wall pair with wall index $s_i$, we have $P_{s_i}(a^*,b_i)=0$ and the previous inequality simplifies to
\[\left\vert P_{s_i}(a,b_i^j)-\nabla P_{s_i}(a^*,b_i)\cdot
\begin{pmatrix}
a-a^*\\
b_i^j-b_i
\end{pmatrix}\right\vert
\leq C\cdot \big(\Vert a-a^*\Vert^2+\Vert b_i^j-b_i\Vert^2\big).\]
Again using (\ref{eq-b-dist}) and (\ref{eq-a-dist}), we obtain
\begin{multline}\label{eq-taylor}
\left\vert P_{s_i}(a,b_i^j)-\nabla P_{s_i}(a^*,b_i)\cdot
\begin{pmatrix}
a-a^*\\
b_i^j-b_i
\end{pmatrix}\right\vert
\leq C\cdot \big(\eps^2\cdot m^2\cdot \big(\Vert z_1\Vert + \dots + \Vert z_d\Vert \big)^2+\eps^2\cdot m^2\cdot \Vert v_i\Vert^2\big)\\
\leq \eps^2\cdot C\cdot m^2\cdot \big(\Vert z_1\Vert + \dots + \Vert z_d\Vert+\Vert v_1\Vert + \dots + \Vert v_d\Vert \big)^2<\frac{1}{2}\cdot \eps,
\end{multline}
where in the last step we used the third property from our choice of $\eps$.

On the other hand, by the choices of $a$ and $b_i^j$ as well as $z_1,\dots,z_d$ and $v_i$, we have
\begin{multline*}
\nabla P_{s_i}(a^*,b_i)\cdot
\begin{pmatrix}
a-a^*\\
b_i^j-b_i
\end{pmatrix}
=\nabla_a P_{s_i}(a^*,b_i)\cdot (a-a^*)+\nabla_b P_{s_i}(a^*,b_i)\cdot (b_i^j-b_i)\\
=\eps\cdot \nabla_a P_{s_i}(a^*,b_i)\cdot (j_1 z_1+\dots+j_d  z_d)+\left(\frac{1}{2}-j\right)\cdot \eps\cdot \nabla_b P_{s_i}(a^*,b_i)\cdot v_i\\
=\eps\cdot j_i\cdot \nabla_a P_{s_i}(a^*,b_i)\cdot z_i+\left(\frac{1}{2}-j\right)\cdot \eps\cdot \nabla_b P_{s_i}(a^*,b_i)\cdot v_i=\eps\cdot j_i+\left(\frac{1}{2}-j\right)\cdot \eps=\left(j_i-j+\frac{1}{2}\right)\cdot \eps.
\end{multline*}
Thus, $\nabla P_{s_i}(a^*,b_i)\cdot
\begin{pmatrix}
a-a^*\\
b_i^j-b_i
\end{pmatrix}\geq \frac{1}{2}\cdot \eps$ if $j\leq j_i$ and $\nabla P_{s_i}(a^*,b_i)\cdot
\begin{pmatrix}
a-a^*\\
b_i^j-b_i
\end{pmatrix}\leq -\frac{1}{2}\cdot \eps$ if $j> j_i$. Together with (\ref{eq-taylor}), this implies that $P_{s_i}(a,b_i^j)>0$ if $j\leq j_i$ and $P_{s_i}(a,b_i^j)<0$ if $j> j_i$. Thus, conditions (ii) and (iii) are satisfied. Finally, note that conditions (ii) and (iii) trivially imply condition (i) for $s=s_i$.

Thus, all the conditions (i) to (iv) are satisfied and we finished the proof of Lemma \ref{lemma-b-i-j}.

\section{Preparations and outline for the proof of Lemma \ref{lemma-main}}

\subsection{Proof outline}
\label{subsect-proof-outline}

In Lemma \ref{lemma-main}, we need to prove that there is a point $a\in U$ with $L_a=\R^d$. In other words, we need to show that there is a point $a\in U$ such that the vectors $ \nabla_a P_s(a,b)$ for all $b\in U$ and $1\leq s\leq k$ for which $(a,b)$ is a general wall pair with wall index $s$ are spanning $\R^d$. In this subsection, we give an outline of the proof.

Let us assume for contradiction that for all $a\in U$ we have $L_a\subsetneq \R^d$. Then for each $a\in U$ we can find a non-zero vector $w(a)\in \R^d$ such that $w(a)$ is orthogonal to the subspace $L_a\subsetneq \R^d$ (which means that $\nabla_a P_s(a,b)\cdot w(a)=0$ whenever $(a,b)$ is a general wall pair with wall index $s$). Moreover, for $a$ in some open neighborhood of a suitably chosen point $a_0\in U$, we can in fact choose the vectors $w(a)$ in such a way that they form a smooth vector field (meaning that $w(a)$ varies smoothly as a function of $a$).

Having such a smooth vector field $w(a)$ for $a$ in an open neighborhood of $a_0$, we can apply the local existence theorem of integral curves for smooth vector fields (see for example \cite[Proposition 9.2]{lee}). By this theorem, there exists a curve $\tau:(-\eps,\eps)\to U$ for some small $\eps>0$ such that $\tau(0)=a_0$ and $\tau'(t)=w(\tau(t))$ for all $t\in (-\eps,\eps)$. Intuitively, the second condition means that at every point $\tau(t)$ of the curve $\tau$, the curve goes in the direction of the vector $w(\tau(t))$.

Using this condition, we can now show the following. For every point $b\in U$ such that $(a_0,b)$ is a general wall pair, there exists some $\eps_b>0$ such that for all $t\in (-\eps_b,\eps_b)$ the pair $(\tau(t),b)$ is also a general wall pair. In other words, if $(a_0,b)=(\tau(0),b)$ is a general wall pair , then $(\tau(t),b)$ is also a general wall pair for all $t$ sufficiently close to $0$. In fact, we can show something slightly stronger: If $(a_0,b)$ is a general wall pair with wall index $s$, then $P_s(a,b)=0$ and it turns out that all points $b'$ close to $b$ with $P_s(a_0,b')=0$ also satisfy that $(a_0,b')$ is a general wall pair with wall index $s$. We can then show that there exists some $\eps_b>0$ such that for all points $b'$ close to $b$ with $P_s(a_0,b')=0$ the pair $(\tau(t),b')$ is a general wall pair with wall index $s$ for all $t\in (-\eps_b,\eps_b)$. The proof of this statement uses the implicit function theorem (or more precisely, Lemma \ref{lemma-beta} stated below, which is proved via the implicit function theorem), as well as the conditions for our curve $\tau:(-\eps,\eps)\to U$.

Very roughly speaking, the statements in the previous paragraph can intuitively be interpreted as saying that for $t$ sufficiently close to $0$, for all points $b'\in U$ such that $(a_0,b')=(\tau(0),b')$ is a general wall pair we also have that $(\tau(t),b')$ is a general wall pair (in fact, this does not actually follow from the statements in the previous paragraph, because there $\eps_b$ depends on $b$). Let us for a moment assume that that for all $t$ sufficiently close to $0$ we have the even stronger property that for all $b'\in U$ the pair $(a_0,b')=(\tau(0),b')$ is a general wall pair if and only if $(\tau(t),b')$ is a general wall pair.

In this case, we can finish the proof in the following way by finding points $a=a_0$ and $a'=\tau(t)$ for some small $t$ contradicting our overarching assumption in Theorem \ref{theo-main}. Let us consider the collection $\mathcal{C}$ of connected components of the open set $\lbrace b'\in U \mid P_s(a_0,b')\neq 0\text{ for }1\leq s\leq k\rbrace$ (i.e.\ the set of points $b'\in U$ where none of the polynomials $P_s(a_0,\_)$ vanishes). By Fact \ref{fact-components-finite} this set has only finitely many connected components. Inside each connected component $C\in \mathcal{C}$ of this set, we fix a point $b_C\in C$. Note that then for all $C\in \mathcal{C}$ and all $b\in C$ we have $\sgn P_s(a_0,b)=\sgn P_s(a_0,b_C)$ for $s=1,\dots,k$ and therefore $\Phi(a_0,b)=\Phi(a_0,b_C)$. Furthermore, if $t$ is sufficiently close to $0$, we have $\sgn P_s(\tau(t),b_C)=\sgn P_s(\tau(0),b_C)=\sgn P_s(a_0,b_C)$ for all $s=1,\dots,k$ and all $C\in \mathcal{C}$, and therefore $\Phi(\tau(t),b_C)=\Phi(a_0,b_C)$ for all $C\in \mathcal{C}$.

Now we can choose $a_1=\tau(t)$ for some $t$ sufficiently close to $0$, such that $a_1\neq a_0$ and $\Phi(a_1,b_C)=\Phi(a_0,b_C)$ for all $C\in \mathcal{C}$, and such that for all $b'\in U$ the pair $(a_0,b')$ is a general wall pair if and only if $(a_1,b')$ is a general wall pair (this last property is due to our assumption made above). By the assumptions of Theorem \ref{theo-main} there needs to be a point $b\in U$ with $\Phi(a_0,b)\neq \Phi(a_1,b)$ and $P_s(a_0,b)\neq 0$ and $P_s(a_1,b)\neq 0$ for $s=1,\dots,k$. This point $b$ lies inside some connected component $C\in \mathcal{C}$, and there is a path inside $C$ between the points $b$ and $b_C$. It turns out that one can in fact choose a path between $b$ and $b_C$ inside $C$ such that for all points $b'$ on this path the pair $(a_1,b')$ is not special (one can show this using Fact \ref{fact-complement-variety} when being careful with the choice of our initial point $a_0\in U$). Note that furthermore for all points $b'$ on this path the pair $(a_0,b')$ is not a general wall pair (because we have $b'\in C$ and therefore $P_s(a_0,b')\neq 0$ for $s=1,\dots,k$), and hence the pair $(a_1,b')$ is not a general wall pair. This means that for all points $b'$ on the path between $b$ and $b_C$ the pair $(a_1,b')$ is not a wall pair. By Claim \ref{claim-path-no-wall} we can conclude that $\Phi(a_1,b)=\Phi(a_1,b_C)$. But all in all we now obtain $\Phi(a_0,b)=\Phi(a_0,b_C)=\Phi(a_1,b_C)=\Phi(a_1,b)$, which contradicts  $\Phi(a_0,b)\neq \Phi(a_1,b)$.

This contradiction completes the argument in the simplified setting where we assume that for all $t$ sufficiently close to $0$, for all $b'\in U$ the pair $(a_0,b')=(\tau(0),b')$ is a general wall pair if and only if $(\tau(t),b')$ is a general wall pair. In general, we need to be more careful  in the proof. Instead of just considering for which points $b$ the pair $(a_0,b)$ is a general wall pair, we consider the set of polynomials $Q\in \R[y_1,\dots,y_d]$ such that $Q$ appears as an irreducible factor of one of the polynomials $P_s(a_0,\_)$ for $1\leq s\leq k$ and such that $Q$ vanishes at some point $b\in U$ such that $(a_0,b)$ is a general wall pair. Roughly speaking, we define $\mathcal{Q}(a_0)$ to be the set of all such polynomials $Q$. For any point $a\in U$, we define $\mathcal{Q}(a)$ in an analogous way. All of these sets $\mathcal{Q}(a)$ for $a\in U$ are finite sets (here, one has to be slightly more careful, and to actually consider the equivalence classes of the respective polynomials $Q$ under scaling by real numbers, see Notation \ref{notation-equiv-class-poly}).

As mentioned above, one can show that for every $b\in U$ such that $(a_0,b)$ is a general wall pair with wall index $s$, there exists some $\eps_b>0$ such that for all points $b'$ close to $b$ with $P_s(a_0,b')=0$ the pair $(\tau(t),b')$ is a general wall pair with wall index $s$ for all $t\in (-\eps_b,\eps_b)$. It turns out that all of these points $b'$ close to $b$ lie in the zero-set of the same irreducible factor $Q$ of the polynomial $P_s(a_0,\_)$, and using this one can show that $Q$ lies in $\mathcal{Q}(\tau(t))$ for all $t\in (-\eps_b,\eps_b)$. Now one can conclude that $\mathcal{Q}(a_0)\su \mathcal{Q}(\tau(t))$ for all $t$ sufficiently close to $0$. If we choose our initial point $a_0\in U$ in such a way that the size of the set $\mathcal{Q}(a_0)$ is maximal, this implies that we actually have $\mathcal{Q}(a_0)=\mathcal{Q}(\tau(t))$ for all $t$ sufficiently close to $0$. Here we use in a crucial way that the sets $\mathcal{Q}(a)$ for $a\in U$ are finite. This is the main reason why for a point $a\in U$ we consider the set $\mathcal{Q}(a)$ of polynomials instead of just considering the set of all $b\in U$ for which $(a,b)$ is a general wall pair (as the set of such $b\in U$ is usually infinite). Once we have $\mathcal{Q}(a_0)=\mathcal{Q}(\tau(t))$ for all $t$ sufficiently close to $0$, we can finish the argument in a similar way as outlined above in the simplified setting, again finding points $a=a_0$ and $a'=\tau(t)$ for some small $t$ contradicting the overarching assumption in Theorem \ref{theo-main}.

The actual details of the proof of Lemma \ref{lemma-main} will be given in Section \ref{sect-proof-lemma-main}, after some preparations in the other two subsections of this section.

\subsection{An auxiliary lemma}

In this subsection, we will prove a lemma that will be used several times within the proof of Lemma \ref{lemma-main}. For any $i\in \lbrace 1,\dots,d\rbrace$, let $\pr_i:\R^d\to \R^{d-1}$ be the projection along the $i$-th coordinate direction.  This means that for all $x\in \R^d$ the point $\pr_i(x)$ is obtained from omitting the $i$-th coordinate of $x$.

The following lemma can, very roughly speaking, be summarized as follows. Let $(a,b)\in U\times U$ be a general wall pair with wall index $s\in \lbrace 1,\dots,k\rbrace$, and suppose that $i\in \lbrace 1,\dots,d\rbrace$ is such that the $i$-th coordinate of $\nabla_b P_s(a,b)$ is non-zero. Then for any $a'$ close to $a$ and any $b^*$ close to $\pr_i(b)$, we can find a point $\beta(a',b^*)$ such that $(a', \beta(a',b^*))$ is also a general wall pair with wall index $s$ and such that the projection $\pr_i(\beta(a',b^*))$ of $\beta(a',b^*)$ along the $i$-th coordinate direction equals $b^*$. In other words, for any $a'$ close to $a$ we can find a point $\beta(a',b^*)$ with any prescribed projection $b^*$ along the $i$-th coordinate direction such that $(a', \beta(a',b^*))$ is a general wall pair with wall index $s$ (assuming that the prescribed projection $b^*$ is close to the projection $\pr_i(b)$ of $b$). Moreover, we can choose these points $\beta(a',b^*)$ in such a way that $\beta$ is a smooth function in $a'$ and $b^*$. We prove this lemma by applying the implicit function theorem to the polynomial $P_s$.

\begin{lemma}\label{lemma-beta}Let $(a,b)\in U\times U$ be a general wall pair with wall index $s\in \lbrace 1,\dots,k\rbrace$. Let $i\in \lbrace 1,\dots,d\rbrace$ be such that the $i$-th coordinate of $\nabla_b P_s(a,b)$ is non-zero. Then there exist open subsets $U_a\su U$ and $V\su \R^{d-1}$ and a smooth function $\beta: U_a\times V\to U$ satisfying the following conditions:
\begin{itemize}
\item $a\in U_a$ and $\pr_i(b)\in V$.
\item $\beta(a,\pr_i(b))=b$.
\item For all $a'\in U_a$ and $b^*\in V$, we have $\pr_i(\beta(a',b^*))=b^*$.
\item For all $a'\in U_a$ and $b^*\in V$, the $i$-th coordinate of $\nabla_b P_s(a',\beta(a',b^*))$ is non-zero.
\item For all $a'\in U_a$ and $b^*\in V$, the pair $(a',\beta(a',b^*))\in U\times U$ is a general wall pair with wall index $s$.
\end{itemize}
\end{lemma}

For any general wall pair $(a,b)\in U\times U$ with wall index $s\in \lbrace 1,\dots,k\rbrace$, we have $P_s(a,b)=0$. Thus, as $(a,b)$ is not a special pair, we can conclude that $\nabla_b P_s(a,b)\neq 0$. Hence, we can always find $i\in \lbrace 1,\dots,d\rbrace$ satisfying the assumption in Lemma \ref{lemma-beta}.

\begin{proof}[Proof of Lemma \ref{lemma-beta}]
We may assume without loss of generality that $i=d$, since otherwise we can temporarily reorder the coordinates. Let $b_d$ be the last coordinate of $b$, then $b$ can be obtained from $\pr_d(b)$ by re-attaching $b_d$ at the end. Now, let us consider the polynomial $P_s$ as a smooth function $P_s: \R^d\times \R^{d-1}\times \R\to \R$. Note that at the point $(a,\pr_d(b),b_d)=(a,b)$ the function $P_s$ has value $P_s(a,b)=0$. Furthermore, the partial derivative of $P_s$ with respect to the last coordinate evaluated at the point $(a,b)$ is the last coordinate of the gradient $\nabla P_s(a,b)$. Since this is also the last coordinate of $\nabla_b P_s(a,b)$, by our assumption on $i=d$, this partial derivative is non-zero. Thus, by the implicit function theorem (see for example \cite[Theorem C.40]{lee}), there exist open neighborhoods $U_a\su U$ and $V\su \R^{d-1}$ of $a$ and $\pr_d(b)$, respectively, and a smooth function $\eta: U_a\times V\to \R$ with $\eta(a,\pr_d(b))=b_d$ such that for all $a'\in U_a$ and $b^*\in V$ we have $P_s(a',b^*,\eta(a',b^*))=0$.

Now, recall that $P_j(a,b)\neq 0$ for all $j\in \lbrace 1,\dots,k\rbrace\sm \lbrace s\rbrace$, as $(a,b)$ is a general wall pair. Thus, when considering $P_j$ as a continuous function $P_j: \R^d\times \R^{d-1}\times \R\to \R$, we have
\[P_j(a,\pr_d(b),\eta(a,\pr_d(b)))=P_j(a,\pr_d(b),b_d)=P_j(a,b)\neq 0\]
for all $j\in \lbrace 1,\dots,k\rbrace\sm \lbrace s\rbrace$. Thus, by making the open sets $U_a$ and $V$ smaller (such that we still have $a\in U_a$ and $\pr_d(b)\in V$), we may assume that
\begin{equation}\label{eq-sgn-aux-lemma}
P_j(a',b^*,\eta(a',b^*))\neq 0\quad \text{and}\quad \sgn P_j(a',b^*,\eta(a',b^*))=\sgn P_j(a,b)
\end{equation}
for all $a'\in U_a$, all $b^*\in V$ and all $j\in \lbrace 1,\dots,k\rbrace\sm \lbrace s\rbrace$. Furthermore, as $(\pr_d(b),\eta(a,\pr_d(b)))=(\pr_d(b),b_d)=b\in U$, we may also assume that $(b^*,\eta(a',b^*))\in U$ for all $a'\in U_a$ and $b^*\in V$. Finally, as the last coordinate of $\nabla_b P_s(a,\pr_d(b),\eta(a,\pr_d(b)))=\nabla_b P_s(a,b)$ is non-zero, we may assume that the last coordinate of $\nabla_b P_s(a',b^*,\eta(a',b^*))$ is non-zero for all $a'\in U_a$ and $b^*\in V$.

Now let us define $\beta: U_a\times V\to U$ by setting $\beta(a',b^*)=(b^*,\eta(a',b^*))$ for all $a'\in U_a$ and $b^*\in V$. Then it is clear that for all $a'\in U_a$ and $b^*\in V$, we have $\pr_d(\beta(a',b^*))=b^*$. Furthermore, we have $\beta(a,\pr_d(b))=(\pr_d(b),\eta(a,\pr_d(b)))=(\pr_d(b),b_d)=b$. For all $a'\in U_a$ and $b^*\in V$, the last coordinate of $\nabla_b P_s(a',\beta(a',b^*))=\nabla_b P_s(a',b^*,\eta(a',b^*))$ is non-zero, as desired (recall that $i=d$). As we already established the conditions $a\in U_a$ and $\pr_d(b)\in V$ above, it only remains to check the last condition. This means that we need to check that $(a',\beta(a',b^*))\in U\times U$ is a general wall pair with wall index $s$ for all $a'\in U_a$ and $b^*\in V$.

So let us fix $a'\in U_a$ and $b^*\in V$. By our choice of $\eta$, we have $P_s(a',\beta(a',b^*))=P_s(a',b^*,\eta(a',b^*))=0$. Furthermore, by (\ref{eq-sgn-aux-lemma}) we have $P_j(a',\beta(a',b^*))\neq 0$ and $\sgn P_j(a',\beta(a',b^*)=\sgn P_j(a,b)$ for all $j\in \lbrace 1,\dots,k\rbrace\sm \lbrace s\rbrace$. We also know that the last coordinate of $\nabla_b P_s(a', \beta(a',b^*))$ is non-zero, so in particular $\nabla_b P_s(a', \beta(a',b^*))\neq 0$. All in all, this establishes that $(a', \beta(a',b^*))$ is not a special pair.

Our next claim states that $(a', \beta(a',b^*))$ is a wall pair. This will establish that $(a', \beta(a',b^*))$ is a general wall pair and as $P_s(a',\beta(a',b^*))=0$ its wall index must be $s$. This will finish the proof of Lemma \ref{lemma-beta}.

\begin{claim} $(a', \beta(a',b^*))$ is a wall pair.
\end{claim}
\begin{proof}
Let $U_b\su U$ be any open neighborhood of $\beta(a',b^*)$. Recall that $P_s(a', \beta(a',b^*))=0$ and the last coordinate of $\nabla_b P_s(a', \beta(a',b^*))$ is non-zero. So the partial derivative of $P_s$ with respect the $d$-th coordinate direction of $\beta(a',b^*)$ evaluated at the point $(a', \beta(a',b^*))$ is non-zero. Thus, close to $\beta(a',b^*)$ we can find points $b^{+},b^{-}\in U_b$ such that $P_s(a',b^{+})>0$ and $P_s(a',b^{-})<0$. As $P_j(a',\beta(a',b^*))\neq 0$ and $\sgn P_j(a',\beta(a',b^*)=\sgn P_j(a,b)$ for all $j\in \lbrace 1,\dots,k\rbrace\sm \lbrace s\rbrace$, by choosing $b^{+},b^{-}\in U_b$ sufficiently close to $\beta(a',b^*)$, we can also ensure that $P_j(a',b^{+})\neq 0$ and $P_j(a',b^{-})\neq 0$ as well as $\sgn P_j(a',b^{+})=\sgn P_j(a',b^{-})=\sgn P_j(a',\beta(a',b^*)=\sgn P_j(a,b)$ for all $j\in \lbrace 1,\dots,k\rbrace\sm \lbrace s\rbrace$. Then we have $P_j(a',b^{+})\neq 0$ and $P_j(a',b^{-})\neq 0$ for all $j\in \lbrace 1,\dots,k\rbrace$. Furthermore
\[\Phi(a',b^{+})=\phi\big(\sgn P_1(a,b),\dots ,\sgn P_{s-1}(a,b),+,\sgn P_{s+1}(a,b),\dots ,\sgn P_{k}(a,b)\big)\]
and
\[\Phi(a',b^{-})=\phi\big(\sgn P_1(a,b),\dots ,\sgn P_{s-1}(a,b),-,\sgn P_{s+1}(a,b),\dots ,\sgn P_{k}(a,b)\big).\]
Hence, by Claim \ref{claim-general-wall-sign-flip} we obtain $\Phi(a',b^{+})\neq \Phi(a',b^{-})$. Thus, we have shown that in any open neighborhood $U_b\su U$ of $\beta(a',b^*)$ we can find points $b^{+}, b^{-}\in U_b$ with $P_j(a',b^{+})\neq 0$ and $P_j(a',b^{-})\neq 0$ for all $j\in \lbrace 1,\dots,k\rbrace$ and such that the values $\Phi(a',b^{+})$ and $\Phi(a',b^{-})$ are different from each other. Therefore we cannot have $(a', \beta(a',b^*))\in T_\lambda$ for any $\lambda\in \Lambda$. Hence $(a', \beta(a',b^*))$ is indeed a wall pair.
\end{proof}
We already saw above that this finishes the proof of Lemma \ref{lemma-beta}.
\end{proof}

\subsection{Further preparations}

This subsection contains more preparations for the proof of Lemma \ref{lemma-main}. In particular, for each general wall pair $(a,b)\in U\times U$ we will define an associated equivalence class $[Q]$ of polynomials under the equivalence relation introduced in Notation \ref{notation-equiv-class-poly}. These associated polynomial classes will be used to define the sets $\mathcal{Q}(a)$ mentioned in the outline for the proof of Lemma \ref{lemma-main} in Subsection \ref{subsect-proof-outline}.

\begin{claim} Let $(a,b)\in U\times U$ be a general wall pair with wall index $s\in \lbrace 1,\dots,k\rbrace$. Then the $d$-variable polynomial $P_s(a,\_)$ has exactly one irreducible factor $Q$ such that $Q(b)=0$ (more precisely, $Q$ is unique up to multiplication by a real number).
\end{claim}
\begin{proof}
Recall that $P_s(a,b)=0$ as $s$ is the wall index of the general wall pair $(a,b)$. Thus, the polynomial $P_s(a,\_)$ has value $0$ at the point $b$. Since the pair $(a,b)$ is not special, by Definition \ref{defi-special} we have $\nabla_b P_s(a,b)\neq 0$. In other words, the gradient vector of $P_s(a,\_)$ at the point $b$ is non-zero. Thus, $P_s(a,\_)$ cannot have two different irreducible factors vanishing at the point $b$ and furthermore $P_s(a,\_)$ cannot be the zero polynomial. Hence, as $P_s(a,b)=0$, the polynomial $P_s(a,\_)$ has exactly one irreducible factor vanishing at $b$.
\end{proof}

For a general wall pair $(a,b)\in U\times U$ with wall index $s\in \lbrace 1,\dots,k\rbrace$, let us define the \emph{associated polynomial class} of $(a,b)$ to be the equivalence class $[Q]$ of the unique irreducible factor $Q$ of $P_s(a,\_)$ such that $Q(b)=0$. Recall that $P_s\in \R[x_1,\dots,x_d,y_1,\dots,y_d]$ and note that therefore $Q\in \R[y_1,\dots,y_d]$. Also note that it is reasonable to consider the class $[Q]$ instead of the polynomial $Q$ itself, since all polynomials $R$ with $R\equiv Q$ can be equally well considered as irreducible factors of $P_s(a,\_)$ and satisfy $R(b)=0$.

\begin{definition}\label{defi-Qa} For a point $a\in U$, let $\mathcal{Q}(a)$ be the collection of all $[Q]$ occurring as an associated polynomial class of a general wall pair of the form $(a,b)$ for some $b\in U$.
\end{definition}

By definition, for each $a\in U$, all $Q$ with $[Q]\in\mathcal{Q}(a)$ are irreducible factors of $P_s(a,\_)$ for some $s\in \lbrace 1,\dots,k\rbrace$. Thus,
\begin{equation}\label{eq-Q-a-size}
\vert \mathcal{Q}(a)\vert\leq \sum_{[Q]\in  \mathcal{Q}(a)}\deg Q\leq \sum_{s=1}^{k}\deg P_s(a,\_)\leq \sum_{s=1}^{k}\deg P_s
\end{equation}
for all $a\in U$.

\begin{lemma}\label{lemma-Q-neighborhood} Let $(a,b)\in U\times U$ be a general wall pair with associated polynomial class $[Q]$. Then there exists an open neighborhood $U_b\su U$ of $b$ such that the following holds: For all points $b'\in U_b$ for which $(a,b')$ is a general wall pair, the associated polynomial class of $(a,b')$ is $[Q]$.
\end{lemma}
\begin{proof} Let $s\in \lbrace 1,\dots,k\rbrace$ be the wall index of $(a,b)$. Then $P_j(a,b)\neq 0$ for all $j\in \lbrace 1,\dots,k\rbrace\sm \lbrace s\rbrace$. Furthermore $R(b)\neq 0$ for all irreducible factors $R$ of $P_s(a,\_)$ with $[R]\neq [Q]$. Thus, we can choose an open neighborhood $U_b\su U$ of $b$ such that for all $b'\in U_b$ we have $P_j(a,b')\neq 0$ for all $j\in \lbrace 1,\dots,k\rbrace\sm \lbrace s\rbrace$ and $R(b')\neq 0$ for all irreducible factors $R$ of $P_s(a,\_)$ with $[R]\neq [Q]$. 

Now we claim that for each $b'\in U_b$ such that $(a,b')$ is a general wall pair, the associated polynomial class of $(a,b')$ is $[Q]$. As $P_j(a,b')\neq 0$ for all $j\in \lbrace 1,\dots,k\rbrace\sm \lbrace s\rbrace$, the wall index of $(a,b')$ must be $s$. Since $R(b')\neq 0$ for all irreducible factors $R$ of $P_s(a,\_)$ with $[R]\neq [Q]$, no irreducible factor of $P_s(a,\_)$ other than $Q$ can give rise to the associated polynomial class of $(a,b')$. Thus, the associated polynomial class of $(a,b')$ must indeed be $[Q]$.
\end{proof}

Finally, we need one more lemma that will be used in the proof of Lemma \ref{lemma-main}. Essentially, this lemma states that for most points $a\in U$, the set of points $b\in \R^d$ such that $(a,b)$ is a special pair is a real algebraic set of dimension at most $d-2$. For all such points $a$, by Fact \ref{fact-complement-variety} the set of points $b\in \R^d$ such that $(a,b)$ is not special is path-connected. We will deduce Lemma \ref{lemma-U1} from Fact \ref{fact-cutting-dimension}.

\begin{lemma}\label{lemma-U1}There exists a non-empty open subset $U_1\su U$ such that for all $a\in U_1$ the following holds: The set
\[\lbrace b\in \R^d\mid (a,b)\text{ is a special pair}\rbrace\]
is a real algebraic set in $\R^d$ of dimension at most $d-2$.
\end{lemma}
\begin{proof}
Recall that in Definition \ref{defi-special} we defined a pair $(a,b)\in \R^d\times \R^d$ to be a special pair if we have $P_s(a,b)=0$ for at least two different indices $s\in \lbrace 1,\dots,k\rbrace$ or if exists an index $s\in \lbrace 1,\dots,k\rbrace$ with $P_s(a,b)=\partial_{b_1} P_s(a,b)=\dots=\partial_{b_d} P_s(a,b)=0$. 

It is possible for some of the polynomials $P_s(a,b)$ for $s\in \lbrace 1,\dots,k\rbrace$ to only depend on $a$, meaning that none of the variables $b_1,\dots,b_d$ occurs in $P_s(a,b)$. If this is the case for some index $s\in \lbrace 1,\dots,k\rbrace$, let us call this index \emph{$a$-dependent}. Otherwise, let us call $s\in \lbrace 1,\dots,k\rbrace$ \emph{non-$a$-dependent}.

Note that if $s\in \lbrace 1,\dots,k\rbrace$ is $a$-dependent, we have $P_s(a,b)\neq 0$ for all $a\in U$ and $b\in \R^d$. Indeed, suppose that $P_s(a,b)=0$ for some $a\in U$ and $b\in \R^d$, then we would have $P_s(a,b)=0$ for this particular point $a\in U$ and all points $b\in \R^d$ (since the polynomial $P_s(a,b)$ is independent of $b$). But then let us fix any point $a'\in U$ with $a'\neq a$. By the assumption of Theorem \ref{theo-main} there needs to be a point $b\in U\su \R^d$ satisfying $P_s(a,b)\neq 0$ (and many other conditions), which is a contradiction. Thus, we indeed have $P_s(a,b)\neq 0$ for all $a\in U$, $b\in \R^d$ and all $a$-dependent indices $s\in \lbrace 1,\dots,k\rbrace$.

Now, let the set $V\su \R^d\times \R^d=\R^{2d}$ be the union of the sets
\begin{equation}\label{eq-set-special-pair1}
\lbrace (a,b)\in \R^d\times \R^d\mid P_s(a,b)=P_{s'}(a,b)=0\rbrace
\end{equation}
for any two distinct indices $s,s'\in \lbrace 1,\dots,k\rbrace$ and the sets
\begin{equation}\label{eq-set-special-pair2}
\lbrace (a,b)\in \R^d\times \R^d\mid P_s(a,b)=\partial_{b_1} P_s(a,b)=\dots=\partial_{b_d} P_s(a,b)=0\rbrace
\end{equation}
for all non-$a$-dependent indices $s\in \lbrace 1,\dots,k\rbrace$.

Note that for any pair $(a,b)\in U\times \R^d$, we have $(a,b)\in V$ if and only if $(a,b)$ is a special pair (since by the argument above we cannot have $P_s(a,b)=0$ for any $a$-dependent index $s$).

We claim that $V$ is a real algebraic set of dimension at most $2d-2$. First, by Fact \ref{fact-two-coprime-poly}, for any two distinct indices $s,s'\in \lbrace 1,\dots,k\rbrace$ the set in (\ref{eq-set-special-pair1}) is a real algebraic set of dimension at most $2d-2$ (here we used that the polynomials $P_s$ are irreducible and mutually coprime). Furthermore, for every non-$a$-dependent index $s\in \lbrace 1,\dots,k\rbrace$, the polynomial $P_s(a,b)$ contains at least one of the variables $b_1,\dots,b_d$. This means that at least one of the partial derivatives $\partial_{b_1} P_s, \dots, \partial_{b_d} P_s(a,b)$ is a non-zero polynomial. Since all these partial derivatives have degree at most $\deg P_s-1$, this implies that at least one of them is not divisible by the polynomial $P_s$. Thus, again by Fact \ref{fact-two-coprime-poly} (again using that $P_s$ is irreducible), for each non-$a$-dependent index $s\in \lbrace 1,\dots,k\rbrace$ the set in (\ref{eq-set-special-pair2}) is a real algebraic set of dimension at most $2d-2$. All in all, using Fact \ref{fact-dim-union}, this implies that $V$ is a real algebraic set of dimension at most $2d-2$.

Now, by Fact \ref{fact-cutting-dimension}, there exists a dense open set $U_1'\su \R^d$ such that for each $a\in U_1'$ the set $\lbrace b\in \R^d\mid (a,b)\in V\rbrace$ is a real algebraic set in $\R^d$ of dimension at most $d-2$.

Choosing $U_1=U_1'\cap U$, the set $U_1$ is a non-empty open subset of $U$ (since $U_1'$ is open and dense). It remains to check that for every $a\in U_1$ the set $\lbrace b\in \R^d\mid (a,b)\text{ is a special pair}\rbrace$ is a real algebraic set in $\R^d$ of dimension at most $d-2$.

So let us fix some $a\in U_1\su U$. We saw above that then for all $b\in \R^d$ we have $(a,b)\in V$ if and only if $(a,b)$ is a special pair. Hence the set $\lbrace b\in \R^d\mid (a,b)\text{ is a special pair}\rbrace$ is the same as the set $\lbrace b\in \R^d\mid (a,b)\in V\rbrace$, which we already know to be a real algebraic set of dimension at most $d-2$ (since $a\in U_1'$).
\end{proof}

\section{Proof of Lemma \ref{lemma-main}}\label{sect-proof-lemma-main}

Now we are finally ready for the proof of Lemma \ref{lemma-main}.

Recall that in Definition \ref{defi-L-a}, for each $a\in U$ we defined a linear subspace $L_a\su \R^d$. Our goal is to prove Lemma \ref{lemma-main}, which states that there is some $a\in U$ with $L_a=\R^d$. So let us assume the contrary, then $\dim_{\R}L_a\leq d-1$ for all $a\in U$.

Let $U_1\su U$ be an open subset as in Lemma \ref{lemma-U1}.

Now, let us consider the set of those $a\in U_1\su U$ for which $\dim_\R L_a$ is maximal. Among all those points, let us fix some $a_0\in U_1$ for which $\vert \mathcal{Q}(a_0)\vert$ is maximal (recall that by (\ref{eq-Q-a-size}) we have $\vert \mathcal{Q}(a)\vert\leq \sum_{s=1}^{k}\deg P_s$ for all $a\in U_1\su U$). Note that by the assumption of our proof by contradiction we have $\dim_{\R}L_{a_0}\leq d-1$.

\begin{lemma}\label{lemma-vector-field}There exists an open neighborhood $U_2\su U_1\su U$ of $a_0$ such that for all $a\in U_2$ we have $\dim_\R L_a=\dim_\R L_{a_0}$, and such that we can find a smooth vector field $w: U_2\to \R^d$ such that for all $a\in U_2$ we have $w(a)\neq 0$ and $v\cdot w(a)=0$ for all $v\in L_a$.
\end{lemma}

\begin{proof}
Let $\l=\dim_\R L_{a_0}\leq d-1$. Then, by the choice of $a_0$, we have $\dim_\R L_{a}\leq \l$ for all $a\in U_1$.

Let $b_1,\dots,b_{\l}\in U$ and $s_1,\dots,s_\l\in \lbrace 1,\dots,k\rbrace$ be such that $(a_0, b_j)$ is a general wall pair with wall index $s_j$ for each $j=1,\dots,\l$ and such that $L_a= \operatorname{span} \lbrace \nabla_a P_{s_1}(a_0,b_1),\dots, \nabla_a P_{s_\l}(a_0,b_\l)\rbrace$. Then the vectors $\nabla_a P_{s_1}(a_0,b_1),\dots, \nabla_a P_{s_\l}(a_0,b_\l)$ are linearly independent.

For each $j=1,\dots,\l$, let us now apply Lemma \ref{lemma-beta} to the general wall pair $(a_0, b_j)$. As $(a_0,b_j)$ is not a special pair, there is an index $i_j\in \lbrace 1,\dots,d\rbrace$ such that the $i_j$-th coordinate of $\nabla_b P_{s_j}(a_0,b_j)$ is non-zero. Then, by Lemma \ref{lemma-beta}, we can find open subsets $U_a^j\su U$ and $V_j\su \R^{d-1}$ with $a_0\in U_a^j$ and $\pr_{i_j}(b_j)\in V_j$ as well as a smooth function $\beta_j: U_a^j\times V_j\to U$ such that $\beta_j(a_0,\pr_{i_j}(b_j))=b_j$ and such that for all $a'\in U_a^j$ and $b^*\in V_j$, the pair $(a',\beta_j(a',b^*))\in U\times U$ is a general wall pair with wall index $s_j$. Let us define a smooth function $\alpha_j: U_a^j\to U$ by setting $\alpha_j(a')=\beta_j(a',\pr_{i_j}(b_j))$ for all $a'\in U_a^j$. Then $(a',\alpha_j(a'))\in U\times U$ is a general wall pair with wall index $s_j$ for each $a'\in U_a^j$, and furthermore $\alpha_j(a_0)=\beta_j(a_0,\pr_{i_j}(b_j))=b_j$.

Now, let us consider the open set $U_a=U_a^1\cap\dots\cap U_a^\l\cap U_1$. Note that $a_0\in U_a\su U_1$ and that all the functions $\alpha_1,\dots,\alpha_\l$ are defined on $U_a$. For each $a\in U_a$, let us consider the $(\l\times d)$-matrix $A(a)$ with rows $\nabla_a P_{s_1}(a,\alpha_1(a)),\dots, \nabla_a P_{s_\l}(a,\alpha_\l(a))$. All the coefficients of the matrix $A(a)$ are smooth functions of $a\in U_a$. Furthermore, the matrix $A(a_0)$ has rows $\nabla_a P_{s_1}(a_0,b_1),\dots, \nabla_a P_{s_\l}(a_0,b_\l)$. Since we saw above that these $\l$ vectors are linearly independent, the matrix $A(a_0)$ has rank $\l$.

So, using $\l\leq d-1$, we can apply Lemma \ref{lemma-matrix} and find an open neighborhood $U_2\su U_a\su U_1$ of $a_0$ such that the matrix $A(a)$ has rank $\l$ for each $a\in U_2$. And furthermore we can choose $U_2$ in such a way that there exists a smooth vector field $w:U_2\to \R^d$ such that for all $a\in U_2$ we have $w(a)\neq 0$ and $A(a)w(a)=0$.

We will now prove that for each $a\in U_2$, we have $\dim_\R L_a=\dim_\R L_{a_0}$ as well as $v\cdot w(a)=0$ for all $v\in L_a$. This will show that $U_2$ and $w$ satisfy all the desired conditions and will therefore finish the proof of Lemma \ref{lemma-vector-field}.

So fix some $a\in U_2\su U_a\su U_1$. First recall that $(a,\alpha_j(a))\in U\times U$ is a general wall pair with wall index $s_j$ for each $j=1,\dots,\l$. Thus, by Definition \ref{defi-L-a} we have $\nabla_a P_{s_j}(a,\alpha_j(a))\in L_a$ for $j=1,\dots,\l$. Since the matrix $A(a)$ with rows $\nabla_a P_{s_1}(a,\alpha_1(a)),\dots, \nabla_a P_{s_\l}(a,\alpha_\l(a))$ has rank $\l$, the vectors $\nabla_a P_{s_1}(a,\alpha_1(a)),\dots, \nabla_a P_{s_\l}(a,\alpha_\l(a))$ are linearly independent. Thus, we must have $\dim_\R L_a\geq \l$. As we already saw that $\dim_\R L_a\leq \l$ at the beginning of the proof of Lemma \ref{lemma-vector-field}, we can conclude that $\dim_\R L_a= \l=\dim_\R L_{a_0}$, as desired. Furthermore we can conclude that $L_a$ is spanned by the vectors $\nabla_a P_{s_1}(a,\alpha_1(a)),\dots, \nabla_a P_{s_\l}(a,\alpha_\l(a))$.

Recall that the vector field $w$ was chosen such that $A(a)w(a)=0$. Because the matrix $A(a)$ has rows $\nabla_a P_{s_1}(a,\alpha_1(a)),\dots, \nabla_a P_{s_\l}(a,\alpha_\l(a))$, this means that $\nabla_a P_{s_j}(a,\alpha_j(a))\cdot w(a)=0$ for $j=1,\dots, \l$. Since the vectors $\nabla_a P_{s_j}(a,\alpha_j(a))$ for $j=1,\dots,\l$ span $L_a$, we can conclude that $v\cdot w(a)=0$ for all $v\in L_a$. This finishes the proof of Lemma \ref{lemma-vector-field}.\end{proof}

Let us fix an open neighborhood $U_2\su U_1$ of $a_0$ and a vector field $w: U_2\to \R^d$ as in Lemma \ref{lemma-vector-field}. By the local existence of integral curves (see for example \cite[Proposition 9.2]{lee}), there exist $\eps>0$ and a smooth curve $\tau:(-\eps,\eps)\to U_2$ such $\tau(0)=a_0$ and the derivative $\tau'$ of $\tau$ satisfies $\tau'(t)=w(\tau(t))$ for all $t\in (-\eps,\eps)$.

A major step towards proving Lemma \ref{lemma-main} will be to establish the following lemma.

\begin{lemma}\label{lemma-Qa} For every $[Q]\in \mathcal{Q}(a_0)$, there exists some $\eps_{[Q]}$ with $0<\eps_{[Q]}<\eps$ such that $[Q]\in \mathcal{Q}(\tau(t))$ for all $t\in (-\eps_{[Q]},\eps_{[Q]})$.
\end{lemma}
\begin{proof} Let us fix some $[Q]\in \mathcal{Q}(a_0)$. By Definition \ref{defi-Qa}, there exists some $b\in U$ such that $(a_0,b)$ is a general wall pair with associated polynomial class $[Q]$.

Let $s\in \lbrace 1,\dots,k\rbrace$ be the wall index of the general wall pair $(a_0,b)$. Then $P_s(a_0,b)=0$ and, as $(a,b)$ is not special, we have $\nabla_b P_s(a_0,b)\neq 0$. Thus, there is some $i\in \lbrace 1,\dots,d\rbrace$ such that the $i$-th coordinate of $\nabla_b P_s(a_0,b)$ is non-zero. By applying Lemma \ref{lemma-beta}, we can find open sets $U_a\su U$ and $V\su \R^{d-1}$ and a smooth function $\beta: U_a\times V\to U$ satisfying the following conditions:
\begin{itemize}
\item[(I)] $a_0\in U_a$ and $\pr_i(b)\in V$.
\item[(II)] $\beta(a_0,\pr_i(b))=b$.
\item[(III)] For all $a'\in U_a$ and $b^*\in V$, we have $\pr_i(\beta(a',b^*))=b^*$.
\item[(IV)] For all $a'\in U_a$ and $b^*\in V$, the $i$-th coordinate of $\nabla_b P_s(a',\beta(a',b^*))$ is non-zero.
\item[(V)] For all $a'\in U_a$ and $b^*\in V$, the pair $(a',\beta(a',b^*))\in U\times U$ is a general wall pair with wall index $s$.
\end{itemize}
By continuity of $\tau$, we can find $\eps_ {[Q]}$ with $0<\eps_{[Q]}<\eps$ such that $\tau(t)\in U_a\cap U_2$ for all $t\in (-\eps_{[Q]},\eps_{[Q]})$.

\begin{claim}\label{claim-beta-constant}We have $\beta(\tau(t),b^*)=\beta(a_0,b^*)$ for all $t\in (-\eps_{[Q]},\eps_{[Q]})$ and all $b^*\in V$.
\end{claim}
\begin{proof}Let us fix some $b^*\in V$. Let us consider the smooth  curve $\rho: (-\eps_{[Q]},\eps_{[Q]})\to U$ given by $\rho(t)=\beta(\tau(t),b^*)$ for all $t\in (-\eps_{[Q]},\eps_{[Q]})$. Note that for all $t\in (-\eps_{[Q]},\eps_{[Q]})$ we have $P_s(\tau(t),\rho(t))=0$, since $(\tau(t),\rho(t))=(\tau(t),\beta(\tau(t),b^*))$ is a general wall pair with wall index $s$. Taking the derivative with respect to $t$, we obtain
\begin{multline}\label{eq-chain-rule}
0=\nabla P_s(\tau(t),\rho(t))\cdot
\begin{pmatrix}
\tau'(t)\\
\rho'(t)
\end{pmatrix}
=\begin{pmatrix}
\nabla_a P_s(\tau(t),\rho(t))\\
\nabla_b P_s(\tau(t),\rho(t))
\end{pmatrix}
\cdot
\begin{pmatrix}
\tau'(t)\\
\rho'(t)
\end{pmatrix}\\
=\nabla_a P_s(\tau(t),\rho(t))\cdot \tau'(t)+\nabla_b P_s(\tau(t),\rho(t))\cdot \rho'(t)
\end{multline}
for all $t\in (-\eps_{[Q]},\eps_{[Q]})$.

We need to show that $\beta(\tau(t),b^*)=\beta(a_0,b^*)$ for all $t\in (-\eps_{[Q]},\eps_{[Q]})$. As $\beta(\tau(t),b^*)=\rho(t)$ and $\beta(a_0,b^*)=\beta(\tau(0),b^*)=\rho(0)$, this is equivalent to proving $\rho(t)=\rho(0)$ for all $t\in (-\eps_{[Q]},\eps_{[Q]})$. Thus, it suffices to prove that $\rho'(t)=0$ for all $t\in (-\eps_{[Q]},\eps_{[Q]})$.

Note that for all $t\in (-\eps_{[Q]},\eps_{[Q]})$ we have $\pr_i(\rho(t))=\pr_i(\beta(\tau(t),b^*))=b^*$ by condition (III). Thus, the curve $\pr_i(\rho(t))$ runs along a line in the $i$-th coordinate direction. Hence, for all $t\in (-\eps_{[Q]},\eps_{[Q]})$, all the coordinates of $\rho'(t)$ are zero except possibly the $i$-th coordinate.

Let us fix some $t\in (-\eps_{[Q]},\eps_{[Q]})$ and recall that we wish to prove $\rho'(t)=0$. The pair $(\tau(t),\rho(t))$ is a general wall pair with wall index $s$ and therefore by Definition \ref{defi-L-a} we have $\nabla_a P_s(\tau(t),\rho(t))\in L_{\tau(t)}$. But then by the choice of the vector field $w$ as in Lemma \ref{lemma-vector-field}, we obtain
\[0=\nabla_a P_s(\tau(t),\rho(t))\cdot w(\tau(t))= \nabla_a P_s(\tau(t),\rho(t))\cdot \tau'(t),\]
where in the second step we used that $\tau'(t)=w(\tau(t))$ by the choice of the curve $\tau$. Thus, the first summand on the right-hand side of (\ref{eq-chain-rule}) is zero. We can conclude that
\begin{equation}\label{eq-inner-product}
\nabla_b P_s(\tau(t),\rho(t))\cdot \rho'(t)=0.
\end{equation}
Recall that all the coordinates of $\rho'(t)$ are zero except possibly the $i$-th coordinate. On the other hand, by condition (IV), the $i$-th coordinate of $\nabla_b P_s(\tau(t),\rho(t))=\nabla_b P_s(\tau(t),\beta(\tau(t),b^*))$ is non-zero. Thus the inner product (\ref{eq-inner-product}) being zero implies that the $i$-th coordinate of $\rho'(t)$ must be zero as well. Hence $\rho'(t)=0$ as desired. This finishes the proof of the Claim \ref{claim-beta-constant}.
\end{proof}

In order to prove Lemma \ref{lemma-Qa}, we need to show that $[Q]\in \mathcal{Q}(\tau(t))$ for all $t\in (-\eps_{[Q]},\eps_{[Q]})$. So let us fix some $t\in (-\eps_{[Q]},\eps_{[Q]})$, and set $a'=\tau(t)\in U_a\cap U_2$. We need to show that $[Q]\in \mathcal{Q}(a')$, which will finish the proof of Lemma \ref{lemma-Qa}. So let us suppose for contradiction that $[Q]\not\in \mathcal{Q}(a')$.

By Claim \ref{claim-beta-constant} we have $\beta(a',b^*)=\beta(\tau(t),b^*)=\beta(a_0,b^*)$ for all $b^*\in V$. So let us define the smooth function $\beta^*: V\to U$ by setting $\beta^*(b^*)=\beta(a',b^*)=\beta(a_0,b^*)$ for all $b^*\in V$. Now, $\beta^*(\pr_i(b))=\beta(a_0, \pr_i(b))=b$ by condition (II). Furthermore recall that by condition (V), the pair $(a_0,\beta^*(b^*))=(a_0,\beta(a_0,b^*))$ is a general wall pair for each $b^*\in V$. Similarly, the pair $(a',\beta^*(b^*))=(a',\beta(a',b^*))$ is also a general wall pair for each $b^*\in V$.

Recall that the associated polynomial class of the general wall pair $(a_0,\beta^*(\pr_i(b)))=(a_0,b)$ is $[Q]$. Hence, by Lemma \ref{lemma-Q-neighborhood} there exists an open neighborhood $U_b\su U$ of $\beta^*(\pr_i(b))$ such that for all points $b'\in U_b$ for which $(a_0,b')$ is a general wall pair, the associated polynomial class of $(a_0,b')$ is $[Q]$. By continuity of $\beta^*$, there is an open neighborhood $V_1\su V$ of $\pr_i(b)$ such that $\beta^*(b^*)\in U_b$ for all $b^*\in V_1$. Then for all $b^*\in V_1$, the associated polynomial class of the general wall pair $(a_0,\beta^*(b^*))$ is $[Q]$ and therefore we have $Q(\beta^*(b^*))=0$. So we have shown that there exists an open neighborhood $V_1\su V$ of $\pr_i(b)$ such that $Q(\beta^*(b^*))=0$ for all $b^*\in V_1$.

Let $[R]\in \mathcal{Q}(a')$ be the associated polynomial class of the general wall pair $(a',\beta^*(\pr_i(b)))=(a',b)$. By the same argument as in the previous paragraph we can show that there exists an open neighborhood $V_2\su V$ of $\pr_i(b)$ such that $R(\beta^*(b^*))=0$ for all $b^*\in V_2$.

Now, $V_1\cap V_2\su V$ is an open neighborhood of $\pr_i(b)$. So we can find some $\delta>0$ such that the entire open ball with radius $\delta>0$ around $\pr_i(b)$ in $\R^{d-1}$ is contained in $V_1\cap V_2$. Then for all $b^*\in \R^{d-1}$ with $\Vert b^*-\pr_i(b)\Vert<\delta$, we have $b^*\in V_1\cap V_2$ and therefore $Q(\beta^*(b^*))=R(\beta^*(b^*))=0$. Let $U'$ be the open ball with radius $\delta$ around $b$ in $\R^d$. Then for all $b'\in U'$ we have $\Vert \pr_i(b')-\pr_i(b)\Vert\leq \Vert b'-b\Vert<\delta$. Hence
\begin{equation}\label{eq-Q-R-zero}
Q(\beta^*(\pr_i(b')))=R(\beta^*(\pr_i(b')))=0
\end{equation}
for all $b'\in U'$.

Note that for every $b^*\in V$,  by condition (III) we have $\pr_i(\beta^*(b^*))=\pr_i(\beta(a_0,b^*))=b^*$. Thus, for all $b'\in U'$ we have $\pr_i(\beta^*(\pr_i(b')))=\pr_i(b')$. In other words, the vectors $\beta^*(\pr_i(b'))$ and $b'$ agree in all coordinates except possibly the $i$-th coordinate. We can define a smooth function $f:U'\to \R$ by defining $f(b')$ to be the $i$-th coordinate of the difference $\beta^*(\pr_i(b'))-b'$ for all $b'\in U'$. Then for each $b'\in U'$ we have $b'=\beta^*(\pr_i(b'))$ if and only if $f(b')=0$. In particular, by (\ref{eq-Q-R-zero}), we have
\begin{equation}\label{eq-Q-R-zero2}
Q(b')=R(b')=0
\end{equation}
for all $b'\in U'$ with $f(b')=0$.

\begin{claim}\label{claim-b-plus-minus}There exists a point $b^{+}\in U'$ with $f(b^{+})>0$ and $Q(b^{+})\neq 0$. Similarly, there exists a point $b^{-}\in U'$ with $f(b^{-})<0$ and $Q(b^{-})\neq 0$.
\end{claim}
\begin{proof}
First, let us show that there exists a point $b'\in U'$ with $f(b')>0$. Recall that $b\in U$ and $\beta^*(\pr_i(b))=b$. Thus $f(b)=0$. However, consider the point $b'\in U'$ obtained from $b$ by subtracting $\delta/2$ from the $i$-th coordinate of $b$. Then $\pr_i(b')=\pr_i(b)$ and hence $\beta^*(\pr_i(b'))=\beta^*(\pr_i(b))=b$. Thus, the $i$-th coordinate of $\beta^*(\pr_i(b'))-b'=b-b'$ is equal to $\delta/2$ and so $f(b')=\delta/2>0$.

Thus, the set $\lbrace b'\in U'\mid f(b')>0\rbrace$ is non-empty and by continuity of $f$ it is open. So by Fact \ref{fact-poly-vanishing} there exists some point $b^{+}$ in this non-empty open subset with $Q(b^{+})\neq 0$. Then $b^{+}\in U'$ and $f(b^{+})>0$ as desired.

The existence of a point $b^{-}\in U'$ with the desired properties can be proved analogously.
\end{proof}

Since $[R]\in \mathcal{Q}(a')$ and we assumed $[Q]\not\in \mathcal{Q}(a')$, we have $[Q]\neq [R]$. Both $Q$ and $R$ are irreducible polynomials, so this implies that $R$ is not divisible by $Q$. Now, by Fact \ref{fact-two-coprime-poly}, the set
\[Z=\lbrace b'\in \R^d\mid Q(b')=R(b')=0\rbrace\]
is a real algebraic set of dimension at most $d-2$. Therefore, by Fact \ref{fact-complement-variety}, the set $U'\sm Z$ is path-connected (recall that $U'$ was defined to be the open ball of radius $\delta$ around $b$ and is therefore open and connected). Now, consider points $b^{+}, b^{-}\in U'$ as in Claim \ref{claim-b-plus-minus}. As $Q(b^{+})\neq 0$ and $Q(b^{-})\neq 0$, we have $b^{+}, b^{-}\in U'\sm Z$. Hence there exists a continuous path inside $U'\sm Z$ connecting $b^{+}$ and $b^{-}$. But since $f(b^{+})>0$ and $f(b^{-})<0$, by the intermediate value theorem, this path would need to contain some point $b'\in U'\sm Z$ with $f(b')=0$. But by (\ref{eq-Q-R-zero2}), this point $b'$ would satisfy $Q(b')=R(b')=0$ and consequently $b'\in Z$. This is a contradiction, which finally finishes the proof of Lemma \ref{lemma-Qa}.
\end{proof}

So for every $[Q]\in \mathcal{Q}(a_0)$, let us fix some $\eps_{[Q]}$ with $0<\eps_{[Q]}<\eps$ as in Lemma \ref{lemma-Qa}.

Now, let $\mathcal{C}$ be the collection of connected components of the open set
\begin{equation}\label{eq-open-set-components-C}
\lbrace b'\in U \mid P_s(a_0,b')\neq 0\text{ for }1\leq s\leq k\rbrace.
\end{equation}
Note that this open set is definable by polynomials. Therefore, by Fact \ref{fact-components-finite}, the collection $\mathcal{C}$ of its connected components is finite.

For each $C\in \mathcal{C}$, let us fix a point $b_C\in C$. Then we have 
\begin{equation}\label{eq-P-a0-bC-nonzero}
P_s(\tau(0),b_C)=P_s(a_0,b_C)\neq 0\text{ for }1\leq s\leq k.
\end{equation}
Thus, there is some $\eps_C$ with $0<\eps_C<\eps$ with $P_s(\tau(t),b_C)\neq 0$ and $\sgn P_s(\tau(t),b_C)=\sgn P_s(\tau(0),b_C)=\sgn P_s(a_0,b_C)$ for all $t\in (-\eps_C,\eps_C)$ and $s=1,\dots,k$. In particular, $\Phi(\tau(t),b_C)=\Phi(a_0,b_C)$ for all $t\in (-\eps_C,\eps_C)$.

As both $\mathcal{C}$ and $\mathcal{Q}(a_0)$ are finite (the latter one by (\ref{eq-Q-a-size})), there exist some $\eps'>0$ with $\eps'<\eps_C$ for all $C\in \mathcal{C}$ and $\eps'<\eps_{[Q]}$ for all $[Q]\in \mathcal{Q}(a_0)$.

Because $\tau'(0)=w(\tau(0))=w(a_0)\neq 0$, the path $\tau(t)$ is not constant for all $t\in (-\eps',\eps')$. Thus, we can fix some $t\in (-\eps',\eps')$ with $\tau(t)\neq a_0$. Let us define $a_1=\tau(t)$, then $a_1\in U_2\su U_1$ and $a_1\neq a_0$.

Note that for all $C\in \mathcal{C}$, we have $\vert t\vert<\eps'<\eps_C$ and therefore
\begin{equation}\label{eq-P-a1-bC-nonzero}
P_s(a_1,b_C)=P_s(\tau(t),b_C)\neq 0\text{ for }1\leq s\leq k
\end{equation}
as well as
\begin{equation}\label{eq-a0-bC-a1}
\Phi(a_0,b_C)=\Phi(\tau(t),b_C)=\Phi(a_1,b_C).
\end{equation}

As $a_1\in U_2\su U_1$, by the choice of the set $U_2$ as in Lemma \ref{lemma-vector-field}, we have $\dim_\R L_{a_1}=\dim_\R L_{a_0}$. Thus, by our choice of $a_0$ we must have $\vert \mathcal{Q}(a_1)\vert\leq \vert \mathcal{Q}(a_0)\vert$. On the other hand, for every $[Q]\in \mathcal{Q}(a_0)$, we have $\vert t\vert<\eps_{[Q]}$ and therefore $[Q]\in \mathcal{Q}(\tau(t))=\mathcal{Q}(a_1)$. Thus, $\mathcal{Q}(a_0)\su \mathcal{Q}(a_1)$ and together with $\vert \mathcal{Q}(a_1)\vert\leq \vert \mathcal{Q}(a_0)\vert$, we obtain $\mathcal{Q}(a_1)=\mathcal{Q}(a_0)$.

Recall that $a_0,a_1\in U_1\su U$ and $a_0\neq a_1$. Thus, by the assumptions of Theorem \ref{theo-main}, there exists a point $b\in U$ with 
\begin{equation}\label{eq-Ps-b-nonzero}
P_s(a_0,b)\neq 0\text{ and }P_s(a_1,b)\neq 0\text{ for all }1\leq s\leq k
\end{equation}
and such that 
\begin{equation}\label{eq-choice-b}
\Phi(a_0,b)\neq \Phi(a_1,b).
\end{equation}

The point $b$ is contained in some connected component $C\in \mathcal{C}$ of the open set (\ref{eq-open-set-components-C}). Clearly, $C$ is a connected open subset of $\R^d$, and for all $b'\in C$ we have $P_s(a_0,b')\neq 0$ for all $1\leq s\leq k$.

\begin{claim}\label{claim-special-C}There is no point $b'\in C$ for which $(a_1,b')$ is a general wall pair.
\end{claim}
\begin{proof}
Suppose that for some $b'\in C$, the pair $(a_1,b')$ was a general wall pair. Then let $[Q]\in \mathcal{Q}(a_1)$ be the associated polynomial class of the general wall pair $(a_1,b')$. Note that $Q(b')=0$.

By $\mathcal{Q}(a_1)=\mathcal{Q}(a_0)$, the class $[Q]$ must also be the associated polynomial class of some general wall pair of the form $(a_0,b'')$ for some $b''\in U$. Then $Q$ is an irreducible factor of $P_s(a_0,\_)$ for some $s\in \lbrace 1,\dots,k\rbrace$. As $Q(b')=0$, this implies that $P_s(a_0,b')=0$ for some $s\in \lbrace 1,\dots,k\rbrace$. But then $b'\not\in C$, which is a contradiction.
\end{proof}

Recall that we fixed a point $b_C\in C$ earlier. By (\ref{eq-P-a1-bC-nonzero}), it follows straight from Definition \ref{defi-special} that the pair $(a_1,b_C)$ is not special. Similarly, by the second part of (\ref{eq-Ps-b-nonzero}), the pair $(a_1,b)$ is not special.

By $a_1\in U_2\su U_1$ and the choice of $U_1$ as in Lemma \ref{lemma-U1}, the set
\[Z=\lbrace b'\in \R^d\mid (a_1,b')\text{ is a special pair}\rbrace\]
is a real algebraic set in $\R^d$ of dimension at most $d-2$. So by Fact \ref{fact-complement-variety}, using that $C\in \mathcal{C}$ is a connected open set, the set $C\sm Z$ is path-connected. Since the pairs $(a_1,b)$ and $(a_1,b_C)$ are not special, we have $b, b_C\in C\sm Z$. Hence there exists a continuous path $\gamma:[0,1]\to C\sm Z$ with $\gamma(0)=b$ and $\gamma(1)=b_C$.

For all $r\in [0,1]$, we have $\gamma(r)\in C$. Therefore, by Claim \ref{claim-special-C}, the pair $(a_1,\gamma(r))$ is not a general wall pair. As $\gamma(r)\not\in Z$, the pair $(a_1,\gamma(r))$ is also not special. Hence, $(a_1,\gamma(r))$ cannot be a wall pair for any $r\in [0,1]$. Therefore, applying Claim \ref{claim-path-no-wall} to the path $r\mapsto (a_1,\gamma(r))$ in $U\times U$ (using (\ref{eq-Ps-b-nonzero}) and (\ref{eq-P-a1-bC-nonzero})), we obtain that
\begin{equation}\label{eq-Phi-a1}
\Phi(a_1,b)=\Phi(a_1,b_C).
\end{equation}

As $\gamma(r)\in C$ for all $r\in [0,1]$, we have $P_s(a_0,\gamma(r))\neq 0$ for all $1\leq s\leq k$ and all $r\in [0,1]$. Thus, by Claim \ref{claim-special-P-s-0}, the pair $(a_0,\gamma(r))$ cannot be a wall pair for any $r\in [0,1]$. Therefore, applying Claim \ref{claim-path-no-wall} to the path $r\mapsto (a_0,\gamma(r))$ in $U\times U$, we obtain that
\begin{equation}\label{eq-Phi-a0}
\Phi(a_0,b)=\Phi(a_0,b_C).
\end{equation}

Now, combining (\ref{eq-Phi-a0}), (\ref{eq-a0-bC-a1}) and (\ref{eq-Phi-a1}) yields
\[\Phi(a_0,b)=\Phi(a_0,b_C)=\Phi(a_1,b_C)=\Phi(a_1,b).\]
But this contradicts (\ref{eq-choice-b}). This contradiction finishes the proof of Lemma \ref{lemma-main}.

\textit{Acknowledgements.} The author would like to thank Jacob Fox for suggesting this project, for many very helpful discussions, and for several suggestions that improved the presentation of this paper. Furthermore, the author is grateful to Aaron Landesman for a useful conversation. Finally, the author would like to thank the anonymous referee for many helpful comments.

\appendix
\section{Appendix}

\subsection{Proof of Theorem \ref{theorem-upper-bound}}\label{sect-a-upper-bound}

This proof is identical with the proof of Theorem 3 in \cite{alon-scheinerman}, and is repeated here only for the reader's convenience. See also \cite[Section 2]{pach-solymosi} and \cite[Section 3]{mcdiarmid-mueller} for similar applications of the same method, and \cite[Section 6.2]{matousek} or \cite[Section 4.1]{spinrad} for an exposition.

Given polynomials $Q_1,\dots,Q_\l\in \R[x_1,\dots,x_m]$ of degree at most $D$, a sign pattern of the polynomials $Q_1,\dots, Q_{\l}$ is an element of $\lbrace +, -, 0\rbrace^\l$ of the form $(\sgn Q_1(x),\dots, \sgn Q_\l(x))$ for some $x\in \R^m$. By \cite[Theorem 2]{alon-scheinerman} due to Alon and Scheinerman, which is based on Warren's theorem \cite[Theorem 3]{warren}, for $\l\geq m$, the number of distinct sign-patterns of the polynomials $Q_1,\dots, Q_{\l}$ is at most $(8e\cdot D\cdot \l/m)^m\leq (24 \cdot D\cdot \l/m)^m$.

Now, let us fix polynomials $P_1,\dots,P_k\in \R[x_1,\dots,x_d,y_1,\dots,y_d]$, a function $\phi$, and an open subset $U\su \R^d$ as in Theorem \ref{theorem-upper-bound}. Recall that we want to prove that the number of $(P_1,\dots , P_k,\phi,U,\Lambda)$-representable edge-labelings of the complete graph on the vertex set $\lbrace 1,\dots, n\rbrace$ is at most $n^{(1+o(1))dn}$.

Each $(P_1,\dots,P_k,\phi,U,\Lambda)$-representable edge-labeling of the complete graph on the vertex set $\lbrace 1,\dots, n\rbrace$ is of the form $F_{P_1,\dots,P_k,\phi}(a_1,\dots,a_n)$ for some points $a_1,\dots,a_n\in U\su \R^d$. After choosing the points $a_1,\dots,a_n\in U$, the labels in the edge-labeling $F_{P_1,\dots,P_k,\phi}(a_1,\dots,a_n)$ can be determined from the signs of the $\binom{n}{2}\cdot k$ polynomials $P_s(a_i, a_j)$ for $1\leq i<j\leq n$ and $1\leq s\leq k$. These $\binom{n}{2}\cdot k$ polynomials can be interpreted as polynomials in the coordinates of $(a_1,\dots,a_n)\in\R^{dn}$. Thus, the number of $(P_1,\dots,P_k,\phi,U,\Lambda)$-representable edge-labelings of the complete graph on the vertex set $\lbrace 1,\dots, n\rbrace$ is at most the number of sign-patterns of those $\binom{n}{2}\cdot k$ polynomials in $dn$ variables. By the result \cite[Theorem 2]{alon-scheinerman} cited above applied to $\l=\binom{n}{2}\cdot k$ and $m=dn$, this number is at most
\[\left(\frac{24\cdot D\cdot \binom{n}{2}\cdot k}{dn}\right)^{dn}\leq \left(12\cdot D\cdot k\cdot n\right)^{dn}=n^{(1+o(1))dn},\]
where $D$ is the maximum of the degrees of the polynomials $P_1,\dots,P_k$. This finishes the proof of Theorem \ref{theorem-upper-bound}.

\subsection{Polynomial conditions for the linking of circles in $\R^3$}
\label{subsect-app-linking}

Recall that in Subsection \ref{subsect-linking} we defined $U=\lbrace (a,b,c,d,e,r)\in \R^6\mid r>0\rbrace$, and that every point $(a,b,c,d,e,r)\in U$ corresponds to a circle $C$ in $\R^3$. Here, we provide the details on how to check whether two circles $C$ and $C'$ corresponding to $(a,b,c,d,e,r), (a',b',c',d',e',r')\in U$ are linked using the signs of a finite list of polynomials in $a,b,c,d,e,r,a',b',c',d',e',r'$.

Recall that we observed that $C$ and $C'$ form a link if and only if there exists a point of $C$ which lies on the plane of $C'$ inside the circle $C'$ and another point of $C$ which lies on the plane of $C'$ outside the circle $C'$. Furthermore, by symmetry, the same holds with the roles of $C$ and $C'$ interchanged.

We start by noting that if $(d,e)=(d',e')$, then the planes of $C$ and $C'$ are parallel (since they are both orthogonal to $(d,e,1)=(d',e',1)$). But if this case, the circles $C$ and $C'$ cannot form a link. Hence let us now assume that $(d,e)\neq (d',e')$, so the planes of $C$ and $C'$ are not parallel and therefore intersect in a unique line $\l$.

The line $\l$ is contained in both of the planes of $C$ and $C'$ and therefore orthogonal to both $(d,e,1)$ and $(d',e',1)$. We can therefore compute the direction of the line $\l$ by taking the cross-product of the vectors $(d,e,1)$ and $(d',e',1)$. Thus, $(e-e', -(d-d'), de'-d'e)$ is a (non-zero) vector in the direction of the line $\l$.

Let $\l_C$ be the line inside the plane of $C$ which passes through the center $(a,b,c)$ of $C$ and is orthogonal to $\l$. Note that the intersection point $L$ of the lines $\l_C$ and $\l$ is the point of $\l$ with minimum distance to $(a,b,c)$ (namely, the foot of $(a,b,c)$ on $\l$).

The line $\l_C$ is orthogonal to $\l$ and to the vector $(d,e,1)$. Hence, the direction of the line $\l_C$ is the cross-product of the vectors $(e-e', -(d-d'), de'-d'e)$ and $(d,e,1)$, which is
\[\tau=\begin{pmatrix}
-(d-d')-e(de'-d'e)\\
-(e-e')+d(de'-d'e)\\
e(e-e')+d(d-d')
\end{pmatrix}.\]
The point $L$ is on the line $\l_C$, so it is of the form $(a,b,c)+t\cdot \tau$ for some $t\in \R$. We can now compute $L$ by solving for the unique $t\in \R$ such that $(a,b,c)+t\cdot \tau$ lies on the plane through $C'$, which is described by the equation $d'x+e'y+z=d'a'+e'b'+c'$ (recall that $(d',e',1)$ is a vector orthogonal to this plane and that the plane contains the center $(a',b',c')$ of $C$). One can check that the point $L$ is of the form $L=(p_1/q, p_2/q, p_3/q)$, where $p_1$, $p_2$, $p_3$ and $q$ are polynomials in $a,b,c,d,e,a',b',c',d',e'$, and it turns out that $q=(d-d')^2+(e-e')^2+(de'-d'e)^2$.

If the point $L$ has distance at least $r$ from the point $(a,b,c)$, then there are no points on the line $\l$ in the interior of the disk described by the circle $C$. But this means that there are not points on $C'$ in the interior of this disk (note that any such point would need to lie on the planes of both $C_1$ and $C_2$, and therefore on $\l$). Hence the circles $C$ and $C'$ cannot form a link if the distance of $L=(p_1/q, p_2/q, p_3/q)$ from $(a,b,c)$ is at least $r$. Letting $h=r^2-((p_1/q)-a)^2-((p_2/q)-b)^2-((p_3/q)-c)^2$, this means that $C$ and $C'$ cannot form a link if $h\leq 0$. Note that $h=p_4/q^2$ for some polynomial $p_4$ in $a,b,c,d,e,r,a',b',c',d',e'$, and in particular it can be checked whether $h\leq 0$ by looking at the sign of $p_4$. Let us now assume that $h>0$, which means that the point $L$ has distance less than $r$ from the point $(a,b,c)$.

Then the line $\l$ intersects the circle $C$ in two distinct points $X_1$ and $X_2$. Note that by Pythagoras' theorem the distance of $X_1$ and $X_2$ to $L$ is precisely $\sqrt{h}$. As $X_1$, $X_2$ and $L$ all lie on the line $\l$, whose direction is given by the vector $(e-e', -(d-d'), de'-d'e)$, we obtain that $X_1$ and $X_2$ are equal to
\begin{multline*}
L\pm \frac{\sqrt{h}}{\sqrt{(d-d')^2+(e-e')^2+(de'-d'e)^2}}\cdot \begin{pmatrix}
e-e'\\
-(d-d')\\
de'-d'e
\end{pmatrix}=\begin{pmatrix}
p_1/q\\
p_2/q\\
p_3/q
\end{pmatrix}\pm \frac{\sqrt{hq}}{q}\cdot \begin{pmatrix}
e-e'\\
-(d-d')\\
de'-d'e
\end{pmatrix}\\
=
\frac{1}{q}\cdot \begin{pmatrix}
p_1\pm \sqrt{hq}\cdot(e-e')\\
p_2\mp \sqrt{hq}\cdot(d-d')\\
p_3\pm \sqrt{hq}\cdot(de'-d'e)
\end{pmatrix}.
\end{multline*}

The points $X_1$ and $X_2$ are the only points of $C$ that lie on the plane of $C'$. Thus, the circles $C$ and $C'$ form a link if and only if one of the points $X_1$ and $X_2$ is inside the circle $C'$ and the other one is outside. This is the case if and only if one of the expressions $\Vert X_1-(a',b',c')\Vert^2-r'^2$ and $\Vert X_2-(a',b',c')\Vert^2-r'^2$ is negative and the other one is positive. Note that
\begin{multline*}
\Vert X_{1,2}-(a',b',c')\Vert^2-r'^2
=\frac{1}{q^2}\left\Vert \begin{pmatrix}
p_1\pm \sqrt{hq}\cdot(e-e')-qa'\\
p_2\mp \sqrt{hq}\cdot(d-d')-qb'\\
p_3\pm \sqrt{hq}\cdot(de'-d'e)-qc'
\end{pmatrix}\right\Vert^2-r'^2\\
=\frac{1}{q^2}\left((p_1-qa'\pm \sqrt{hq}\cdot(e-e'))^2+(p_2-qb'\mp \sqrt{hq}\cdot(d-d'))^2+(p_3-qc'\pm \sqrt{hq}\cdot(de'-d'e))^2-q^2r'^2\right)\\
=\frac{1}{q^2}\Big((p_1-qa')^2+hq(e-e')^2+(p_2-qb')^2+hq(d-d')^2+(p_3-qc')^2+hq(de'-d'e)^2-q^2r'^2\\
\pm 2\sqrt{hq}\cdot\big((p_1-qa')(e-e')-(p_2-qb')(d-d')+(p_3-qc')(de'-d'e)\big)\Big).
\end{multline*}
Thus, the circles $C$ and $C'$ form a link if and only if we have
\begin{multline*}
\left\vert 2\sqrt{hq}\cdot\big((p_1-qa')(e-e')-(p_2-qb')(d-d')+(p_3-qc')(de'-d'e)\big)\right\vert\\
>\left\vert(p_1-qa')^2+hq(e-e')^2+(p_2-qb')^2+hq(d-d')^2+(p_3-qc')^2+hq(de'-d'e)^2-q^2r'^2\right\vert,
\end{multline*}
which (recalling that $h=p_4/q^2$ and $q=(d-d')^2+(e-e')^2+(de'-d'e)^2>0$) is equivalent to
\begin{multline*}
\left\vert 2\sqrt{p_4q}\cdot\big((p_1-qa')(e-e')-(p_2-qb')(d-d')+(p_3-qc')(de'-d'e)\big)\right\vert\\
>\left\vert q(p_1-qa')^2+p_4(e-e')^2+q(p_2-qb')^2+p_4(d-d')^2+q(p_3-qc')^2+p_4(de'-d'e)^2-q^3r'^2\right\vert.
\end{multline*}
Thus, $C$ and $C'$ form a link if and only if
\begin{multline*}
4p_4q\cdot\big((p_1-qa')(e-e')-(p_2-qb')(d-d')+(p_3-qc')(de'-d'e)\big)^2\\
>\left(q(p_1-qa')^2+p_4(e-e')^2+q(p_2-qb')^2+p_4(d-d')^2+q(p_3-qc')^2+p_4(de'-d'e)^2-q^3r'^2\right)^2.
\end{multline*}
As $p_1,p_2,p_3,p_4,q$ are polynomials in $a,b,c,d,e,r,a',b',c',d',e'$, this establishes that we can check whether $C$ and $C'$ form a link using the signs of a finite list of polynomials in $a,b,c,d,e,r,a',b',c',d',e',r'$.

We remark that this finite list of polynomials consists of the polynomials $d-d'$, $e-e'$, $p_4$ and the polynomial obtained from subtracting the two sides of the last inequality. Recall that $p_4$ is zero if an only if the line $\l$ is tangent to $C$. Furthermore, observe that the last polynomial is zero if and only if $\Vert X_1-(a',b',c')\Vert^2-r'^2=0$ or $\Vert X_2-(a',b',c')\Vert^2-r'^2=0$, meaning that one of the points $X_1$ and $X_2$ lies on $C'$. Note that this happens if and only if the circles $C$ and $C'$ intersect each other. Hence all four polynomials in our list are non-zero if $d\neq d'$, $e\neq e'$, the line $\l$ is not tangent to $C$ and the circles $C$ and $C'$ are disjoint.

\subsection{Proofs of Facts \ref{fact-poly-vanishing}, \ref{fact-dim-union}, \ref{fact-two-coprime-poly} and \ref{fact-components-finite}}
\label{sect-a-facts}

Before proving Fact \ref{fact-poly-vanishing} in general, let us first consider the special case $\l=1$.

\begin{fact}\label{fact-poly-vanishing-one} Let $m\geq 1$ and let $Q\in \R[x_1,\dots,x_m]$ be a non-zero polynomial. Then for any non-empty open set $U\su \R^m$, we can find a point $x\in U$ such that $Q(x)\neq 0$.
\end{fact}
\begin{proof} Suppose for contradiction that $Q(x)=0$ for all $x\in U$. Then all higher order partial derivatives of $Q$ would also be zero on $U$. Let $x_1^{a_1}\dotsm x_m^{a_m}$ be a monomial of $Q$ of maximum degree and let its coefficient in $Q$ be $c\neq 0$. But note that then $(\partial_{x_1})^{a_1}\dots(\partial_{x_m})^{a_m} Q(x)=a_1!\dotsm a_m!\cdot c\neq 0$ for all $x\in \R^m$, which contradicts $(\partial_{x_1})^{a_1}\dots(\partial_{x_m})^{a_m} Q$ being zero on $U$.
\end{proof}

Now, Fact \ref{fact-poly-vanishing} follows easily from Fact \ref{fact-poly-vanishing-one}.

\begin{proof}[Proof of Fact \ref{fact-poly-vanishing}] We prove the desired statement by induction on $\l$. The case $\l=1$ is given in Fact \ref{fact-poly-vanishing-one}. Suppose now that $\l>1$ and that we are given non-zero polynomials $Q_1,\dots,Q_\l\in \R[x_1,\dots,x_m]$ and a non-empty open subset $U\su \R^m$. By Fact \ref{fact-poly-vanishing-one}, the polynomial $Q_\l$ cannot vanish on the entire set $U$. Thus, $U'= \lbrace x\in U\mid Q_\l(x)\neq 0\rbrace$ is a non-empty open subset of $\R^m$. Now, by the induction hypothesis, there exists a point $x\in U'\su U$ such that $Q_i(x)\neq 0$ for $i=1,\dots,\l-1$. Note that by the definition of $U'$ we also have $Q_\l(x)\neq 0$.
\end{proof}

Next, let us prove Fact \ref{fact-dim-union}.

\begin{proof}[Proof of Fact \ref{fact-dim-union}] If $\l=0$ or if all of $V_1,\dots,V_\l$ are the empty set, the statement is trivially true. Otherwise, we may omit any $V_i$ that are empty. So let us from now on assume that $\l\geq 1$ and that $V_1,\dots,V_\l$ are non-empty.

Now, for each $i=1,\dots,\l$, let $I_i\su \R[x_1,\dots,x_m]$ be an ideal such that $V_i=\mathcal{Z}(I_i)$. Furthermore, let $I=I_1\dotsm I_\l$ be the ideal generated by all elements of the form $Q_1\dotsm Q_\l$ with $Q_1\in I_1, \dots, Q_\l\in I_\l$. We claim that then we have $V=\mathcal{Z}(I)$. First, note that any product $Q_1\dotsm Q_\l$ with $Q_1\in I_1, \dots, Q_\l\in I_\l$ vanishes on each of the sets $V_i$ for $i=1,\dots,\l$, since $Q_i$ vanishes on $V_i$. Thus, $Q_1\dotsm Q_\l$ vanishes on the entire set $V=V_1\cup\dots\cup V_\l$. Hence every polynomial in the ideal $I$ vanishes on $V$ and we have $V\su \mathcal{Z}(I)$. For the reverse inclusion, fix any $x\in \mathcal{Z}(I)$. Suppose we had $x\not\in V$, then $x\not\in V_i=\mathcal{Z}(I_i)$ for all $i=1,\dots,\l$. Thus, for each $i=1,\dots, \l$, there exists a polynomial $Q_i\in I_i$ such that $Q_i(x)\neq 0$. But then $Q_1(x)\dotsm Q_\l(x)\neq 0$, which contradicts $Q_1\dotsm Q_\l\in I$ and $x\in \mathcal{Z}(I)$. Thus, we indeed have $V=\mathcal{Z}(I)$. Hence $V$ is a real algebraic set in $\R^m$.

Let us now prove that $\dim V=\max_i\, (\dim V_i)$. For each $i=1,\dots,\l$, we have $V_i\su V$ and therefore $\mathcal{I}(V)\su \mathcal{I}(V_i)$. Thus, every chain of prime ideals $\mathfrak{p}_0, \mathfrak{p}_1,\dots, \mathfrak{p}_d$ in $\R[x_1,\dots,x_m]$ with $\mathcal{I}(V_i)\su \mathfrak{p}_0\subsetneq \mathfrak{p}_1\subsetneq\dots\subsetneq \mathfrak{p}_d$ also satisfies $\mathcal{I}(V)\su \mathfrak{p}_0\subsetneq \mathfrak{p}_1\subsetneq\dots\subsetneq \mathfrak{p}_d$ and we obtain $\dim V\geq \dim V_i$ for $i=1,\dots,\l$.

For the opposite inequality, consider a chain of prime ideals $\mathfrak{p}_0, \mathfrak{p}_1,\dots, \mathfrak{p}_d$ in $\R[x_1,\dots,x_m]$ with $\mathcal{I}(V)\su \mathfrak{p}_0\subsetneq \mathfrak{p}_1\subsetneq\dots\subsetneq \mathfrak{p}_d$ and $d=\dim V$. We claim that we must have $\mathcal{I}(V_i)\su \mathfrak{p}_0$ for some $1\leq i\leq \l$. If this is indeed the case, then $\mathcal{I}(V_i)\su \mathfrak{p}_0\subsetneq \mathfrak{p}_1\subsetneq\dots\subsetneq \mathfrak{p}_d$ and therefore $\dim V_i\geq d=\dim V$, as desired.

So let us assume for contradiction that $\mathcal{I}(V_i)\not\su \mathfrak{p}_0$ for all $1\leq i\leq \l$. Then for each $i=1,\dots,\l$ we can choose a polynomial $Q_i\in \mathcal{I}(V_i)\sm \mathfrak{p}_0$. As $\mathfrak{p}_0$ is a prime ideal, the product $Q=Q_1\dotsm Q_\l$ satisfies $Q\not\in \mathfrak{p}_0$. But on the other hand, for each $i=1,\dots,\l$ the polynomial $Q_i$ vanishes on every point in the set $V_i$, and therefore the product $Q=Q_1\dotsm Q_\l$ must vanish on every point in $V=V_1\cup\dots\cup V_\l$. Thus, $Q\in \mathcal{I}(V)\su \mathfrak{p}_0$, which is the desired contradiction.
\end{proof}

Now, we prove Fact \ref{fact-two-coprime-poly}.

\begin{proof}[Proof of Fact \ref{fact-two-coprime-poly}]
Let $V=\lbrace x\in \R^m\mid P_1(x)=\dots=P_\l(x)=0\rbrace$. We claim that $V$ is the zero-set $\mathcal{Z}((P_1,\dots,P_\l))$ of the ideal $(P_1,\dots,P_\l)$ generated by $P_1,\dots,P_\l$. Since all the $P_i$ are elements of this ideal, they all vanish on the set $\mathcal{Z}((P_1,\dots,P_\l))$ and so $\mathcal{Z}((P_1,\dots,P_\l))\su V$. On the other hand, every polynomial $Q\in (P_1,\dots,P_\l)$ is of the form $Q=R_1\cdot P_1+\dots +R_\l\cdot P_\l$ for some polynomials $R_1, \dots, R_\l\in \R[x_1,\dots,x_m]$ and therefore satisfies $Q(x)=R_1(x)\cdot 0+\dots+ R_\l(x)\cdot 0=0$ for every $x\in V$. Thus, every element of the ideal $(P_1,\dots,P_\l)$ vanishes on the entire set $V$ and consequently $V\su \mathcal{Z}((P_1,\dots,P_\l))$. This establishes that $V=\mathcal{Z}((P_1,\dots,P_\l))$ is a real algebraic set in $\R^m$.

Now, suppose for contradiction that $\dim V\geq m-1$. Then we can find a chain of prime ideals $\mathfrak{p}_0, \mathfrak{p}_1,\dots, \mathfrak{p}_{m-1}$ in $\R[x_1,\dots,x_m]$ with $\mathcal{I}(V)\su \mathfrak{p}_0\subsetneq \mathfrak{p}_1\subsetneq\dots\subsetneq \mathfrak{p}_{m-1}$.

By the definition of $V$ we have $P_1,\dots,P_\l\in \mathcal{I}(V)$. Thus, $(P_1)\su \mathcal{I}(V)$. On the other hand, some $P_i$ with $2\leq i\leq \l$ is not divisible by $P_1$ and so we have $P_i\not\in (P_1)$. This establishes $(P_1)\subsetneq \mathcal{I}(V)$ and therefore $(P_1)\subsetneq \mathfrak{p}_0$. But now
$(0)\subsetneq (P_1) \subsetneq \mathfrak{p}_0\subsetneq \mathfrak{p}_1\subsetneq\dots\subsetneq \mathfrak{p}_{m-1}$
is a chain of $m+2$ prime ideals in $\R[x_1,\dots,x_m]$ (note that $(P_1)$ is a prime ideal since $P_1$ is irreducible). However, $\dim \R[x_1,\dots,x_m]=m$ (see, for example, \cite[Theorem 14.98]{goertz-wedhorn} or \cite[Theorem A, p. 221]{eisenbud}), and therefore any nested chain of prime ideals in $\R[x_1,\dots,x_m]$ has length at most $m+1$. This is a contradiction. Hence $\dim V\leq m-2$.
\end{proof}

Finally, let us prove Fact \ref{fact-components-finite}. We will deduce this fact from a theorem of Milnor \cite{milnor} bounding the sum of the Betti numbers of a real algebraic set (a similar theorem was independently proved by Thom \cite{thom}). The deduction uses an argument of Petersen \cite{mathoverflow}.

\begin{proof}[Proof of Fact \ref{fact-components-finite}]
Let the set $U\su \R^d$ be given as
\[U=\lbrace x\in \R^d \mid (\sgn Q_1(x),\dots,\sgn Q_\l(x))\in S\rbrace\]
for some finite list of polynomials $Q_1,\dots,Q_\l\in \R[x_1,\dots,x_d]$ and some subset $S\su \lbrace +,-,0\rbrace^\l$. We may assume without loss of generality that all the polynomials $Q_1,\dots,Q_\l$ are non-zero. Indeed, if some of these polynomials are zero, we only need to consider those $\l$-tuples in $S$ that have a zero in position $j$ for all the $j$ with $Q_j=0$. We can then ignore all the polynomials $Q_j$ with $Q_j=0$ and delete the zeros in the corresponding positions in all $\l$-tuples in $S$. This does not change the set $\lbrace x\in \R^d \mid (\sgn Q_1(x),\dots,\sgn Q_\l(x))\in S\rbrace$. So let us from now on assume that $Q_j\neq 0$ for all $j=1\dots,\l$.

Recall that $U$ is open and that we need to show that the open set 
\begin{multline}\label{eq-set-conn-comp1}
\lbrace x\in U \mid R_i(x)\neq 0\text{ for }i=1,\dots,k\rbrace\\
=\lbrace x\in \R^d \mid (\sgn Q_1(x),\dots,\sgn Q_\l(x))\in S\text{ and }R_i(x)\neq 0\text{ for }i=1,\dots,k\rbrace
\end{multline}
has only finitely many connected components.

We claim that each connected component of the set (\ref{eq-set-conn-comp1}) contains at least one connected component of the open set
\begin{equation}\label{eq-set-conn-comp2}
\lbrace x\in \R^d \mid Q_j(x)\neq 0\text{ for }j=1,\dots,\l\text{ and }R_i(x)\neq 0\text{ for }i=1,\dots,k\rbrace.
\end{equation}
Indeed, any connected component $C$ of the open set (\ref{eq-set-conn-comp1}) is itself a (non-empty) open set, and therefore by Fact \ref{fact-poly-vanishing} contains a point $x$ with $Q_j(x)\neq 0$ for $j=1,\dots,\l$. As $x$ lies in the set (\ref{eq-set-conn-comp1}), we also have $R_i(x)\neq 0$ for $i=1,\dots,k$ and $(\sgn Q_1(x),\dots,\sgn Q_\l(x))\in S$. Now, the point $x$ is contained in some connected component $C'$ of the set (\ref{eq-set-conn-comp2}). Note that for every point $x'\in C'$ we have $\sgn Q_j(x')=\sgn Q_j(x)$ for $j=1,\dots,\l$. This implies that $(\sgn Q_1(x'),\dots,\sgn Q_\l(x'))\in S$ for all $x'\in C'$. Therefore we can conclude that each $x'\in C'$ is contained in the set (\ref{eq-set-conn-comp1}), and so $C'$ is a connected subset of the set (\ref{eq-set-conn-comp1}). Thus, $C'$ must be a subset of one of the connected components of the set (\ref{eq-set-conn-comp1}). As $x\in C'\cap C$, this connected component must be $C$, so $C'\su C$. This shows that every connected component of the set (\ref{eq-set-conn-comp1}) contains at least one connected component of the set (\ref{eq-set-conn-comp2}).

Thus, it suffices to prove that the set (\ref{eq-set-conn-comp2}) has only finitely many connected components. Note that the set (\ref{eq-set-conn-comp2}) can also be described as
\[ \lbrace x\in \R^d \mid Q_1(x)\dotsm Q_\l(x)\cdot R_1(x)\dotsm R_k(x)\neq 0\rbrace\]
and this set is homeomorphic to the set
\begin{equation}\label{eq-set-conn-comp3}
\lbrace (x,y)\in \R^d \times \R\mid Q_1(x)\dotsm Q_\l(x)\cdot R_1(x)\dotsm R_k(x)\cdot y=1\rbrace
\end{equation}
(this is an idea due to Peterson \cite{mathoverflow}). But the number of connected components of the set (\ref{eq-set-conn-comp3}) equals the $0$-th Betti number of this set. By a Theorem of Milnor \cite{milnor} the sum of all Betti numbers of the set (\ref{eq-set-conn-comp3}) is bounded by $s(2s-1)^d$ where $s=\deg Q_1+\dots+\deg Q_\l+\deg R_1+\dots+\deg R_k+1$ (see also \cite[Theorem 11.5.3]{real-ag-book}, and note that a similar theorem was proved independently by Thom \cite{thom}). Since all Betti numbers are non-negative integers, this implies that the $0$-th Betti number of the set (\ref{eq-set-conn-comp3}) is at most $s(2s-1)^d$ and is therefore in particular finite. Thus, the number of connected components of the set (\ref{eq-set-conn-comp3}) is finite. Since this set is homeomorphic to the set (\ref{eq-set-conn-comp2}), we have proved Fact \ref{fact-components-finite}.
\end{proof}

\subsection{Proof of Fact \ref{fact-cutting-dimension}}
\label{sect-a-cutting-dim}

For proving Fact \ref{fact-cutting-dimension}, we will use the following easy observation.

\begin{fact}\label{fact-ideal-equiv} Let $a=(a_1,\dots,a_m)\in \R^m$ and let $I\su \R[x_1,\dots,x_m]$ be any ideal. Then we have $a\in \mathcal{Z}(I)$ if and only if $I\su (x_1-a_1,\dots, x_m-a_m)$
\end{fact}
\begin{proof}
Note that any polynomial $Q\in (x_1-a_1,\dots, x_m-a_m)$ vanishes on the point $a$. Furthermore, the set of polynomials $Q\in \R[x_1,\dots,x_m]$ satisfying $Q(a)=0$ is a proper ideal in $\R[x_1,\dots,x_m]$. Since the ideal $(x_1-a_1,\dots, x_m-a_m)$ is maximal, this implies that the set of $Q\in \R[x_1,\dots,x_m]$ with $Q(a)=0$ equals the ideal $(x_1-a_1,\dots, x_m-a_m)$. Hence, we have $I\su (x_1-a_1,\dots, x_m-a_m)$ if and only if every polynomial $Q\in I$ satisfies $Q(a)=0$, and this is by definition equivalent to $a\in \mathcal{Z}(I)$.
\end{proof}

Now, we will deduce Fact \ref{fact-cutting-dimension} from a more general scheme-theoretic statement, namely \cite[Theorem 11.4.1]{ravi}.

\begin{proof}[Proof of Fact \ref{fact-cutting-dimension}]
Let $I=\mathcal{I}(V)\su \R[x_1,\dots,x_n,y_1,\dots,y_n]$. Then we have $\dim \R[x_1,\dots,x_n,y_1,\dots,y_n]/I=\dim V\leq 2n-2$. Let $\mathfrak{p}_1,\dots,\mathfrak{p}_\l$ be the minimal prime ideals in $\R[x_1,\dots,x_n,y_1,\dots,y_n]$ that contain $I$ (there are only finitely many by \cite[Exercise 1.2]{eisenbud}, see also \cite[Proposition 3.6.15]{ravi} as these prime ideals correspond to the irreducible components of the scheme $\operatorname{Spec} \R[x_1,\dots,x_n,y_1,\dots,y_n]/I$). For each $i=1,\dots,\l$ we have
\begin{equation}\label{eq-dim-pi}
\dim \R[x_1,\dots,x_n,y_1,\dots,y_n]/\mathfrak{p}_i\leq \dim \R[x_1,\dots,x_n,y_1,\dots,y_n]/I\leq 2n-2.
\end{equation}
For $i=1,\dots,\l$, let $V_i=\mathcal{Z}(\mathfrak{p}_i)$.

\begin{claim}\label{claim-V-union}$V=V_1\cup\dots\cup V_\l$.
\end{claim}
\begin{proof}As $V$ is a real algebraic set, we have $V=\mathcal{Z}(J)$ for some ideal $J\su I=\mathcal{I}(V)$. Then for each $i=1,\dots,\l$ we have $J\su I\su \mathfrak{p}_i$, and therefore $V_i=\mathcal{Z}(\mathfrak{p}_i)\su \mathcal{Z}(J)=V$. Thus, $V_1\cup\dots\cup V_\l\su V$.

For the opposite inclusion, consider any point $(a,b)=(a_1,\dots, a_n,b_1,\dots,b_n)\in V\su \R^n\times \R^n$. Since all polynomials $Q\in I=\mathcal{I}(V)$ vanish on $V$, we have $Q(a,b)=0$ for all $Q\in I$, and therefore $(a,b)\in \mathcal{Z}(I)$. By Fact \ref{fact-ideal-equiv}, this implies $I\su (x_1-a_1,\dots, x_n-a_n, y_1-b_1,\dots,y_n-b_n)$. On the other hand, note that the ideal $(x_1-a_1,\dots, x_n-a_n, y_1-b_1,\dots,y_n-b_n)$ is maximal and therefore prime. Thus, by the choice of $\mathfrak{p}_1,\dots,\mathfrak{p}_\l$ we must have $\mathfrak{p}_i\su (x_1-a_1,\dots, x_n-a_n, y_1-b_1,\dots,y_n-b_n)$ for some $1\leq i\leq \l$. So, again by Fact \ref{fact-ideal-equiv}, we obtain $(a,b)\in \mathcal{Z}(\mathfrak{p}_i)=V_i$. This proves that $V\su V_1\cup\dots\cup V_\l$.
\end{proof}

In light of Claim \ref{claim-V-union}, it is sufficient to prove the following claim.

\begin{claim}\label{claim-i-fixed}For each $i=1,\dots,\l$, there is a dense open set $U_i\su \R^n$ such that each point $a\in U_i$ satisfies the following condition: The set $\lbrace b\in \R^n\mid (a,b)\in V_i\rbrace$ is a real algebraic set of dimension at most $n-2$.
\end{claim}

Let us postpone the proof of Claim \ref{claim-i-fixed} for a moment, and first finish the rest of the proof of Fact \ref{fact-cutting-dimension}. From Claim \ref{claim-i-fixed} we obtain dense open sets $U_1,\dots, U_\l\su \R^n$. Then $U= U_1\cap\dots\cap U_\l$ is also a dense open subset of $\R^n$ and for each $a\in U$ each of the sets $\lbrace b\in \R^n\mid (a,b)\in V_i\rbrace$ for $i=1,\dots,\l$ is a real algebraic set of dimension at most $n-2$. On the other hand, by Claim \ref{claim-V-union} we have
\[\lbrace b\in \R^n\mid (a,b)\in V\rbrace=\bigcup_{i=1}^{\l}\lbrace b\in \R^n\mid (a,b)\in V_i\rbrace.\]
Thus, for each $a\in U$, by Fact \ref{fact-dim-union}, $\lbrace b\in \R^n\mid (a,b)\in V\rbrace$ is also real algebraic set of dimension at most $n-2$. This finishes the proof of Fact \ref{fact-cutting-dimension} up to proving Claim \ref{claim-i-fixed}.
\end{proof}

\begin{proof}[Proof of Claim \ref{claim-i-fixed}]
Let us fix some $i\in \lbrace 1,\dots,\l\rbrace$. The scheme $\operatorname{Spec} \R[x_1,\dots,x_n,y_1,\dots,y_n]/\mathfrak{p}_i$ is an irreducible variety over $\R$, and by (\ref{eq-dim-pi}) its dimension is at most $2n-2$. Furthermore, $\operatorname{Spec} \R[x_1,\dots,x_n]$ is an irreducible variety over $\R$ of dimension $n$ (see \cite[Theorem 11.2.1]{ravi}). So by \cite[Theorem 11.4.1]{ravi} applied to the map $\pi: \operatorname{Spec} \R[x_1,\dots,x_n,y_1,\dots,y_n]/\mathfrak{p}_i\to \operatorname{Spec} \R[x_1,\dots,x_n]$ there exists a non-empty Zariski-open subset $U_i^{*}\su \operatorname{Spec} \R[x_1,\dots,x_n]$ such that for every $q\in U_i^*$ the fiber of $\pi$ over $q$ has dimension at most $(2n-2)-n=n-2$.

The Zariski-open subset $U_i^{*}\su \operatorname{Spec} \R[x_1,\dots,x_n]$ is given as the complement of the vanishing set $V(J)$ of some ideal $J\su \R[x_1,\dots,x_n]$. As $U_i^{*}$ is non-empty, we have $J\neq (0)$. Now define $U_i=\R^n\sm \mathcal{Z}(J)$. This is clearly an open subset of $\R^n$ and by Fact \ref{fact-poly-vanishing}, it is dense in $\R^n$: Indeed, choose any non-zero polynomial $Q\in J$. Then for any non-empty open subset $U'\su \R^n$, by Fact \ref{fact-poly-vanishing} there is a point $x\in U'$ such that $Q(x)\neq 0$. Hence $x\not\in  \mathcal{Z}(J)$ and therefore $x\in U'\cap U_i$. This establishes that the intersection $U'\cap U_i$ is non-empty for every non-empty open subset $U'\su \R^n$. Thus, $U_i=\R^n\sm \mathcal{Z}(J)$ is indeed a dense open subset of $\R^n$.

Let us fix any point $a=(a_1,\dots, a_n)\in U_i\su \R^n$. We need to show that the set $\lbrace b\in \R^n\mid (a,b)\in V_i\rbrace$ is a real algebraic set of dimension at most $n-2$. First, note that $(a_1,\dots, a_n)\in U_i$ means that $(a_1,\dots, a_n)\not\in \mathcal{Z}(J)$. By Fact \ref{fact-ideal-equiv} this means that $J\not\su (x_1-a_1,\dots, x_n-a_n)$. Therefore the point $q\in \operatorname{Spec} \R[x_1,\dots,x_n]$ corresponding to the prime ideal $(x_1-a_1,\dots, x_n-a_n)$ in $\R[x_1,\dots,x_n]$, does not lie in the vanishing set $V(J)$ of the ideal $J$. Thus, $q\in U_i^{*}$ and the fiber of $\pi$ over $q$ has dimension at most $n-2$. This fiber is given by
\begin{multline*}\operatorname{Spec}\left( \R[x_1,\dots,x_n,y_1,\dots,y_n]/\mathfrak{p}_i\otimes_{\R[x_1,\dots,x_n]}\R[x_1,\dots,x_n]/(x_1-a_1,\dots, x_n-a_n)\right)\\
=\operatorname{Spec} \R[x_1,\dots,x_n,y_1,\dots,y_n]/(\mathfrak{p}_i+(x_1-a_1,\dots, x_n-a_n))
\end{multline*}
(note that on the right-hand side, $(x_1-a_1,\dots, x_n-a_n)$ denotes the ideal generated by $x_1-a_1, \dots, x_n-a_n$ in the ring $\R[x_1,\dots,x_n,y_1,\dots,y_n]$). Thus, the ring
\begin{equation}\label{eq-ring-pi}
\R[x_1,\dots,x_n,y_1,\dots,y_n]/(\mathfrak{p}_i+(x_1-a_1,\dots, x_n-a_n))
\end{equation}
has dimension at most $n-2$.

Now, consider the surjective ring homomorphism $\theta: \R[x_1,\dots,x_n,y_1,\dots,y_n]\to \R[y_1,\dots,y_n]$ sending $x_j$ to $a_j$ for $j=1,\dots,n$. Let the ideal $\mathfrak{t}_i\su \R[y_1,\dots,y_n]$ be the image of the ideal $\mathfrak{p}_i\su \R[x_1,\dots,x_n,y_1,\dots,y_n]$ under $\theta$. In other words, $\mathfrak{t}_i\su \R[y_1,\dots,y_n]$ is the ideal obtained from $\mathfrak{p}_i$ when replacing every variable $x_j$ by $a_j$. Since the kernel of $\theta$ is the ideal $(x_1-a_1,\dots, x_n-a_n)$ in $\R[x_1,\dots,x_n,y_1,\dots,y_n]$, the preimage of $\mathfrak{t}_i$ under $\theta$ is $\mathfrak{p}_i+(x_1-a_1,\dots, x_n-a_n)$. Thus, $\theta$ induces an isomorphism of the ring in (\ref{eq-ring-pi}) and the ring $\R[y_1,\dots,y_n]/\mathfrak{t}_i$. In particular, we obtain $\dim \R[y_1,\dots,y_n]/\mathfrak{t}_i\leq n-2$.

Finally, let us turn to the set $\lbrace b\in \R^n\mid (a,b)\in V_i\rbrace$. As $V_i=\mathcal{Z}(\mathfrak{p}_i)$, this is the set of points $b=(b_1,\dots,b_n)\in \R^n$ such that $Q(a_1,\dots,a_n,b_1,\dots,b_n)=0$ for all $Q\in \mathfrak{p}_i$. But this is the same as the set of points $b=(b_1,\dots,b_n)\in \R^n$ such that $T(b_1,\dots,b_n)=0$ for all $T\in \mathfrak{t}_i$. Thus,
\[\lbrace b\in \R^n\mid (a,b)\in V_i\rbrace=\mathcal{Z}(\mathfrak{t}_i)\]
is a real algebraic set and its dimension is (using that $\mathfrak{t}_i\su \mathcal{I}(\mathcal{Z}(\mathfrak{t}_i))$)
\[\dim \R[y_1,\dots,y_n]/\mathcal{I}(\mathcal{Z}(\mathfrak{t}_i))\leq\dim \R[y_1,\dots,y_n]/\mathfrak{t}_i\leq n-2.\]
This finishes the proof of Claim \ref{claim-i-fixed}.
\end{proof}

\subsection{Proof of Fact \ref{fact-complement-variety}}
\label{sect-a-complement}

First, we will prove Fact \ref{fact-transversality} below using standard transversality arguments. Given smooth manifolds $X$ and $Y$, a smooth map $F: Y\to X$ is called \emph{transverse} to an embedded submanifold $M\su X$ if for every point $p\in F^{-1}(M)$ the vector spaces $T_{F(p)} M$ and $dF_p(T_p Y)$ together span the entire tangent space $T_{F(p)} X$ (see \cite[p. 143]{lee}). Note that in case $F$ is a smooth submersion (which means that the linear map $dF_p: T_p Y\to T_{F(p)} X$ is surjective for each $p\in Y$), the map $F$ is automatically transverse to every embedded submanifold $M\su X$. Also note that if $\dim M+\dim Y<\dim X$, then for every point $p\in Y$ we have
\[\dim T_{F(p)} M+\dim dF_p(T_p Y)\leq \dim T_{F(p)} M+\dim T_p Y= \dim M+\dim Y<\dim X=\dim T_{F(p)} X.\]
Thus, in the case $\dim M+\dim Y<\dim X$, the map $F$ is transverse to $M$ if and only if the preimage $F^{-1}(M)$ is empty (which means that $F(p)\not\in M$ for all $p\in Y$).

\begin{fact}\label{fact-transversality}
Let $X\su \R^m$ be a convex open subset, and let $M\su \R^m$ be an embedded smooth manifold of dimension at most $m-2$. Suppose that the set $X\sm M$ is open. Then the set $X\sm M$ is connected.
\end{fact}
\begin{proof}
We will prove that the set $X\sm M$ is path connected, which implies that it is connected. So let us fix two distinct points $a,b\in X\sm M$. We need to show that there is a (continuous) path in $X\sm M$ connecting $a$ and $b$. Without loss of generality we may assume that $a=(0,0,\dots,0)\in \R^m$ and $b=(1,0,\dots,0)\in \R^m$ (otherwise we can apply an invertible affine linear transformation of $\R^m$ mapping $a$ and $b$ to these points).

As $X\sm M$ is open, there exists some $\eps>0$ such that the open balls of radius $\eps$ around $a$ and $b$ are both entirely contained in $X\sm M$. Let $B_\eps\su \R^{m-1}$ denote the open ball of radius $\eps$ around the origin in $\R^{m-1}$. Then, for every $s\in B_\eps$, the point $(0,s)\in \R\times \R^{m-1}=\R^m$ is contained in $X\sm M$ and furthermore the entire segment connecting $a=(0,0,\dots,0)$ and $(0,s)$ is also contained in $X\sm M$. Thus, there exists a path entirely contained in $X\sm M$ that connects $a$ and $(0,s)$. Similarly, there exists a path entirely contained in $X\sm M$ that connects $b=(1,0,\dots,0)$ and $(1,s)$.

For each $s\in B_\eps$, let $F_s:(0,1)\to \R^m$ be the map given by $F_s(t)=(t,s)\in \R\times \R^{m-1}=\R^m$. Our goal is to prove that for some $s\in B_\eps$ we have $F_s(t)=(t,s)\in X\sm M$ for all $t\in (0,1)$. This would yield a path entirely contained in $X\sm M$ that connects the points $(0,s)$ and $(1,s)$ (recall from above that these points are contained in $X\sm M$). As we already saw that there are paths in $X\sm M$ connecting $a$ and $(0,s)$ as well as connecting $b$ and $(1,s)$, this gives a path in $X\sm M$ connecting $a$ and $b$, as desired. Thus, it indeed suffices to prove that for some $s\in B_\eps$ we have $F_s(t)=(t,s)\in X\sm M$ for all $t\in (0,1)$.

For all $s\in B_\eps$, we have $(0,s)\in X$ and $(1,s)\in X$ and therefore by convexity of $X$ also $F_s(t)=(t,s)\in X$ for all $t\in (0,1)$. So we need to show that for some $s\in B_\eps$ we have $F_s(t)\not\in M$ for all $t\in (0,1)$.

Note that the map $F:(0,1)\times B_\eps\to \R^m$ given by $F(t,s)=F_s(t)=(t,s)$ is an smooth submersion and therefore transverse to $M\su \R^m$. Thus, by \cite[Theorem 6.35]{lee} there exists $s\in B_\eps$ such that the map $F_s:(0,1)\to \R^m$ is transverse to $M$. However, as $\dim M+\dim\, (0,1)\leq m-2+1<m$, this means for this $s\in B_\eps$ we have $F_s(t)\not\in M$ for all $t\in (0,1)$. This finishes the proof of Fact \ref{fact-transversality}.
\end{proof}

For a real algebraic set $V\su \R^m$ and $0\leq d\leq m$, a point $p\in V$ is  \emph{a non-singular point of $V$ in dimension $d$} if there exist $m-d$ polynomials $Q_1,\dots,Q_{m-d}\in \mathcal{I}(V)$ and an open neighborhood $U\su \R^m$ of $p$ such that $V\cap U= \mathcal{Z}((Q_1,\dots,Q_{m-d}))\cap U$ and such that the Jacobian matrix $\left(\partial_{x_j} Q_i(p)\right)_{i,j}$ has rank $m-d$ (see \cite[Proposition 3.3.10]{real-ag-book}). Note that if $p\in V$ is a non-singular point of $V$ in dimension $d$, then there exists an open subset $U'\su \R^m$ such that every point $p'\in V\cap U$ is a non-singular point of $V$ in dimension $d$. Furthermore, it follows from the implicit function theorem that the set of non-singular points of $V$ in dimension $d$ forms an embedded smooth manifold in $\R^m$ of dimension $d$ (see also \cite[Proposition 3.3.11]{real-ag-book}).

Now we are finally ready for the proof of Fact \ref{fact-complement-variety}.

\begin{proof}[Proof of Fact \ref{fact-complement-variety}] First, note that the set $U\sm V$ is clearly open, since $U$ is open and $V$ is closed (since $V$ is a real algebraic set). Furthermore, note that the statement is clearly true if $V=\emptyset$, so we may assume that $V$ is non-empty.

Let us suppose for contradiction that there exists a non-empty real algebraic set $V\su \R^m$ of dimension $d\leq m-2$ and a connected open set $U\su \R^m$ such that $U\sm V$ is not connected. Then let us choose such sets $U$ and $V$ with minimum dimension $d=\dim V$. This way, we may assume that $U\sm V'$ is connected for all real algebraic sets $V'$ of dimension $\dim V'<d$.

Let $M$ be the set of all non-singular points of $V$ in dimension $d$, and let $V'=V\sm M$. By \cite[Proposition 3.3.14]{real-ag-book} the set $V'$ is a real algebraic set in $\R^m$ of dimension $\dim V'<\dim V=d$. In particular, by our choice of $U$ and $V$, we obtain that $U\sm V'$ is connected. Let us define $U'=U\sm V'$, then $U'$ is connected and open (as $V'$ is a real algebraic set and therefore closed).

The set $M$ of all non-singular points of $V$ in dimension $d$ is an embedded smooth manifold in $\R^m$ of dimension $d\leq m-2$. Furthermore, note that for every open ball $B\su U'=U\sm V'$ we have $B\sm M=B\sm (M\cup V')=B\sm V$. Since $V$ is closed (it is a real algebraic set), we can conclude that $B\sm M=B\sm V$ is open for every open ball $B\su U'$. Thus, by Fact \ref{fact-transversality}, for every open ball $B\su U'$ the set $B\sm M$ is connected.

Note that $U'\sm M=(U\sm V')\sm M=U\sm (V'\cup M)=U\sm V$ and recall that we assumed that this open set is not connected. Hence there exist disjoint non-empty open sets $U_1,U_2\su \R^m$ such $U'\sm M=U_1\cup U_2$.

Now let us define open subsets $T_1, T_2\su U'\su \R^m$ as follows: Let $T_1$ be the set of all those points $x\in \R^m$ that are contained in some open ball $B\su U'\cap (U_1\cup M)$. In other words, $T_1$ is the interior of the set $U'\cap (U_1\cup M)$. Similarly, let $T_2$ be the set of all those points $x\in \R^m$ that are contained in some open ball $B\su U'\cap (U_2\cup M)$.

It is clear from their definitions that the sets $T_1$ and $T_2$ are open subsets of $\R^m$ and that $T_1, T_2\su U'$. We claim that $T_1\cup T_2=U'$. Indeed, fix any point $x\in U'$ and consider any open ball $B\su U'$ containing $x$ (recall that $U'$ is an open set). We saw above that the set $B\sm M$ is connected. However, observe that
\[B\sm M=B\cap (U'\sm M)=B\cap (U_1\cup U_2)=(B\cap U_1)\cup (B\cap U_2).\]
As $B\cap U_1$ and $B\cap U_2$ are disjoint open sets (since $U_1$ and $U_2$ are disjoint open sets), this implies that $B\cap U_1=\emptyset$ or $B\cap U_2=\emptyset$. Let us assume without loss of generality that $B\cap U_1=\emptyset$, then $B\sm M\su B\cap U_2\su U_2$ and hence $B\su U_2\cup M$. As we also have $B\su U'$, this implies $B\su U'\cap (U_2\cup M)$ and therefore $x\in T_2$. This shows that $U'\su T_1\cup T_2$ and therefore $T_1\cup T_2=U'$.

Next, we claim that the sets $T_1$ and $T_2$ are disjoint. Suppose there exists a point $x\in T_1\cap T_2$. Then there are open balls $B_1\su U'\cap (U_1\cup M)$ and $B_2\su U'\cap (U_2\cup M)$ with $x\in B_1\cap B_2$. Now,
$B_1\su U_1\cup M$ and $B_2\su U_2\cup M$, and since $U_1\cap U_2=\emptyset$, this implies $B_1\cap B_2\su (U_1\cap U_2)\cup M=M\su V$. Thus, every polynomial $Q\in \mathcal{I}(V)$ vanishes on the entire set $B_1\cap B_2$. On the other hand, $B_1\cap B_2$ is a non-empty open set (since $x\in B_1\cap B_2$), and so by Fact \ref{fact-poly-vanishing} the only polynomial vanishing on all of $B_1\cap B_2$ is the zero-polynomial. Thus, we can conclude that $\mathcal{I}(V)=(0)$. But then $\dim V=\dim \R[x_1,\dots,x_m]=m$ (see \cite[Theorem 14.98]{goertz-wedhorn} or \cite[Theorem A, p. 221]{eisenbud}), which contradicts our assumption that $\dim V\leq m-2$. Hence there cannot exist a point $x\in T_1\cap T_2$ and consequently the sets $T_1$ and $T_2$ are disjoint.

Finally, we claim that the sets $T_1$ and $T_2$ are non-empty. Indeed, $U_1$ is a non-empty open set and therefore there exists an open ball $B\su U_1$. Note that we have $B\su U_1\su U'\sm M\su U'$ and $B\su U_1\cup M$ and therefore $B\su U'\cap (U_1\cup M)$. Thus, every point $x\in B$ is contained in $T_1$. This establishes that $T_1$ is non-empty, and it can be shown in the same way that $T_2$ is non-empty.

All in all, we have proved that $T_1$ and $T_2$ are disjoint non-empty open subsets of $\R^m$ with $T_1\cup T_2= U'$. But this contradicts $U'$ being connected. This contradiction finishes the proof of Fact \ref{fact-complement-variety}.
\end{proof}

\subsection{Proof of Lemma \ref{lemma-matrix}}
\label{sect-a-lemma-matrix}

Since the $(\l\times d)$-matrix $A(x_0)$ has rank $\l$, it has $\l$ linearly independent columns. Upon reordering the columns of $A$, we may assume without loss of generality that these are the first $\l$ columns. Then the $(\l\times \l)$-matrix formed by the first $\l$ columns of $A(x_0)$ has a non-zero determinant. On the other hand, the determinant of the $(\l\times \l)$-matrix formed by the first $\l$ columns of $A(x)$ is a smooth function of $x\in U$ (as all the coefficients are smooth functions). Thus, there exists an open neighborhood $U'\in U$ of $x_0$ such that for all $x\in U'$ the $(\l\times \l)$-matrix formed by the first $\l$ columns of $A(x)$ is non-singular. In particular, for all $x\in U'$ the matrix $A(x)$ has rank $\l$.

Furthermore, we can construct the desired vector field $w:U'\to \R^d$ as follows. For all $i\in \lbrace \l+1,\dots,d\rbrace$ and all $x\in U'$, define the $i$-th coordinate $w_i(x)$ of $w(x)$ to be $1$. If we denote the coordinates of $A(x)$ by $a_{i,j}(x)$ for $1\leq i\leq \l$ and $1\leq j\leq d$, then the condition $A(x)w(x)=0$ reads
\[
\begin{pmatrix}
a_{1,1}(x)&\dots &a_{1,\l}(x)& a_{1,\l+1}(x)&\dots &a_{1,d}(x)\\
\vdots& &\vdots& \vdots& &\vdots\\
a_{\l,1}(x)&\dots &a_{\l,\l}(x)& a_{\l,\l+1}(x)&\dots &a_{\l,d}(x)
\end{pmatrix}
\begin{pmatrix}
w_1(x)\\
\vdots\\
w_\l(x)\\
1\\
\vdots\\
1
\end{pmatrix}
=
\begin{pmatrix}
0\\
\vdots\\
0
\end{pmatrix},
\]
which is equivalent to
\[
\begin{pmatrix}
a_{1,1}(x)&\dots &a_{1,\l}(x)\\
\vdots& &\vdots\\
a_{\l,1}(x)&\dots &a_{\l,\l}(x)
\end{pmatrix}
\begin{pmatrix}
w_1(x)\\
\vdots\\
w_\l(x)
\end{pmatrix}
=
-\begin{pmatrix}
a_{1,\l+1}(x)+\dots +a_{1,d}(x)\\
\vdots\\
a_{\l,\l+1}(x)+\dots +a_{\l,d}(x)
\end{pmatrix}.
\]
Recall that for all $x\in U'$, the matrix on the left-hand side is non-singular. Thus, we can define the remaining coordinates $w_1(x),\dots,w_\l(x)$ by the equation
\[\begin{pmatrix}
w_1(x)\\
\vdots\\
w_\l(x)
\end{pmatrix}
=
-\begin{pmatrix}
a_{1,1}(x)&\dots &a_{1,\l}(x)\\
\vdots& &\vdots\\
a_{\l,1}(x)&\dots &a_{\l,\l}(x)
\end{pmatrix}^{-1}
\begin{pmatrix}
a_{1,\l+1}(x)+\dots +a_{1,d}(x)\\
\vdots\\
a_{\l,\l+1}(x)+\dots +a_{\l,d}(x)
\end{pmatrix}.
\]
The inverse matrix on the right-hand side can be computed from the determinant and the adjugate matrix. From this description it is clear that all the coefficients of this inverse matrix are smooth functions of $x\in U'$. Therefore we can conclude that $w_1(x),\dots,w_\l(x)$ are also smooth functions of $x\in U'$. All in all we defined a smooth vector field $w:U'\to \R^d$ with $A(x)w(x)=0$ for each $x\in U'$.

It remains to check that $w(x)\neq 0$ for all $x\in U'$. However, this is clear since $\l<d$ and we defined $w_i(x)=1$ for all $i\in \lbrace \l+1,\dots,d\rbrace$ and all $x\in U'$.

\end{document}